\documentclass[12pt]{article}

\usepackage{amsmath,amssymb,amsthm,graphics,amsfonts,graphicx,float}
\usepackage{calc, eurosym}
\usepackage{lineno}
\usepackage{tikz}
\usepackage{ifthen}
\usepackage{ifpdf}
\usepackage[breaklinks, colorlinks=true, citecolor=blue]{hyperref}
\usepackage{color}
\usepackage{enumerate}
\usepackage{caption}
\usepackage{subcaption}
\usepackage{pstricks}
\usetikzlibrary{arrows, decorations.pathmorphing, backgrounds, positioning, fit, petri, shapes}
\usepackage{multicol}
\usepackage{breakurl}
\hypersetup{breaklinks=true}
\usepackage{epstopdf}

\usepackage{bm}

\newtheorem{theorem}{Theorem}
\newtheorem{corollary}{Corollary}
\newtheorem{remark}{Remark}
\usepackage{geometry}

\def\today{\number\day \space\ifcase\month\or
  January\or February\or March\or April\or May\or June\or
  July\or August\or September\or October\or November\or December\fi
  \space\number\year}

\newboolean{shownotes}
\setboolean{shownotes}{true}

\ifthenelse{\boolean{shownotes}}{
	\newcommand{\notes}[1]{[\textcolor{red}{\emph{#1}}]}
}{
\newcommand{\notes}[1]{}
}

\newcommand*\patchAmsMathEnvironmentForLineno[1]{%
	\expandafter\let\csname old#1\expandafter\endcsname\csname #1\endcsname
	\expandafter\let\csname oldend#1\expandafter\endcsname\csname end#1\endcsname
	\renewenvironment{#1}%
	{\linenomath\csname old#1\endcsname}%
	{\csname oldend#1\endcsname\endlinenomath}}%
\newcommand*\patchBothAmsMathEnvironmentsForLineno[1]{%
	\patchAmsMathEnvironmentForLineno{#1}%
	\patchAmsMathEnvironmentForLineno{#1*}}%
\AtBeginDocument{%
	\patchBothAmsMathEnvironmentsForLineno{equation}%
	\patchBothAmsMathEnvironmentsForLineno{align}%
	\patchBothAmsMathEnvironmentsForLineno{flalign}%
	\patchBothAmsMathEnvironmentsForLineno{alignat}%
	\patchBothAmsMathEnvironmentsForLineno{gather}%
	\patchBothAmsMathEnvironmentsForLineno{multline}%
}

\begin{document}
\begin{center}  
{\Large \it \bf Bifurcation Thresholds and Optimal Control in Transmission Dynamics of Arboviral Diseases

}
\end{center}

\smallskip
\begin{center}
{\footnotesize \textsc{Hamadjam Abboubakar}\footnote{Corresponding author. Present Address:  UIT--Department of Computer Science, P. O. Box 455,  Ngaoundere, Cameroon, email: abboubakarhamadjam@yahoo.fr or h.abboubakar@gmail.com, Tel. (+237) 694 52 31 11}$^{,\star,\dagger}$, 
\textsc{Jean C. Kamgang}\footnote{Co-author address: ENSAI, Department of Mathematics and Computer Science, P.O. Box 455,  Ngaoundere, Cameroon, 
	email: jckamgang@gmail.com, Tel. (+237) 697 961 489
	}$^{,\dagger}$, 
\textsc{Daniel Tieudjo}\footnote{Co-author address: ENSAI, Department of Mathematics and Computer Science, P.O. Box 455,  Ngaoundere, Cameroon, 
	email: tieudjo@yahoo.com, Tel. (+237) 677 562 433
	}$^{,\dagger}$
}
\end{center}
\begin{center} 
{\small \sl $^{\star}$ The University of Ngaoundere, UIT, Laboratoire d'Analyse, Simulation et Essai, P. O. Box 455 Ngaoundere, Cameroon.}\\
{\small \sl $^{\dagger}$ The University of Ngaoundere, ENSAI, Laboratoire de Math\'ematiques Exp\'erimentales, P. O. Box 455 Ngaoundere, Cameroon.
}
\end{center}

\bigskip
{\small \centerline{\bf Abstract} 
In this paper, we derive and analyse a model for the control of arboviral diseases which takes into account an imperfect vaccine combined with some other mechanisms of control already studied in the literature. We begin by analyse the basic model without controls. We prove the existence of two disease-free equilibrium points and the possible existence of up to two endemic equilibrium
points (where the disease persists in the population). We show the existence of a transcritical bifurcation and a possible saddle-node bifurcation and explicitly derive threshold conditions for both, including defining the basic reproduction number, $\mathcal R_0$, which determines whether the disease can persist in the population or not.~The epidemiological consequence of saddle-node bifurcation (backward bifurcation) is that the classical requirement of having the reproduction number less than unity, while necessary, is no longer sufficient for disease elimination from the population. It is further shown that in the absence of disease--induced death, the model does not exhibit this phenomenon. We perform the sensitivity analysis to determine the model robustness to parameter values. That is to help us to know the parameters that are most influential in determining disease dynamics. The model is extended by reformulating the model as an optimal control problem, with the use of five time dependent controls, to assess the impact of vaccination combined with treatment, individual protection and vector control strategies (killing adult vectors, reduction of eggs and larvae). By using optimal control theory, we establish optimal conditions under which the disease can be eradicated and we examine the impact of a possible combined control tools on the disease transmission. The Pontryagin's maximum principle is used to characterize the optimal control. Numerical simulations, efficiency analysis and cost effectiveness analysis show that, vaccination combined with other control mechanisms, would reduce the spread of the disease appreciably, and this at low cost.
}
\medskip

\medskip


\section{Introduction}
\label{sec:IntAr3opc}

Arboviral diseases are affections transmitted by hematophagous arthropods. There are currently 534 viruses registered in the International Catalog of Arboviruses and 25\% of them have caused documented illness in human populations \cite{ch,ka,gu}. Examples of those kinds of diseases are Dengue, Yellow fever, Saint Louis fever, Encephalitis, West Nile fever and Chikungunya. A wide range of arboviral diseases are transmitted by mosquito bites and constitute a public health emergency of international concern. For example, Dengue, caused by any of four closely-related virus serotypes (DEN-1-4) of the genus Flavivirus, causes 50--100 million infections worldwide every year, and the majority of patients worldwide are children aged 9 to 16 years \cite{Sanofi2013,WHO2006,WHO2009}. 

The dynamics of arboviral diseases like Dengue or Chikungunya are influenced by many factors such as human and mosquito behaviours. The virus itself (multiple serotypes of dengue virus~\cite{WHO2006,WHO2009}, and  multiple strains of chikungunya virus~\cite{moulayThesis,Parola}), as well as the environment, affects directly or indirectly all the present mechanisms of control~\cite{BrasseurThsesis,Carvalho2015}. Indeed, in the absence of conditions which favour the development of their larvae, eggs of certain Aedes mosquitoes (Aedes albopictus, for example) enter in diapause phenomenon, allowing the eggs to hatch even after two years~\cite{Paupy2009,Sota1992}. 
Taking the case of Aedes mosquitoes for example, the main control method used in many countries continues to be space spraying of insecticide for adult mosquito control. This strategy must be repeated constantly, its cost is high, and its effectiveness is limited. Also, Ae. aegypti, for example, prefers to rest inside houses, so truck or aerial insecticide spraying simply does not reach mosquitoes resting in hidden places such as cupboards~\cite{WHO2004}. 
The different types of control mechanisms put in place to reduce the proliferation of vectors responsible for the transmission of pathogens such as arboviruses are listed below..

\begin{itemize}
	\item[(i)] Biological control or "biocontrol" is the use of natural enemies to manage vector populations: introduction of parasites, pathogens and predators to target vectors. for example, effective biocontrol agents include predatory fish that feed on mosquito larvae such as mosquitofish (Gambusia affinis) and some cyprinids (carps and minnows) and killifish. Tilapia also consume mosquito larvae~\cite{Alcaraz2007}. As biological control does not cause chemical pollution, it is considered as a better method for mosquito control by many people. However, there are limitations on employing biological agents for mosquito control. The agent introduced usually has to be substantial in number for giving desirable effect.
	\item[(ii)] Mechanical control consist at the environmental sanitation measures to reduce mosquito breeding sites, such as the physical management of water containers (e.g. mosquito-proof covers for water storage containers, polystyrene beads in water tanks), better designed and reliable water supplies, and recycling of solid waste such as discarded tyres, bottles, and cans~\cite{WHO2004,duch}.
	\item[(iii)] Chemical methods~\cite{WHO2004,duch}: 
	\begin{itemize}
		\item chemical methods against the mosquito's aquatic stages for use in water containers (larviciding --killing of larvae),
		\item chemical methods directed against adult mosquitoes, such as insecticide space sprays or residual applications (adulticiding --killing of adult mosquitoes),.
	\end{itemize}
	\item[(iv)] Personal protection consist at the use of repellents, vaporizers, mosquito coils, and insecticide treated screens, curtains, and bednets (for daytime use against Aedes)~\cite{WHO2004}.
\end{itemize}

The main problem encountered in the implementation of some of these control mechanisms is the preservation of the ecological systems. For example, in the "biocontrol" mechanism, direct introduction of tilapia and mosquitofish into ecosystems around the world have had disastrous consequences~\cite{Alcaraz2007}. Also, the chemical methods can not be applied in continuous times. Some chemical product like {\it Deltamethrin} seems to be effective only during a couple of hours~\cite{duch,Bosc,HelenaSofia}. So its use over a long period and continuously, leads to strong resistance of the wild populations of {\it Aedes aegypti}~\cite{IRD}, for example.


For all the diseases mentioned above, only yellow fever has a licensed vaccine. Nevertheless, considerable efforts are made to obtain vaccines for other diseases. In the case of dengue, for example, tests carried out in Asia and Latin America, have shown that the future dengue vaccine will have a efficacy between 30.2\%  and 77.7\%, and this, depending on the serotype~\cite{Arunee2012,Villar2015}. Also, the future dengue vaccine will have an overall efficacy of 60.8\% against all forms of the disease in children and adolescents aged 9-16 years who received three doses of the vaccine\cite{Sanofi2014}.
 
As the future vaccines (e.g., dengue vaccine) will be imperfect, it is therefore necessary to combine such vaccines with some control mechanisms cited above, to find the best sufficient combination (in terms of efficacy and costs), which permit to decrease the expansion of these kind of diseases in human communities.

A number of studies have been conducted to study host-vector models for arboviral diseases transmission (see~
\cite{duch,AbbouEtAl2015,AbbouEtAl2016,Aldila2013,Antonio2001,BlaynehaetAl,caga,coutinho2006,cretal,Derouich2006,esva98,esva99,fevh,gaguab,maya,moaaca,moaaHee2012,poetal,HelenaSofiaRodrigues2014}). 
Some of these works have been conducted to explore optimal control theory for arboviral disease models~(see~\cite{Aldila2013,BlaynehaetAl,moaaHee2012,HelenaSofiaRodrigues2014,Dias2015}). 

In~\cite{Aldila2013}, Dipo Aldila and co-workers derive a optimal control problem for a host-vector Dengue transmission model, in which treatments with mosquito repellent are given to adults and children and those who undergo treatment are classified in treated compartments. The only control considered by the authors is the treatment of people with clinical signs of the disease. 
Blayneh et {\it al.} in~\cite{BlaynehaetAl} consider a deterministic model for the transmission dynamics of West Nile virus (WNV) in the mosquito-bird-human zoonotic cycle. They use two control functions, one for mosquito-reduction strategies and the other for personal (human) protection, and redefining the demographic parameters as density-dependent rates.
In~\cite{moaaHee2012}, Moulay et {\it al.} derive optimal prevention (individual protection), vector control (Larvae reduction) and treatment strategies used during the Chikungunya R\'eunion Island epidemic in 2006.
Authors in~\cite{HelenaSofiaRodrigues2014} derive the optimal control efforts for vaccination in order to prevent the spread of a Dengue disease using a system of ordinary differential equations (ODEs) for the host and vector populations.
Recently, Dias et {\it al.} in~\cite{Dias2015} analyse the Dengue vector control problem in a multiobjective optimization approach, in which the intention is to minimize both social and economic costs, using a dynamic mathematical model representing the mosquitoes' population. This multiobjective optimization approach consists in finding optimal alternated step-size control policies combining chemical (via application of insecticides) and biological control (via insertion of sterile males produced by irradiation).

None of the above mentioned models ~\cite{Aldila2013,BlaynehaetAl,moaaHee2012,HelenaSofiaRodrigues2014,Dias2015} takes into account the combination of optimal control mechanisms such as vaccination, individual protection, treatment and vector control strategies.
In our effort, we investigate such optimal strategies for vaccination combined with individual protection, treatment and two vector controls (adulticiding--killing of adult vectors, and larviciding--killing eggs and larvae), using two systems of ODEs which consist of a complete stage structured model Eggs-Larvae-Pupae for the vectors, and a SEI/SEIR type model for the vector/host population. This provides a new different mathematical perspective to the subject. Furthermore, a efficiency analysis and cost effectiveness analysis, are performed here in order to evaluate the control combination that is most effective in the design of optimal strategies.

We start with the formulation of a model without control which is an modified of the previous models developed in \cite{AbbouEtAl2015,AbbouEtAl2016}. 
We compute the net reproductive number $\mathcal{N}$, as well as the basic reproduction number, $\mathcal R_0$, and investigate the existence and stability of equilibria. We prove that the trivial equilibrium is globally asymptotically stable whenever $\mathcal{N}<1$. When $\mathcal{N}>1$ and $\mathcal R_0<1$, we prove that the system exhibit the backward bifurcation phenomenon. The implication of this occurrence is that the classical epidemiological requirement for effective eradication of the disease, $\mathcal R_0<1$, is no longer sufficient, even though necessary. We show the existence of a transcritical bifurcation and a possible saddle-node bifurcation and explicitly derive threshold conditions for both.

Then, we formulate an optimal control model by adding five control functions: three for human  (vaccination, protection against mosquitoes bites and treatment), and two for mosquito-reduction strategies (the use of adulticide to kill adult vectors, and the use of larvicide to increase the mortality rate of eggs and larvae). By Using optimal control theory, we derive the conditions under which it is optimal to eradicate the disease and examine the impact of a possible combination of vaccination, treatment, individual protection and vector control strategies on the disease transmission. The Pontryagin's maximum principle is used to characterize the optimal control. Numerical simulations, efficiency analysis, as well as, the cost effectiveness analysis, are performed to determine the best combination (in terms of efficacy and cost).

The rest of the paper is organized as follows. In section~\ref{sec:basicModelOPC}, we present the basic transmission model and carry out some analysis by determining important thresholds such as the net reproductive number $\mathcal{N}$ and the basic reproduction number $\mathcal R_0$, and different equilibria of the model. We then demonstrate the  stability of equilibria and carry out bifurcation analysis, by deriving the threshold conditions for saddle--node bifurcation. In Section~\ref{sec2ar3opc} we present the optimal control problem and its mathematical analysis. Section \ref{NUMopc} is devoted to numerical simulations, efficiency analysis and cost effectiveness analysis. A conclusion round up the paper.

\section{The basic model and its analysis}
\label{sec:basicModelOPC}
The model we propose here is based on the modelling approach given in \cite{AbbouEtAl2015,AbbouEtAl2016}. For the reader convenience, we briefly recall here
main results which are developed in this work.

It is assumed that the human and vector populations are divided into compartments described by time--dependent state variables. The compartments in which the populations are divided are the following ones:

\indent(i) For humans, we consider a SEIR model: Susceptible (denoted by $S_h$), exposed ($E_h$), infectious ($I_h$) and resistant or immune ($R_h$) which includes naturally-immune individuals. The recruitment in human population is at the constant rate $\Lambda_h$, and newly recruited individuals enter the susceptible compartment $S_h$. Are concern by recruitment people that are totally naive from the disease. Each human compartment, individual goes out from the dynamics at natural mortality rates $\mu_h$. The human susceptible population is decreased following infection, which can be acquired via effective contact with an exposed or infectious vector at a rate 
\begin{equation}
\label{fihOPC}
\lambda_h=\dfrac{a\beta_{hv}(\eta_vE_v+I_v)}{N_h},
\end{equation}
where $a$ is the biting rate per susceptible vector, $\beta_{hv}$ is the transmission probability from an exposed/infectious vector ($E_v$ or $I_v$) to a susceptible human ($S_h$). The expression of $\lambda_h$ is obtained as follows. The probability that a vector chooses a particular human or other source of blood to bite can be assumed as $\dfrac{1}{N_h}$. Thus, a human receives in average $a\dfrac{N_v}{N_h}$ bites per unit of times. Then, the infection rate per susceptible human is given $a\beta_{hv}\dfrac{N_v}{N_h}\dfrac{(\eta_vE_v+I_v)}{N_v}$. In expression of $\lambda_h$, the modification parameter $0 <\eta_v< 1$ accounts for the assumed reduction in transmissibility of exposed mosquitoes relative to infectious mosquitoes~\cite{AbbouEtAl2015,AbbouEtAl2016,gaguab} (see the references therein for the specific sources). Latent humans ($E_h$) become infectious ($I_h$) at rate $\gamma_h$. Infectious humans recover at a constant rate, $\sigma$ or dies  as consequence of infection, at a disease-induced death rate $\delta$. Immune humans retain their immunity for life. 
We denote the total human population by $N_h$,
	\begin{equation}
		\label{NhBasic}
	N_h = S_h + E_h + I_h + R_h.
	\end{equation}

\indent(ii) Following \cite{moaaca}, the stage structured model is used to describe the vector population dynamics, which consists of three main stages: embryonic (E), larvae (L) and pupae (P). Even if eggs (E) and immature stages (L and P) are both aquatic, it is important to dissociate them because, for optimal control point of view, drying the breeding sites does not kill eggs, but only larvae and pupae. Moreover, chemical interventions on the breeding sites have a more great impact on the larval population, but not on the eggs~\cite{moaaca}. The number of laid eggs is assumed proportional to the 
number of females. The system of stage structured model of aquatic phase development of vector is given by (see \cite{moaaca} for details) 
\begin{subequations}
\label{AR3}
\begin{align}
\dot{E}&=\mu_b\left(1-\dfrac{E}{\Gamma_{E}} \right)(S_v+E_v+I_v)-(s+\mu_E)E\\
\dot{L}&=sE\left(1-\dfrac{L}{\Gamma_{L}} \right)-(l+\mu_L)L\\
\dot{P}&=lL-(\theta+\mu_P)P.
\end{align}
\end{subequations}
Unlike the authors of \cite{moaaca}, we take into account the pupal stage in the development of the vector. This is justified by the fact that they do not feed during this transitional stage of development, as they transform from larvae to adults. So, the control mechanisms can not be applied to them. 

With a rate $\theta$, pupae become female adults. Each vector compartment, individuals goes out from the dynamics at natural mortality rates $\mu_v$. The vector susceptible population is decreased following infection, which can be acquired via effective contact with an exposed or infectious human at a rate \begin{equation}
\label{fivOPC}
\lambda_v=\dfrac{a\beta_{vh}(\eta_hE_h+I_h)}{N_h},
\end{equation}
 where $\beta_{vh}$ is the transmission probability from an exposed/infectious human ($E_h$ or $I_h$) to a susceptible vector ($S_v$). As well as in the expression of $\lambda_h$, the modification parameter $0 <\eta_h< 1$ in the expression of $\lambda_v$ accounts for the assumed reduction in transmissibility of exposed humans relative to infectious humans~\cite{AbbouEtAl2015,AbbouEtAl2016,gaguab}. Latent vectors ($E_v$) become infectious ($I_v$) at rate $\gamma_v$. The vector population does not have an immune class, since it is assumed that their infectious period ends with their death \cite{esva99}. So, we denote the total adult vector population by $N_v$,
	\begin{equation}
	\label{NvBasic}
	N_v = S_v + E_v + I_v.
	\end{equation}

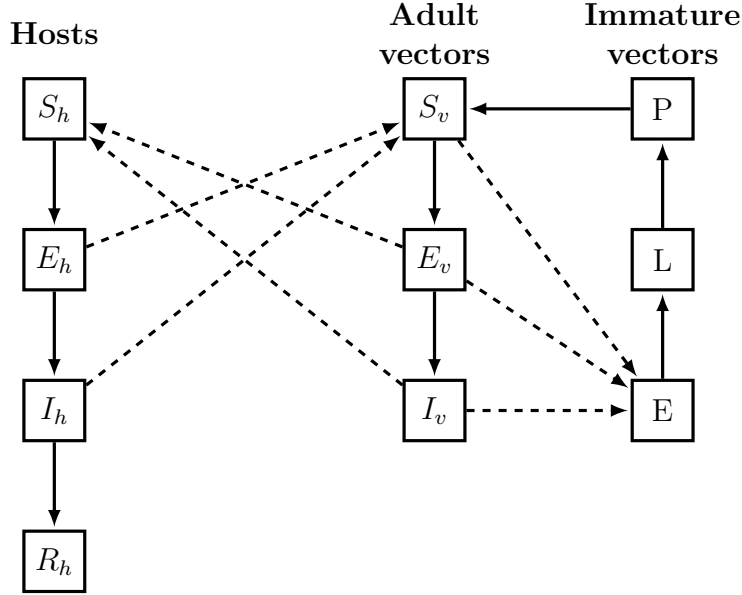
\begin{figure}[h!]
	\begin{center}
		\begin{tikzpicture}
		[host/.style={rectangle,draw=black!0,fill=black!0,very thick,
			inner sep=3pt,minimum size=8mm, font=\bf, align=center},
		compartment/.style={rectangle,draw=black!100,fill=black!0,very thick,
			inner sep=3pt,minimum size=8mm},
		intervention/.style={align=center,red},
		bend angle = 15,
		post/.style={->,shorten >=1pt,>=latex, very thick},
		pre/.style={<-,shorten >=1pt,>=latex, very thick},
		both/.style={<->,shorten >=1pt,>=latex, very thick},
		predashed/.style={dashed,<-,shorten >=1pt,>=latex, very thick},
		postdashed/.style={dashed,->,shorten >=1pt,>=latex, very thick},
		]
		
		\node  [host]  (hosts)        at (0,1)        {Hosts};
		\node  [host]  (hosts)        at (5,1)        {Adult\\ vectors};
		\node  [host]  (hosts)        at (8,1)        {Immature\\ vectors};
		\node [compartment] (noeud1) at (0,0) {$S_h$};
		\node [compartment] (noeud3) at (0,-2) {$E_h$};
		\node [compartment] (noeud4) at (0,-4) {$I_h$};
		\node [compartment] (noeud5) at (0,-6) {$R_h$};
		
		\node [compartment] (noeud6) at (5,0) {$S_v$};
		\node [compartment] (noeud7) at (5,-2) {$E_v$};
		\node [compartment] (noeud8) at (5,-4) {$I_v$};
		
		\node [compartment] (noeud9) at (8,0) {P};
		\node [compartment] (noeud10) at (8,-2) {L};
		\node [compartment] (noeud11) at (8,-4) {E};
		\draw[post] (noeud1)-- (noeud3);
		\draw[post] (noeud3)-- (noeud4);
		\draw[post] (noeud4)-- (noeud5);
		\draw[post] (noeud6)-- (noeud7);
		\draw[post] (noeud7)-- (noeud8);
		\draw[postdashed] (noeud3)-- (noeud6);
		\draw[postdashed] (noeud4)-- (noeud6);
		\draw[postdashed] (noeud7)-- (noeud1);
		\draw[postdashed] (noeud8)-- (noeud1);
		\draw[post] (noeud11)-- (noeud10);
		\draw[post] (noeud10)-- (noeud9);
		\draw[post] (noeud9)-- (noeud6);
		\draw[postdashed] (noeud6)-- (noeud11);
		\draw[postdashed] (noeud7)-- (noeud11);
		\draw[postdashed] (noeud8)-- (noeud11);
		
		\end{tikzpicture}
	\end{center}
	\caption{Schematic of the vector-borne epidemic model with development stage of vectors.} \label{shematicBasicModel}
\end{figure}

Therefore, our basic arboviral disease model reads as
\begin{subequations}
	\label{BasicModel}
	\begin{align}
	\dot{S}_h&=\Lambda_h-(\lambda_h+\mu_h) S_h \\
	\dot{E}_h&=\lambda_hS_h-(\mu_h+\gamma_h)E_h\\
	\dot{I}_h&=\gamma_hE_h-\left[\mu_h+\delta+\sigma\right] I_h\\
	\dot{R}_h&=\sigma I_h-\mu_hR_h\\
	\dot{S}_v&=\theta P-\lambda_vS_v-\mu_vS_v\\
	\dot{E}_v&=\lambda_vS_v-(\mu_v+\gamma_v)E_v\\
	\dot{I}_v&=\gamma_vE_v-(\mu_v)I_v\\
	\dot{E}&=\mu_b\left(1-\dfrac{E}{\Gamma_E} \right)N_v-(s+\mu_E)E\\
	\dot{L}&=sE\left(1-\dfrac{L}{\Gamma_L} \right)-(l+\mu_L)L\\
	\dot{P}&=lL-(\theta+\mu_P)P
	\end{align}
\end{subequations}
where the upper dot denotes the time derivative and $\lambda_h$ and $\lambda_v$ are given by \eqref{fihOPC} and \eqref{fivOPC}, respectively. A schematic of the model is shown in Figure \ref{shematicBasicModel}. The states and parameters are all strictly positive constants and are described in Table \ref{octab1} and \ref{octab2}, respectively.

\begin{table}[h!]
	\begin{center}
		\caption{The state variables of model (\ref{BasicModel}).}
		\label{octab1}
		\begin{tabular}{rlrlrlrl}
			\hline
			& Humans  &  &Aquatic Vectors&& Adult Vectors \\
			\hline
			$S_{h}$:& Susceptible & $E$:& Eggs&$S_{v}$:& Susceptible\\
			$E_{h}$:& Infected in latent stage  &$L$:& Larvae&$E_v$& Infected in latent stage\\
			$I_{h}$:& Infectious   &$P$:& Pupae&$I_v$&Infectious\\
			$R_{h}$:& Resistant (immune)  & & & &\\
			\hline
		\end{tabular}
	\end{center}
\end{table}
\begin{table}[h!]
	\begin{center}
		\caption{Description and baseline values/range of parameters of model (\ref{BasicModel}). The baseline values refer to dengue fever transmission. \label{octab2}} 
		\begin{tabular}{llll}
			\hline\noalign{\smallskip}
			Parameter& Description&Baseline&Sources\\&& value/range &\\
			\noalign{\smallskip}\hline\noalign{\smallskip}
			$\Lambda_h$ & Recruitment rate of humans &  2.5 $day^{-1}$&\cite{gaguab}  \\
			$\mu_h$ & Natural mortality rate in humans & $\frac{1}{(67\times 365)}~day^{-1}$&\cite{gaguab}\\
			$a$ & Average number of bites& 1~$day^{-1}$&\cite{Aldila2013,gaguab}\\
			$\beta_{hv}$ & Probability of transmission of &0.1, 0.75 $day^{-1}$&\cite{Aldila2013,gaguab}\\
			&infection from an infected vector&&\\
			&to a susceptible human&&\\
			$\gamma_h$ & Progression rate from $E_h$ to $I_h$ &$\left[\frac{1}{15},\frac{1}{3}\right]~day^{-1}$ &\cite{duch,Scott2010}\\
			$\delta$ & Disease--induced death rate& 10$^{-3}~day^{-1}$ &\cite{gaguab}\\
			$\sigma$ & Recovery rate for humans & 0.1428~$day^{-1}$&\cite{Aldila2013,gaguab}\\
			$\eta_h$,$\eta_v$& Modifications parameter&$\left[0,1\right)$&\cite{gaguab}\\
			$\mu_v$ & Natural mortality rate of vectors &$\left[\frac{1}{30},\frac{1}{14}\right]~ day^{-1}$&\cite{Aldila2013,gaguab}\\
			$\gamma_v$ & Progression rate from $E_v$ to $I_v$ &$\left[\frac{1}{21},\frac{1}{2}\right]day^{-1}$&\cite{duch,Scott2010} \\
			$\beta_{vh}$ &Probability of transmission of &0.1, 0.75~$day^{-1}$&\cite{Aldila2013,gaguab}\\
			&infection from an infected human&&\\
			&to a susceptible vector&&\\
			$\theta$& Maturation rate from pupae to adult & 0.08~$day^{-1}$&\cite{duch,moaaca,moaaHee2012}\\
			$\mu_b$&Number of eggs at each deposit&6~$day^{-1}$&\cite{duch,moaaca,moaaHee2012}\\
			$\Gamma_{E}$&Carrying capacity for eggs&$10^3,10^6$&\cite{Aldila2013,moaaca}\\
			$\Gamma_{L}$&Carrying capacity for larvae&$5\times 10^2,5\times10^5$&\cite{Aldila2013,moaaca}\\
			$\mu_{E}$&Eggs death rate  &0.2 or 0.4 &\cite{moaaHee2012}\\
			$\mu_{L}$&Larvae death rate&0.2 or 0.4&\cite{moaaHee2012}\\
			$\mu_{P}$&Pupae death rate&$0.4$&Assumed\\
			$s$&Transfer rate from eggs to larvae&0.7~$day^{-1}$&\cite{moaaHee2012}\\
			$l$&Transfer rate from larvae to pupae&0.5~$day^{-1}$&\cite{moaaca}\\
			\noalign{\smallskip}\hline
		\end{tabular}
	\end{center}
\end{table}
\begin{remark}
	\indent(i) It is important to note that, in	the case of other arboviral diseases (e.g, Chikungunya), the exposed humans and vectors do not play any role in the infectious process, in this case $\eta_h=\eta_v=0$.	\\
	\indent(ii) The model \eqref{BasicModel} is the same that we studied in a previous work (see model 18 in~\cite{AbbouEtAl2016}) to show that the backward bifurcation is caused by the disease--induced death in human. In this previous work, we have just showed that the occurrence of the backward bifurcation is possible in the model without vaccination. In the present work, we give a sufficient and necessary condition 
	, as well as the explicit expressions of the thresholds which governing this phenomenon. 
\end{remark}

\subsection{Basic properties and equilibria}
The rates of change of the total populations of humans~\eqref{NhBasic} and adult vectors \eqref{NvBasic} for the basic arboviral model \eqref{BasicModel} are,
\begin{align*}
\dot N_h &= \Lambda_h - \mu_h N_h - \delta I_h, \\
\dot N_v &=\theta P-\mu_v N_v.
\end{align*}
Therefore, by standard arguments (see~\cite{AbbouEtAl2015,AbbouEtAl2016,moaaca}) it follows that the feasible region for model \eqref{BasicModel} is
{\footnotesize
\begin{equation*}
\begin{split}
\mathcal{D}&=\left\lbrace
(S_h,E_h,I_h,R_h,S_v,E_v,I_v,E,L,P)\in\mathbb R^{10}_{+}:
N_h\leq\Lambda_h/\mu_h; E\leq K_E; L\leq K_L;P\leq\dfrac{lK_L}{k_7};N_v\leq \dfrac{\theta lK_L}{k_7k_8}\right\rbrace,\\
\end{split}
\end{equation*}
}
where $\mathbb{R}^{10}_+$ represents the non-negative orthant of $\mathbb{R}^{10}$.

The model is epidemiologically (the state variables have a valid physical interpretation) and mathematically (the system of equations has a unique solution that is bounded and exists for all time) well-posed in the region $\mathcal{D}$.

For easier readability, we introduce the following quantities,
\begin{equation}
\label{intervaria}
\begin{array}{l}
k_1:=\mu_h; k_3:=\mu_h+\gamma_h;\,\,\,
k_4:=\mu_h+\delta+\sigma,\,\, 
k_5:=s+\mu_E;\\
k_6:=l+\mu_L;\,\,\,k_7:=\theta+\mu_P;\,\,\,k_8:=\mu_v;\,\,
k_9:=\mu_v+\gamma_v; k_{10}=\eta_hk_4+\gamma_h;\,\,k_{11}=\eta_vk_8+\gamma_v,
\end{array}
\end{equation}
and (the positive quantity), $k_2=k_3k_4-\delta\gamma_h =\mu_hk_4+\gamma_h(\mu_h+\sigma)$.

Without disease in the both populations (i.e $\lambda_h=\lambda_v=0$ or $E_h=I_h=E_v=I_v=0$), the basic arboviral model~\eqref{BasicModel} have two \emph{disease--free equilibria } given by
$\mathcal{E}_{0}=\left(N^{0}_h,0,0,0,0,0,0,0,0,0\right)$ which correspond to the trivial equilibrium, and
$\mathcal{E}_{1}=\left(N^{0}_h,0,0,0,N^{0}_v,0,0,E,L,P\right)$ which correspond to the biological disease--free equilibrium, where
\begin{equation}
\label{TEandDFE}
\begin{array}{l}
N^{0}_h=\dfrac{\Lambda_h}{\mu_h},\,\,\,
N^{0}_v=\dfrac{\Gamma_{E}\Gamma_{L}k_{5}k_{6}\left(\mathcal{N}-1\right)} {\mu_{b}\left(\Gamma_{E}s+k_{6}\Gamma_{L}\right)},
P=\dfrac{\Gamma_{E}\Gamma_{L}k_{5}k_{6}k_{8}\left(\mathcal{N}-1\right)}
{\mu_{b}\theta\left(\Gamma_{E}s+k_{6}\Gamma_{L}\right)}\\
L=\dfrac{\Gamma_{E}\Gamma_{L}k_{5}k_{6}k_{7}k_{8}\left(\mathcal{N}-1\right)}{\mu_{b}\theta l\left(\Gamma_{E}s+k_{6}\Gamma_{L}\right)},\,\,
E=\dfrac{\Gamma_{E}\Gamma_{L}k_{5}k_{6}k_{7}k_{8}\left(\mathcal{N}-1\right)}
{s\left(\mu_{b}l\Gamma_{L}\theta+k_{5}k_{7} k_{8}\Gamma_{E}\right)},
\end{array} 
\end{equation}
and $\mathcal{N}$ is the net reproductive number~\cite{AbbouEtAl2016,moaaca} given by
\begin{equation}
\label{Nar3}
\mathcal{N}=\dfrac{\mu_b\theta ls}{k_5k_6k_7k_8}.
\end{equation}

Define the basic reproductive number~\cite{dihe,vawa02}
\begin{equation}
\label{R0ar3}
\begin{split}
\mathcal R_{0}=\sqrt{\dfrac{a^2\beta_{hv}\beta_{vh}(\gamma_{h}+k_4\eta_h)(\gamma_v+k_8\eta_v)N^{0}_{v}}{k_{3}k_{4}k_{8}k_{9}N^{0}_{h}}}.
\end{split}
\end{equation}

We note that,
\begin{equation*}
\mathcal	R_0 = \sqrt{K_{vh} K_{hv}},
\end{equation*}
where,
\begin{equation*}
\begin{split}
	K_{vh} &=(K^{E_h}_{vh}+K^{I_h}_{vh})\\
	&= \left(a\right)\left(\beta_{vh}\right)\left(\frac{N^{0}_v}{N^{0}_h}\right)
	\left(\frac{1}{k_3}\right)
	\left[\left(\eta_{h}\right)
	+\left(\frac{\gamma_h}{k_4}\right)\right]\\
	&=\dfrac{a\beta_{vh}(\gamma_{h}+k_4\eta_h)N^{0}_{v}}{k_{3}k_{4}N^{0}_{h}},
	\end{split}
\end{equation*}
is the number of vector that one human infects through his/her latent/infectious life time. It is equal to the sum of the number of vector infections generated by an exposed human (near the DFE, $\mathcal E_1$) $K^{E_h}_{vh}$, and the number of vector infections generated by an infectious human (near the DFE) $K^{I_h}_{vh}$.
$K^{E_h}_{vh}$ is given by the product of the infection rate of exposed humans ($a\beta_{vh}\eta_hN^{0}_v/N^{0}_h$) and the average duration in the exposed ($E_h$ ) class ($1/k_3$). $K^{I_h}_{vh}$ is given by the product of the infection rate of infectious humans ($a\beta_{vh}N^{0}_v/N^{0}_h$), the probability that an exposed human survives the exposed stage and move to the infectious stage ($\gamma_h/(\mu_h+\gamma_h)$) and the average duration in the infectious stage ($1/(\mu_h+\delta+\sigma)$).

Analogously, we have
\begin{equation*}
\begin{split}
K_{hv}&=K^{E_v}_{hv}+K^{I_v}_{hv}\\
&=\left(a\right)\left(\beta_{hv}\right)\left(\frac{1}{k_9}\right)
\left(\eta_{v}+\frac{\gamma_v}{k_8}\right)\\
&=\dfrac{a\beta_{hv}(\gamma_{v}+k_8\eta_h)}{k_{8}k_{9}},
\end{split}
\end{equation*}
which is the number of humans that one vector infects through its infectious life time.
It is equal to the sum of the number of human infections generated by an exposed vector (near the DFE, $\mathcal E_1$),$K^{E_v}_{hv}$, and the number of human infections generated by an infectious vector (near the DFE), $K^{I_v}_{hv}$. $K^{E_v}_{hv}$ is given by the product of the infection rate of exposed vectors ($a\beta_{vh}\eta_h$) and the average duration in the exposed ($E_v$ ) class ($1/(\mu_v+\gamma_v)$). $K^{I_v}_{vh}$ is given by the product of the infection rate of infectious humans ($a\beta_{vh}$), the probability that an exposed human survives the exposed stage and move to the infectious stage ($\gamma_h/k_9$) and the average duration in the infectious stage ($1/k_8$). The basic reproduction number is equal to the geometric mean of $K_{vh}$ and $K_{hv}$ because infection from human to human goes through one generation of vectors.

The local asymptotic stability result of equilibria $\mathcal{E}_{0}$ and $\mathcal{E}_{1}$  is given in the following.
\begin{theorem}
	\label{th3Basic}$\;$
	
	\begin{itemize}
		\item[(i)] if $\mathcal{N}\leq 1$, the trivial equilibrium $\mathcal{E}_{0}$ is locally asymptotically stable in $\mathcal{D}$; 
		\item[(ii)] if $\mathcal{N}>1$, the trivial equilibrium is unstable and the disease--free equilibrium $\mathcal{E}_{1}$ is locally asymptotically stable in $\mathcal{D}$ whenever $\mathcal R_0< 1$.
	\end{itemize}
\end{theorem}
\begin{proof}
	See appendix~\ref{AppBasic1}.\hfill.\hfill
\end{proof}
The epidemiological implication of item (ii) in Theorem \ref{th3Basic} is that, in general, when the basic reproduction number, $\mathcal R_0$ is less than unity, a small influx of infectious vectors into the community would not generate large outbreaks, and the disease dies out in time (since the DFE is locally asymptotically stable)~\cite{AbbouEtAl2016,gaguab,dihe,vawa02,cretal2005}. However, we will show in the subsection~\ref{sec:BackwardBasicModel} that the disease may still persist even when $\mathcal R_0 < 1$.

The global stability of the trivial equilibrium is given by the following result:
\begin{theorem}
	\label{GasTEBasic} 
	If $\mathcal{N}\leq 1$, then $\mathcal{E}_0$ is globally asymptotically stable on $\mathcal{D}$.
\end{theorem}
\begin{proof}
	See appendix~\ref{AppBasic2}.\hfill
\end{proof}

We now show the existence of endemic equilibria, that is, steady states of model \eqref{BasicModel} where all state variables are positive.  First we introduce:
\begin{align}
	\label{psiar3opc}
	\psi &=k_{10}a\mu_{h}\beta_{vh}-\delta\gamma_hk_{8}, \\
	\label{R_car3opc}
	\mathcal R_{c} &=\sqrt{\dfrac{2k_{8}k_{2} +k_{10}a\mu_{h}\beta_{vh}}{k_{3}k_{4}k_{8}}} , \\
	\label{R_1bar3opc}
	\mathcal R_{1b} &=\dfrac{1}{k_{3}k_{4}}
	\left( \sqrt{\dfrac{1}{k_{8}k_{9}}}\right) 
	\left| \sqrt{\delta\gamma_h(a\mu_{h}\beta_{vh}k_{10}+k_{2}k_{8})}
	-\sqrt{(-k_{2}\psi)}\right| , \\
	\label{R_2bar3opc}
	\mathcal R_{2b} &=\dfrac{1}{k_{3}k_{4}}
	\left( \sqrt{\dfrac{1}{k_{8}k_{9}}}\right) 
	\left( \sqrt{\delta\gamma_h(a\mu_{h}\beta_{vh}k_{10}+k_{2}k_{8})}
	+\sqrt{(-k_{2}\psi)}\right) .
\end{align}

Note that (as shown in the Appendix~\ref{AppBasic3}), when $\mathcal R_c < \mathcal R_0 < 1$, $\psi \leq 0$ and correspondingly, $\mathcal R_{1b}$ and $\mathcal R_{2b}$ are real.

With the inequalities,
\begin{subequations}
\begin{align}
	\label{R0sn2Basic1}
	\mathcal R_{c}<\mathcal R_{0}<\min(1,\mathcal R_{1b}), \\
	\label{R0sn2Basic2}
	\max(\mathcal R_{c},\mathcal R_{2b})<\mathcal R_{0}<1,
\end{align}
\end{subequations}
we claim the following result:
\begin{theorem}
	\label{EEBasic}
The number of endemic equilibrium points of the basic arboviral disease model~\eqref{BasicModel} depends on $\mathcal R_0$ as follows:
	\begin{enumerate}
		\item[(i)] For $\mathcal R_0 > 1$, the system has a unique endemic equilibrium point.
		\item[(ii)] For $\mathcal R_0 = 1$, the system has
		\begin{enumerate}
			\item A unique endemic equilibrium point if $\mathcal R_c < 1$.
			\item No endemic equilibrium points otherwise.    
		\end{enumerate}
 \item[(iii)] For $\mathcal R_0 < 1$, the system has
\begin{enumerate}
	\item Two endemic equilibrium points if either inequality \eqref{R0sn2Basic1} or \eqref{R0sn2Basic2} is satisfied.
	\item A unique endemic equilibrium point if $\mathcal R_c < \mathcal R_0$ and either $\mathcal R_0 = \mathcal R_{1b}$ or $\mathcal R_0 = \mathcal R_{2b}$.
	\item No endemic equilibrium points otherwise.
\end{enumerate}
	\end{enumerate}
\end{theorem}
\begin{proof}
	See appendix~\ref{AppBasic3}.\hfill
\end{proof}
It is clear that case \emph{(iii)} (item \emph{(a)}) of theorem \ref{EEBasic} indicate the possibility of backward bifurcation (where the locally-asymptotically stable DFE co-exists with a locally asymptotically stable endemic equilibrium when $\mathcal R_{0}<1$) in the model~\eqref{BasicModel}. In a previous work (see model 18 in~\cite{AbbouEtAl2016}), we just showed that the model exhibited the backward bifurcation phenomenon. In the following, we provide not only a sufficient condition, but also the thresholds which governing this phenomenon.


\subsection{Bifurcation analysis}
\label{sec:BackwardBasicModel}
Here, we use the centre manifold theory \cite{guho} to explore the possibility of bifurcation in \eqref{AR3} at criticality (i.e. the existence and stability
of the equilibrium points bifurcating from $\mathcal E_1$ at $\mathcal R_0=1$) by studying the centre manifold near the criticality through the approach developed in \cite{vawa02,duhucc,ccso,Carr}, which is based on general centre manifold theory~\cite{guho}. To do so, a bifurcation parameter $\beta^{*}_{hv}$ is chosen, by solving for $\beta_{hv}$ from $\mathcal R_0=1$, giving
\begin{equation}
\label{bifparam}
\beta^{*}_{hv}=\dfrac{k_{3}k_{4}k_{8}k_{9}N^{0}_{h}}{a^2\beta_{vh}k_{10}k_{11}N^{0}_{v}}.
\end{equation}
Let $J_{\beta^{*}_{hv}}$ denotes the Jacobian of the system \eqref{AR3} evaluated at the DFE ($\mathcal{E}_1$ ) and with $\beta_{hv}=\beta^{*}_{hv}$. Thus,
{\footnotesize
	$
	J=\left( \begin{array}{cccccccccc}
	-k_1&0&0&0&0&-\beta^{*}_{hv}\eta_v&-\beta^{*}_{hv}&0&0&0\\
	0&-k_3&0&0&0&\beta^{*}_{hv}\eta_v&\beta^{*}_{hv}&0&0&0\\
	0&\gamma_h&-k_4&0&0&0&0&0&0&0\\
	0&0&\sigma&-\mu_h&0&0&0&0&0&0\\
	0&-\dfrac{\beta_{vh}\eta_hS^{0}_v}{N^{0}_h}&-\dfrac{\beta_{vh}S^{0}_v}{N^{0}_h}
	&0&-k_8&0&0&0&0&\theta\\
	0&\dfrac{\beta_{vh}\eta_hS^{0}_v}{N^{0}_h}&\dfrac{\beta_{vh}S^{0}_v}{N^{0}_h}&0&0&-k_9&0&0&0&0\\
	0&0&0&0&0&\gamma_v&-k_8&0&0&0\\
	0&0&0&0&	K_1&K_1&K_1&-K_2&0&0\\
	0&0&0&0&0&0&0&K_3&-K_4&0\\
	0&0&0&0&0&0&0&0&l&-k_7
	\end{array} \right),
	$
}

with $K_1=\mu_b\left(1-\dfrac{E^{*}}{K_E}\right)$, $K_2=k_5+\dfrac{\mu_b}{K_E}S^{0}_v$. $K_3=s\left(1-\dfrac{L^{*}}{K_L}\right)$, and 
$K_4=\left(k_6+\dfrac{sE^{*}}{K_L}\right)$.

Note that the system~\eqref{BasicModel}, with $\beta_{hv}=\beta^{*}_{hv}$, has a hyperbolic equilibrium point, i.e., the linearised system \eqref{BasicModel} has a simple eigenvalue with zero real part and all other eigenvalues have negative real part (this follows from the loss of stability of the disease--free equilibrium, $\mathcal E_1$ through the transcritical bifurcation). Hence, the centre manifold theory~\cite{vawa02,duhucc,ccso,guho,Carr} can be used to analyse the dynamics of the model \eqref{AR3} near $\beta_{hv}=\beta^{*}_{hv}$. The technique in Castillo-Chavez and Song (2004)~\cite{ccso} entails finding the left and right eigenvectors of the linearised system above as follows.

The left eigenvector components of $J_{\beta^{*}_{hv}}$, which correspond to the uninfected states are zero (see Lemma 3 in \cite{vawa02}). Thus a non-zero components correspond to the infected states. It follows that the matrix $J_{\beta^{*}_{hv}}$ has a left eigenvector given by ${\bf v} = (v_1,v_2,\hdots,v_{10} )$, where 
\begin{equation}
\label{LeftVectorBasicModel}
\begin{array}{l}
v_1=v_4=v_5=v_8=v_9=v_{10}=0;\,\, v_2=\dfrac{k_8}{\beta^{*}_{hv}}v_7;\,\,
v_3=\dfrac{\beta_{vh}S^{0}_v(\eta_vk_8+\gamma_v)}{k_4k_9N^{0}_h}v_7;\\
v_6=\dfrac{(\eta_vk_8+\gamma_v)}{k_9}v_7,\,\,
v_7>0.
\end{array} 
\end{equation}

Similarly, the component of the right eigenvector ${\bf w}$ 
are given by
\begin{equation}
\label{RightVectorBasicModel}
\begin{array}{l}
w_{7}>0,\;\; w_{10}>0,\\
w_1=-\dfrac{\beta^{*}_{hv}(\eta_vk_8+\gamma_v)}{\gamma_vk_1}w_7;\,\,
w_4=\dfrac{\beta^{*}_{hv}\gamma_h\sigma(\eta_vk_8+\gamma_v)}{\mu_h\gamma_vk_3k_4}w_7;\,\,
w_2=\dfrac{\mu_hk_4}{\gamma_h\sigma}w_4;\,\,w_3=\dfrac{\mu_h}{\sigma}w_4;\\
w_5=-\dfrac{k_8+\gamma_v}{\gamma_v}w_7+\dfrac{k_7K_2K_4}{lK_1K_3}w_{10};\,\, w_6=\dfrac{k_8}{\gamma_v}w_7;\,\,w_8=\dfrac{k_7K_4}{lK_3}w_{10};\,\,w_9=\dfrac{k_7}{l}w_{10}.
\end{array} 
\end{equation}

Theorem 4.1 in Castillo-Chavez and Song \cite{ccso} is then applied to establish the existence of backward bifurcation in~\eqref{BasicModel}. To apply such a theorem, it is convenient to let $f_k$ represent the right-hand side of the $k^{th}$ equation of the system \eqref{AR3} and let $x_k$ be the state variables whose derivative is given by the $k^{th}$ equation for $k = 1,\hdots,10$. The local bifurcation analysis near the bifurcation point ($\beta_{hv}=\beta^{*}_{hv}$) is then determined by the signs of two associated constants, denoted by $\mathcal{A}_1$ and $\mathcal{A}_2$, defined by
\begin{equation}
\label{ccsBasicModel}
\mathcal{A}_1=\sum\limits_{k,i,j=1}^{10}v_kw_iw_j\dfrac{\partial^{2}f_k(0,0)}{\partial x_i\partial x_j}\qquad and\qquad
\mathcal{A}_2=\sum\limits_{k,i=1}^{10}v_kw_i\dfrac{\partial^{2}f_k(0,0)}{\partial x_i\partial \phi}
\end{equation}
with $\phi=\beta_{hv}-\beta^{*}_{hv}$. It is important to note that in $f_k(0,0)$, the first zero corresponds to the disease--free equilibrium, $\mathcal{E}_{1}$, for the system \eqref{BasicModel}. Since $\beta_{hv}=\beta^{*}_{hv}$ is the bifurcation parameter, it follows from $\phi=\beta_{hv}-\beta^{*}_{hv}$ that $\phi=0$ when $\beta_{hv}=\beta^{*}_{hv}$ which is the second component in $f_k(0,0)$.

Using Eqs.~\eqref{LeftVectorBasicModel} and \eqref{RightVectorBasicModel} in Eq.~\eqref{ccsBasicModel}, we obtain	
{\footnotesize
	\begin{equation}
	\label{ccsBasicModel1}
	\mathcal{A}_1=v_2\sum\limits_{i,j=1}^{10}w_iw_j\dfrac{\partial^{2}f_2(0,0)}{\partial x_i\partial x_j}
	+v_6\sum\limits_{i,j=1}^{10}w_iw_j\dfrac{\partial^{2}f_6(0,0)}{\partial x_i\partial x_j}\quad and \quad
	\mathcal{A}_2=v_2\sum\limits_{i=1}^{10}w_i\dfrac{\partial^{2}f_2(0,0)}{\partial x_i\partial \phi}
	\end{equation}
}
It follows then, after some algebraic computations (see the details in appendix~\ref{ccBasicModeldetails}), that	
\begin{equation}
\label{ccsBasicModel2}
\begin{split}
\mathcal{A}_1&=\zeta_1-\zeta_2, \\ 
\end{split}
\end{equation}

where we have set (see the  appendix~\ref{ccBasicModeldetails} for details on derivation of this quantity)
{\footnotesize
	\begin{equation}
	\label{}
	\begin{array}{l}
	\zeta_1=\left\lbrace 2\dfrac{k_7K_2K_4}{lK_1K_3}\dfrac{a\beta_{vh}}{N^{0}_h}\left( \eta_hw_2+w_3\right)w_{10} 
	-2\dfrac{a\beta_{vh}S^{0}_v}{(N^{0}_h)^{2}}\left( \eta_hw_2+w_3\right)w_1\right\rbrace v_6\\
	\begin{split}
	\zeta_2&=2\dfrac{a\beta^{*}_{hv}}{N^{0}_h}\left(\eta_vw_6+w_7\right)(w_2+w_3+w_4) v_2\\ 
&+2\dfrac{a\beta_{vh}}{N^{0}_h}
\left\lbrace \dfrac{S^{0}_v}{N^{0}_h}\left( \eta_hw^{2}_2
	+(\eta_h+1)w_2w_3+\eta_hw_2w_4+\eta_hw^{2}_3+w_3w_4\right) +\dfrac{(k_8+\gamma_v)}{\gamma_v}\left( \eta_hw_2+w_3\right)w_7\right\rbrace v_6\\
	\end{split}
	\end{array}
	\end{equation}
}
According to \eqref{LeftVectorBasicModel} and \eqref{RightVectorBasicModel}, we have $\zeta_1>0$ and $\zeta_2>0$.

We then have
\[
\begin{split}
\mathcal{A}_2&=\dfrac{aS^{0}_v}{N^{0}_h}\left(\eta_hw_6+w_7\right)v_2. 
\end{split}
\]
Note that the coefficient $\mathcal{A}_2$ is always positive. Thus, using Theorem 4.1 in \cite{ccso}, the following result is established.
\begin{theorem}
	\label{thbifARopt}
	The basic model \eqref{BasicModel} exhibits a backward bifurcation at $\mathcal R_0 = 1$ whenever $\mathcal{A}_1>0$ (i.e., $\zeta_1>\zeta_2$). If the reversed inequality holds, then the bifurcation at $\mathcal R_0 = 1$ is
	forward.
\end{theorem}

The direct consequence of Theorem~\ref{thbifARopt} is the following.
\begin{corollary}
	If $\mathcal{A}_1<0$ (i.e., $\zeta_1<\zeta_2$), then the unique endemic equilibrium point
of the basic model~\eqref{BasicModel} is locally asymptotically stable whenever $\mathcal R_0 > 1$.
\end{corollary}

The backward bifurcation phenomenon is illustrated by numerical simulation of the model with the following set of parameter values (it should be noted that these parameters are chosen for illustrative purpose only, and may not necessarily be realistic epidemiologically): $\Lambda_h=30$, $\beta_{hv}=0.008$, $\eta_h=0.78$, $\eta_v=0.99$, $\delta=1$, $\sigma=0.01428$, $\beta_{vh}=0.5$, $\gamma_v=1/14$, $\Gamma_E=10^4$, $\Gamma_L=\Gamma_E/2$. All other parameters are as in Table \ref{octab2}.  In this case the conditions required by Theorem~\ref{EEBasic}, case (iii), are satisfied, as well as $\zeta_1=1.0772\times 10^{-7}>\zeta_2=1.0250\times 10^{-9}$ (so $\mathcal A_1=1.0669\times 10^{-7}>0$) in Theorem~\ref{thbifARopt}. Note, in particular, that
with this set of parameters, $\mathcal R_{c}=0.0367<1$, $\mathcal R_{0}=0.4359<1$ (so that $\mathcal R_{c}<\mathcal R_{0}<1$). It follows: $d_2=-5.6537\times 10^{-8}<0$, $d_1=1.1504\times 10^{-10}>0$ and $d_0=-2.4857\times 10^{-14}<0$, so that $d^{2}_1-4d_2d_0=7.6134\times 10^{-21}>0$. The resulting two endemic equilibria 
$\mathcal{E}_{2}=(S^{*}_h,E^{*}_h,I^{*}_h,R^{*}_h,S^{*}_v,E^{*}_v,I^{*}_v,E,L,P)$, are:\\
$\mathcal{E}^{*}_{2}=(16394,16384,29, 10092,6558,406,869,8393,31334,3264)$,\\
which is locally stable and\\
$\mathcal{E}^{**}_{2}=(104660,104660,25, 8850,7530,97,206,8393,31334,3264)$,\\
which is unstable.

The associated bifurcation diagram is depicted in figure~\ref{backwardar3Basic}. This clearly shows the co-existence of two locally-asymptotically stable equilibria when $\mathcal R_{0}<1$, confirming that the model~\eqref{BasicModel} undergoes the phenomenon of backward bifurcation.  
\begin{figure}[h!]
	\begin{center}
		\includegraphics[scale=0.30]{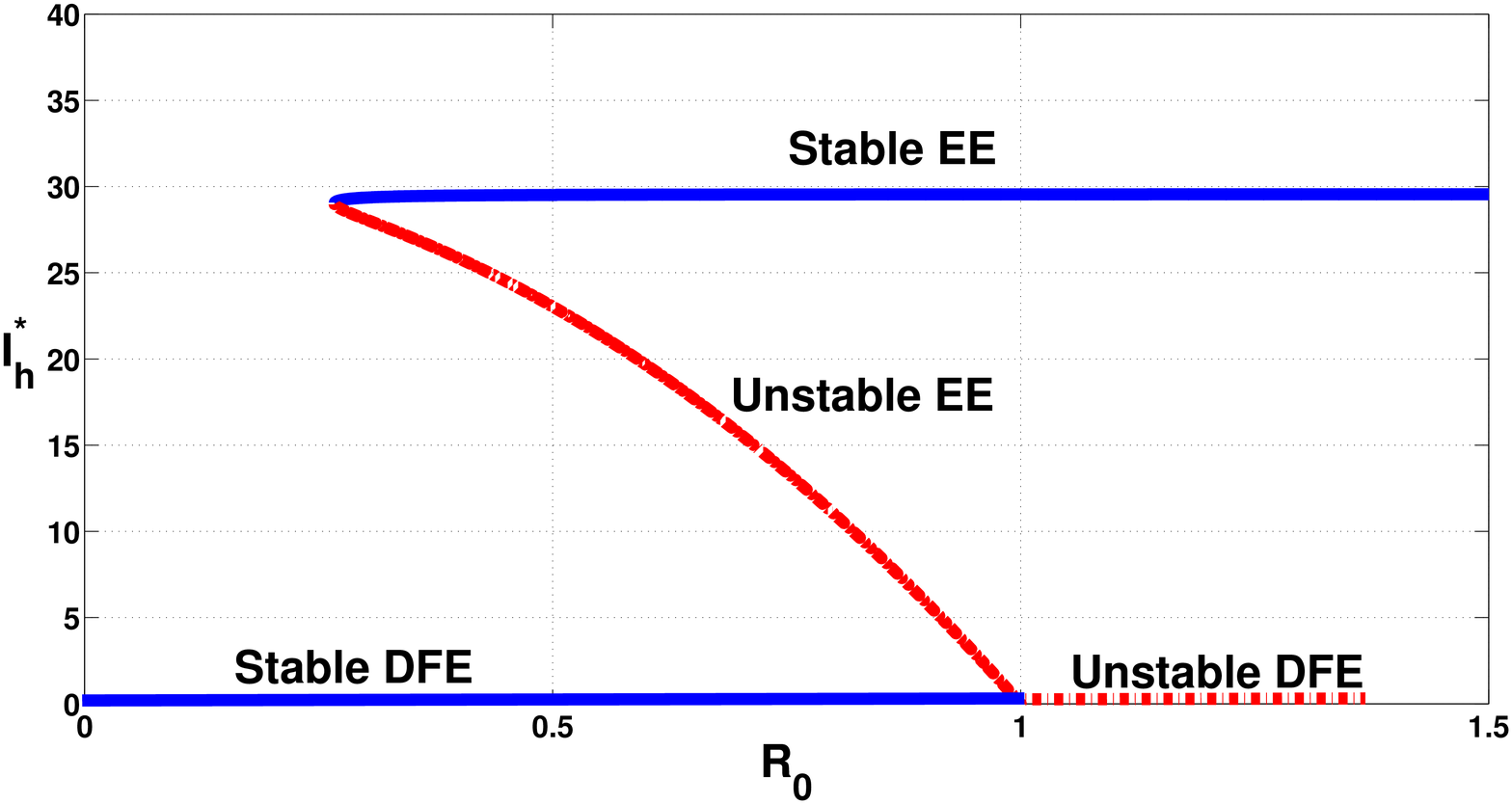}
		\includegraphics[scale=0.30]{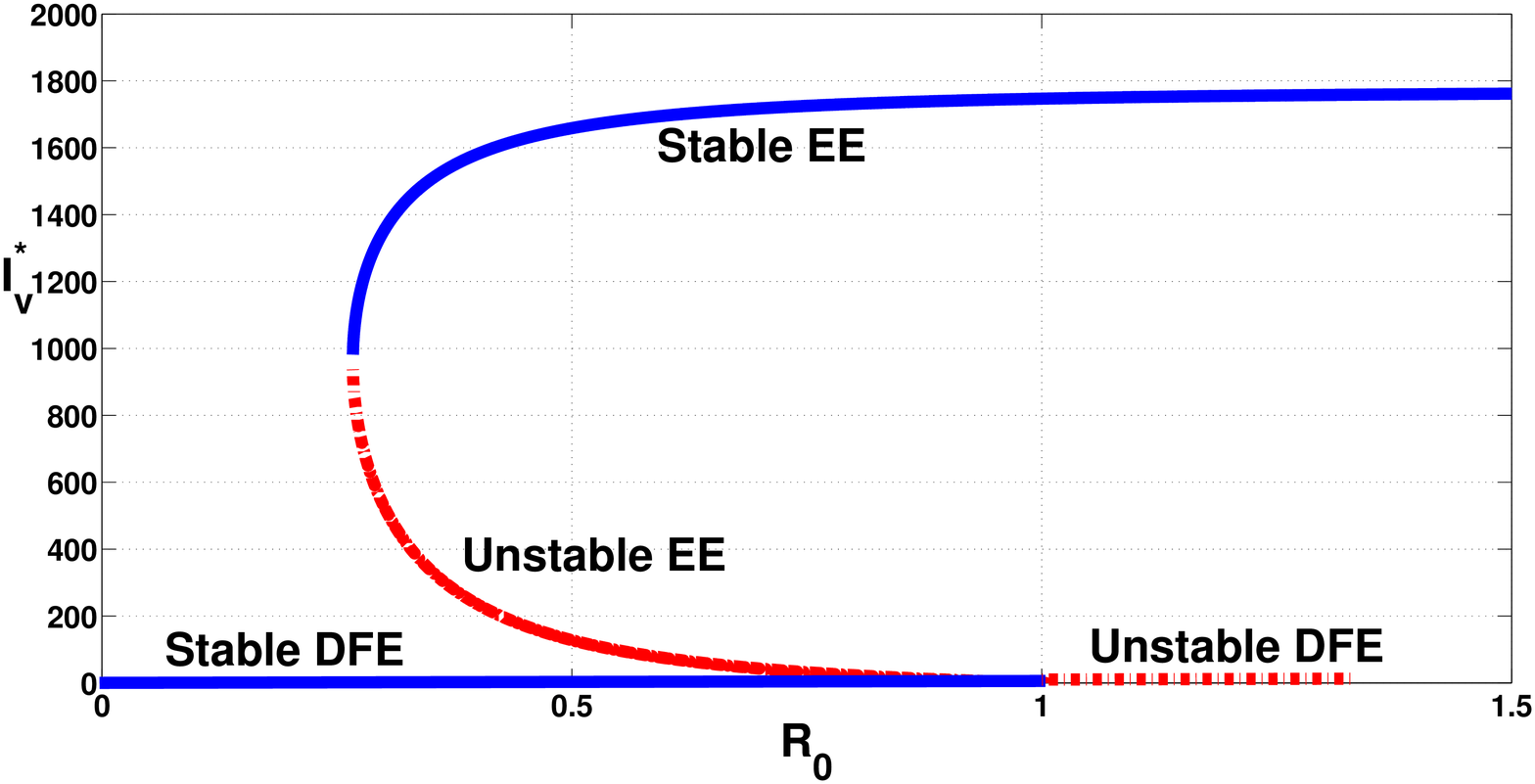}
		\caption{The backward bifurcation curves for model system~\eqref{BasicModel} in the $(\mathcal R_{0}, I^{*}_{h})$ and $(\mathcal R_{0}, I^{*}_{v})$ planes. The parameter $\beta_{hv}$ varied in the range [0, 0.0877] to allow $\mathcal R_{0}$ to vary in the range [0, 1.5]. Two endemic equilibrium points coexist for values of $\mathcal R_{0}$ in the range (0.2671, 1) (corresponding to the range (0.0028, 0.0390) of $\beta_{hv}$). The notation EE and DFE stand for endemic equilibrium and disease--free equilibrium, respectively.  Solid lines represent stable equilibria and dash lines stand for  unstable equilibria. \label{backwardar3Basic}}
	\end{center}
\end{figure}

The occurrence of the backward bifurcation can be also seen in Figure~\ref{IhIvBackwardplot}. Here, $\mathcal R_0$ is less than the transcritical bifurcation threshold $\mathcal R_0=1$ ($\mathcal R_0=0.4359 < 1$), but the solution of the model~\eqref{BasicModel} can approach either the endemic equilibrium point or the disease-free equilibrium point, depending on the initial condition.
\begin{figure}[h!]
	\begin{center}
		\includegraphics[scale=0.30]{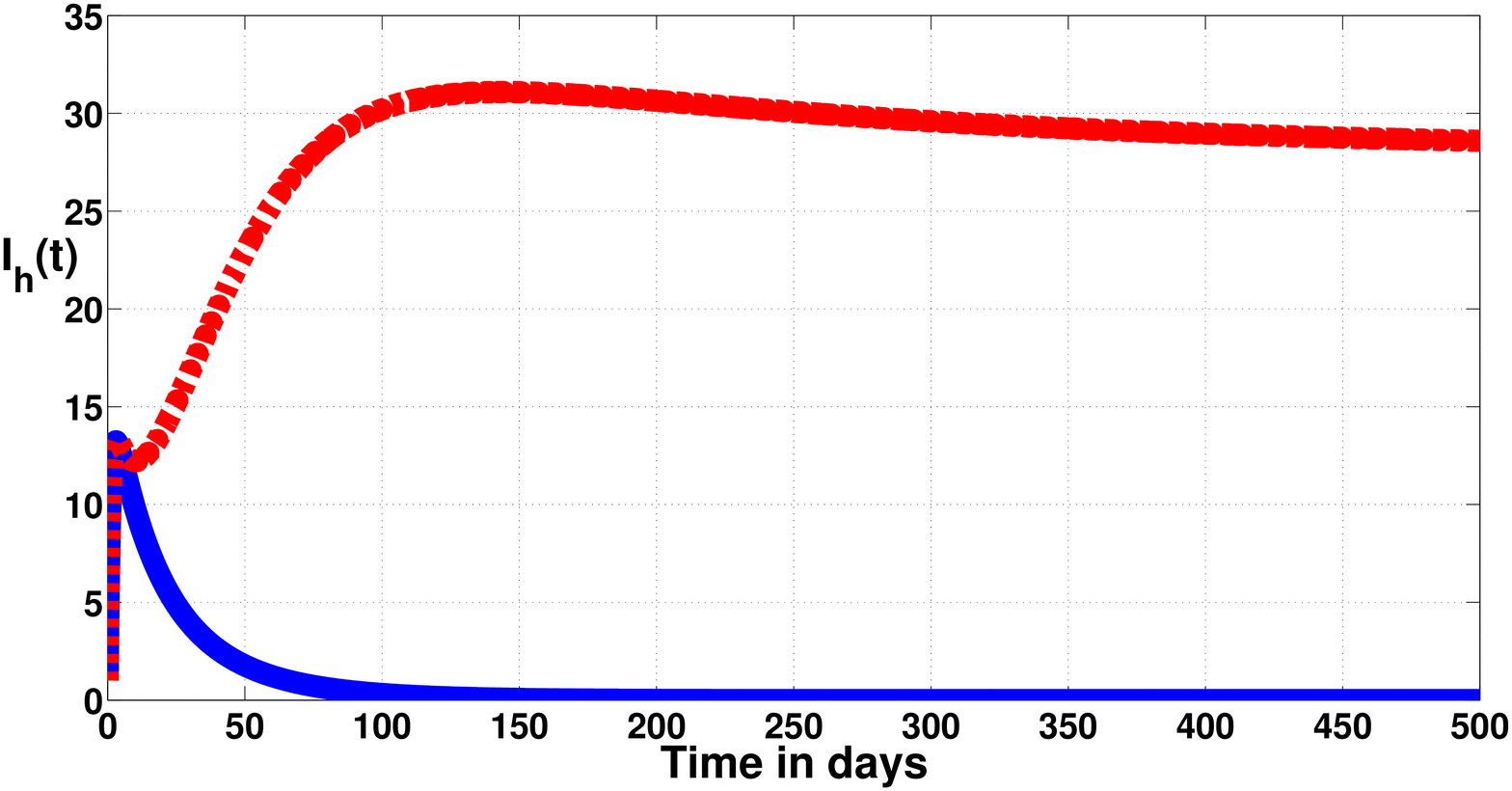}
		\includegraphics[scale=0.30]{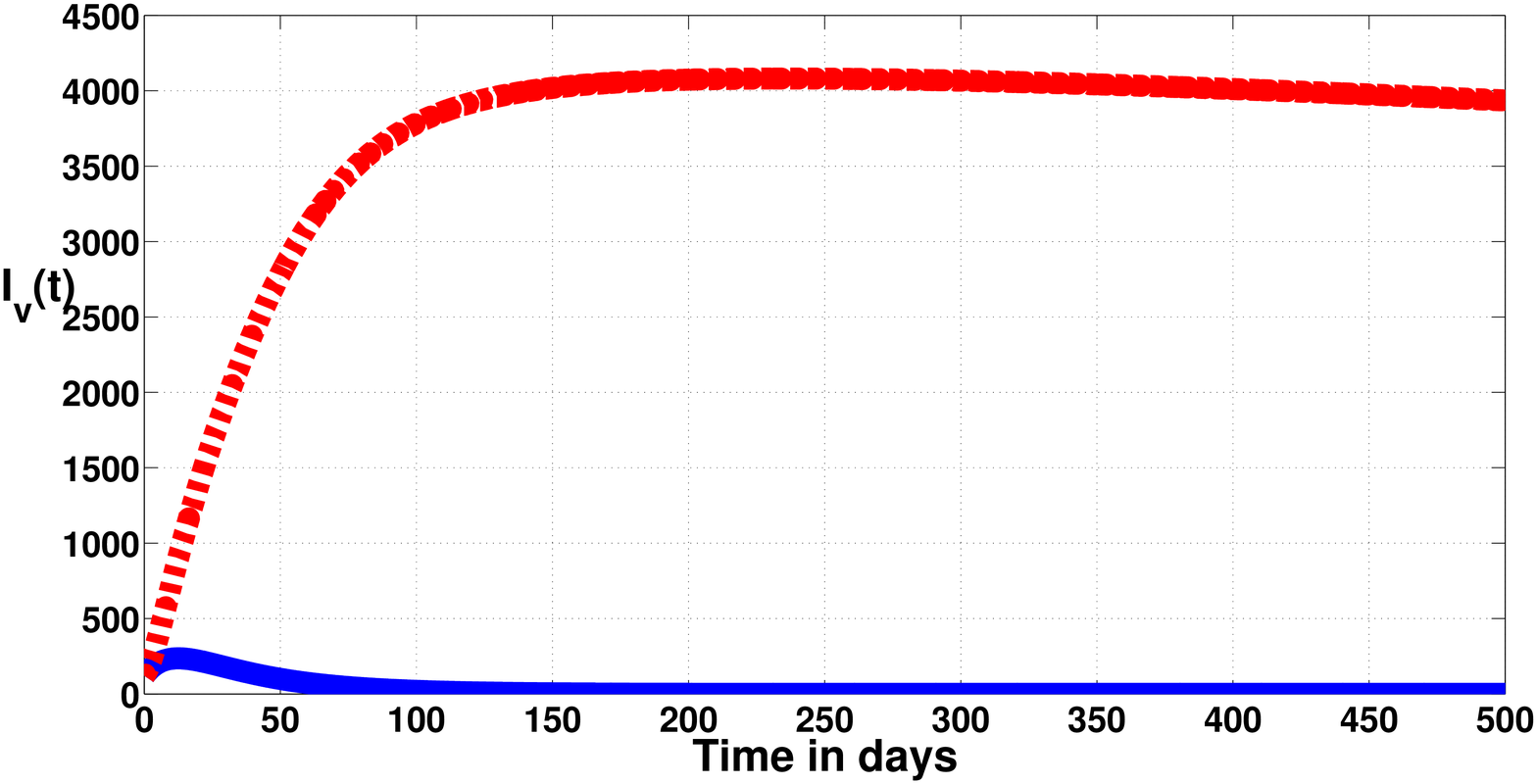}
		\caption{Solutions of model \eqref{BasicModel} of the number of infectious humans, $I_h$, and the number of infectious vectors, $I_v$, for parameter values given in the bifurcation diagram in Figure~\ref{backwardar3Basic} with $\beta_{hv}=0.008$, so $\mathcal R_0=0.4359<1$, for two different set of initial conditions. The first set of initial conditions (corresponding to the dotted trajectory) is $S_h=700$, $E_h=220$, $I_h=15$, $R_h=60$, $S_v=3000$, $E_v=400$, $I_v=120$, $E=10000$, $L=5000$ and $P=3000$. The second set of initial conditions (corresponding to the solid trajectory) is $S_h=733650$, $E_h=220$, $I_h=15$, $R_h=60$, $S_v=3000$, $E_v=400$, $I_v=120$, $E=10000$, $L=5000$ and $P=3000$. The solution for initial condition~1 approaches the locally asymptotically stable DFE point, while the solution for initial condition~2 approaches the locally asymptotically stable endemic equilibrium .} \label{IhIvBackwardplot}
	\end{center}
\end{figure}

We know from Theorem~\ref{thbifARopt} that a backward bifurcation scenario is possible for model (\ref{BasicModel}). Here, we characterize the critical value in terms of a single
parameter, the transmission rate $\beta_{hv}$, at which the saddle--node bifurcation occurs, i.e.  the threshold for the appearance of two endemic equilibria (see Figure \ref{backwardar3Basic}). 

We follow the approach given in~\cite{saetal}.  Introducing the quantities,
\begin{equation}
\label{betacrit1Basic1}
\overline{\beta}=\dfrac{\left[a\mu_h\beta_{vh}k_{10}+2k_2k_8\right]N^{0}_h }{a^2\beta_{vh}k_{10}k_{11}N^{0}_v},
\end{equation}
and,
\begin{equation}
\label{beta+-ar3opc}
\beta_{\pm}=\dfrac{N^{0}_{h}}{k_{3}k_{4}k_{10}k_{11}a^2N^{0}_{v}\beta_{vh}}
\left\lbrace \sqrt{\delta\gamma_h(a\mu_{h}\beta_{vh}k_{10}+k_{2}k_{8})}
\pm\sqrt{(-k_{2}\psi)}\right\rbrace^{2},
\end{equation}
we have the following result (see the Appendix~\ref{sadleAr3opc} for the proof).
\begin{theorem}
	\label{sdlenodeBasic}
	Assume that $\psi<0$, where $\psi$ is given by \eqref{psiar3opc}. Then the backward bifurcation phenomenon takes place in the basic model~\eqref{BasicModel} if and only if 
	\begin{equation}
	\label{Betasn2Basic}
	\bar{\beta}_{hv}<\beta_{hv}<\min(\beta_{-},\beta^{*}_{hv}) 
	\quad or\quad \max(\bar{\beta}_{hv},\beta_{+})<\beta_{hv}<\beta^{*}_{hv}.
	\end{equation} 
\end{theorem}

The previous analysis is in line with the observation made by Wangari et \emph{al.} in \cite{Wangari2015} concerning the bifurcation thresholds of epidemiological models.
\subsection{Non-existence of endemic equilibria for $\mathcal R_{0}<1$ and $\delta=0$}
In this case, we have the following result.
\begin{theorem}
	\label{BasicdeltaNULL}
	\noindent(i) The model~\eqref{BasicModel} without disease--induced death ($\delta=0$) has no endemic equilibrium  when $\mathcal R_{0,\delta=0}\leq 1$, and has a unique endemic equilibrium otherwise.\\
	\noindent(ii) The DFE, $\mathcal{E}_{1}$, of model~\eqref{BasicModel} without disease--induced death ($\delta=0$), is globally asymptotically stable (GAS) in $\mathcal{D}$ if $\mathcal R_{0,\delta=0}<1$.
\end{theorem}
\begin{proof}
	See appendix~\ref{appBasic4}.
\end{proof}

\subsection{Sensitivity analysis}
We carried out the sensitivity analysis to determine the model robustness to parameter values. That is to help us to know the parameters that are most influential in determining disease dynamics. Following the approach by Marino \emph{et al.}~\cite{Marino2008} and Wu \emph{et al.}~\cite{Wu2013}, partial rank correlation coefficients (PRCC) between the basic reproduction number $\mathcal R_0$ and each parameter are derived from 5,000 runs of the Latin hypercube sampling (LHS) method, which is a stratified Monte Carlo sampling method that divides each parameter's range into N equal intervals and randomly draws one sample from each interval~\cite{Wu2013,Stein1987}. The parameters are assumed to be random variables with uniform distributions with their mean value listed in Table~\ref{octab2}.

With these 5,000 runs of LHS, the derived distribution of $\mathcal R_0$ is given in Figure~\ref{histogramofR0opcar3opc}. This sampling shows that the mean of $\mathcal R_0$ is 1.9583 and the standard deviation is 1.8439. This implies that for the mean of parameter values given in Table~\ref{octab2}, we may be confident that the model predicts an endemic state, since the basic reproduction number is greater that unity. The probability that $\mathcal R_0>1$ (the disease--free equilibrium is unstable and there is exactly one endemic equilibrium point) is 64.48\%. 
\begin{figure}[H]
	\begin{center}
		\includegraphics[width=\textwidth]{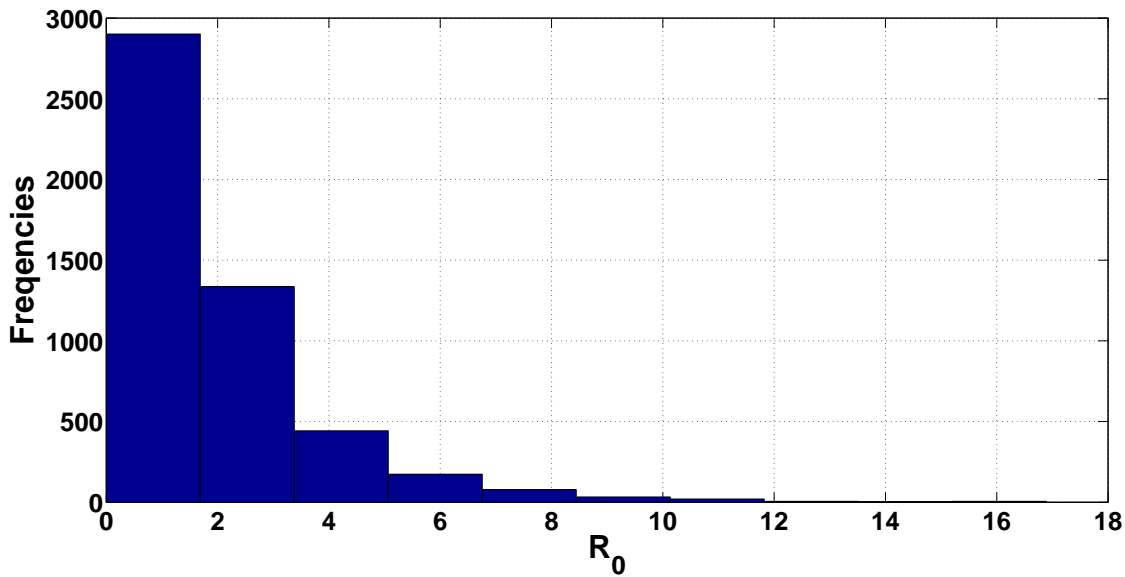}
		\caption{Sampling distribution of $\mathcal R_0$ from 5,000 runs of Latin hypercube sampling. The mean of $\mathcal R_0$ is 1.9583 and the standard deviation is 1.8439. Furthermore, $\mathbb{P}(\mathcal R_0\geq 1)=64.48\%$.  \label{histogramofR0opcar3opc}}
	\end{center}
\end{figure}

We also evaluate the probabilities that conditions (i), (ii) and (iii) in Theorem~\ref{EEBasic} are satisfied. Let us set $\mathbb{P}[X]$ the probability of $X$, 
and the sets of parameter values for which $(\mathcal N>1)$ is true by $\Phi_{1}$, the sets of parameter values for which $(\mathcal R_c<\mathcal R_0<\min(1,\mathcal R_{1b}))$ or $(\max(\mathcal R_c,\mathcal R_{2b})<\mathcal R_{0}<1)$ by $\Phi_2$, and the sets of parameter values for which $\left\lbrace (\mathcal R_0 = R_{1b})\,\,or\,\,(\mathcal R_0 = R_{2b})\right\rbrace$ by $\Phi_3$,

\begin{subequations}
	\begin{align}
	\label{Tear3ccopc}
	\mathbb{P}\left[\neg\Phi_1\right]=\mathbb{P}\left[\mathcal N\leq 1\right] &=0.0020 , \\
	\label{Tear4ccopc}
	\mathbb{P}\left[\Phi_1\right]=\mathbb{P}\left[\mathcal N> 1\right] &=0.9980 , \\
	\label{P0EE12ar3ccopc}
	\mathbb{P}\left[\Phi_1\,\,\text{and}\,\,(\mathcal R_0<1)\,\,\text{and}\,\, \Phi_2\,\, \right] &=0.0026,\\
	\label{P0EE13ar3ccopc}
	\mathbb{P}\left[\Phi_1\,\,\,\text{and} \quad(\mathcal R_0 < 1)\,\,\, \text{and} 
	\,\,\, \Phi_3 \right] &=0 , \\
	\label{P0EE14ar3ccopc}
	\mathbb{P}\left[\Phi_1\,\,\,\text{and} (\mathcal R_0<1)\,\,\,\text{and}\,\,\,\neg\Phi_2\,\,\, \text{and} 
	\,\,\, \neg\Phi_3 \right] &=0.3526 , \\
	\label{P0EE2ar3ccopc}
	\mathbb{P}\left[\Phi_1\quad\text{and} \quad(\mathcal R_0 < 1)\right] &=0.3552 , \\
	\label{P0EE3ar3ccopc}
	\mathbb{P}\left[\Phi_1\quad\text{and} \quad(\mathcal R_0 \geq 1)\right] &=0.6448.
		\end{align}
\end{subequations}

Therefore, the probability that the trivial disease--free equilibrium is locally asymptotically stable
is 0.0020 (from \eqref{Tear3ccopc}), the probability that the disease free equilibrium point is locally asymptotically stable is 0.3552 (from \eqref{P0EE2ar3ccopc}), the probability that the disease free equilibrium point is locally asymptotically stable and (i) there are no endemic equilibrium points is 35.26\%
(ii) there are exactly one endemic equilibrium point is 0 (from \eqref{P0EE13ar3ccopc}), (iii) there are exactly two endemic equilibrium points is 0.0026 (from \eqref{P0EE12ar3ccopc}). This implies that for the ranges of parameter values given in Table~\ref{vaueR0ar3opc}, the disease-free equilibrium point is likely to be locally asymptotically stable and the probability of co-existence of a locally asymptotically stable endemic equilibrium  point (occurrence of backward bifurcation phenomenon) is  very small and insignificant.

We now use sensitivity analysis to analyse the influence of each parameter on the basic reproductive number. From the previously sampled parameter values, we compute the PRCC between $\mathcal R_0$ and each parameter of model (\ref{AR3}). The parameters with large PRCC values ($>0.5$ or $<-0.5$) statistically have the most influence \cite{Wu2013}. 

\begin{table}[H]
	\begin{center}
		\caption{Partial Rank Correlation Coefficients between $\mathcal R_0$ and each parameters of model (\ref{BasicModel}).\label{prccR0ar3opc}}
		\begin{tabular}{lcccccc}
			\hline
			Parameter&Correlation &Parameter&Correlation &Parameter&Correlation \\
			&Coefficients&&Coefficients&&Coefficients\\
			\hline
		$\beta_{vh}$&0.7345&$\Gamma_L$&0.3813&$\mu_h$&0.1717\\
		$\beta_{hv}$&0.7285&$\Gamma_E$&0.3698&$\gamma_v$&0.0872\\
		$a$&0.6454&$s$&0.3733&$\eta_v$&0.06891\\
		$\theta$&0.6187&$\mu_P$&\textbf{--}0.2565&$\mu_L$&\textbf{--}0.0556\\
		$\mu_v$&\textbf{--}0.5521&$\eta_h$&0.2331&$\mu_b$&0.0531\\
		$\Lambda_h$&\textbf{--}0.5435&$\gamma_h$&\textbf{--}0.2204&$\mu_E$&\textbf{--}0.0022\\
		$l$&0.5144&$\sigma$&\textbf{--}0.1867&$\delta$&0.0003\\
			\hline
		\end{tabular}
	\end{center}
\end{table}

\begin{figure}[H]
	\begin{center}
				\includegraphics[width=\textwidth]{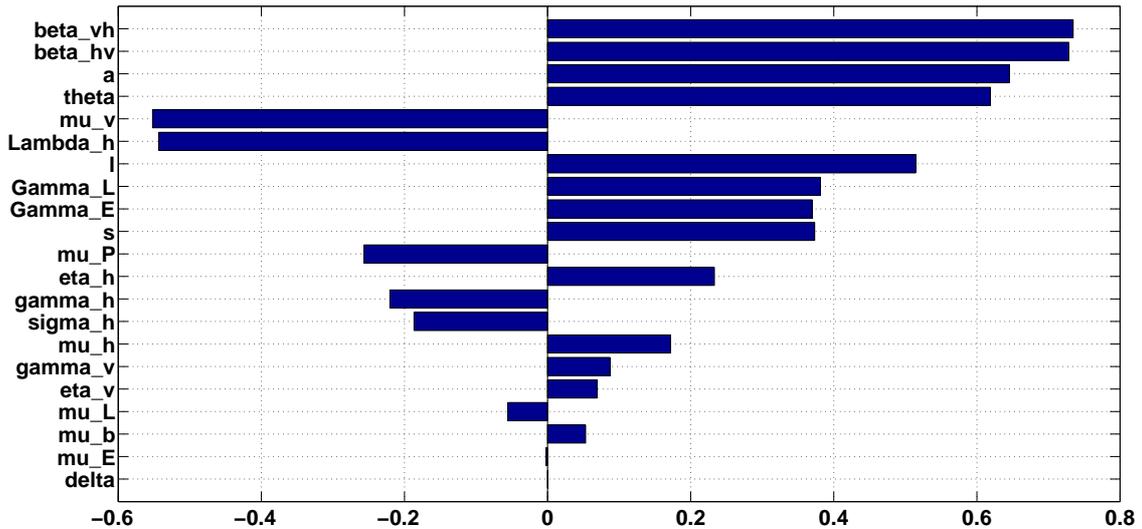}
				\label{F:PRCCR0opc2}
		\caption{Partial rank correlation coefficients for $\mathcal R_0$} \label{F:SAresultsar3opc2}
	\end{center}
\end{figure}

The results, displayed in Table \ref{prccR0ar3opc} and Figure \ref{F:SAresultsar3opc2}, show that the
basic reproduction number $\mathcal R_0$ outcome measures are sensitive to changes in the parameters
$\beta_{vh}$, $\beta_{hv}$, $a$, $\theta$, $\mu_v$, $\Lambda_h$ and $l$. 

The sensitivity results suggest that the control of the epidemic of arboviral diseases pass through a combination of immunization against arbovirus, individual protection against vector bites, treatment of infected human, and vector control mechanisms. 

\section{A Model for Optimal Control}
\label{sec2ar3opc}

There are several possible interventions in order to reduce or limit the proliferation of mosquitoes and the explosion of the number of infected humans and mosquitoes. 
In addition of controls used in \cite{moaaHee2012}, we add vaccination and the control of adult vectors as control variables to reduce or even eradicate the disease. So we introduce five time dependent controls:
\begin{itemize}
	\item[1.] The first control $0\leq u_1(t)\leq 1$ denotes the percentage of susceptible individuals that one decides to vaccinate at time t. A parameter $\omega$ associated to the control $u_1(t)$ represents the waning immunity process \cite{HelenaSofiaRodrigues2014}.
	\item[2.] The second control $0\leq u_2(t)\leq 1$ represents efforts made to protect human from mosquito bites. It mainly consists  to the use of mosquito nets or wearing appropriate clothes \cite{moaaHee2012}. Thus we modify the infection term as follows: 
	\begin{equation}
	\label{lambdac}
	\lambda^{c}_h=(1-\alpha_1 u_2(t))\lambda_h,\;\;\;, \lambda^{c}_v=(1-\alpha_1 u_2(t))\lambda_v
	\end{equation}
	where $\alpha_1$ measures the effectiveness of the  prevention measurements against mosquito bites. 
	\item[3.] The third control $0\leq u_3(t)\leq 1$ represents efforts made for treatment. It mainly consists in isolating infected patients in hospitals, installing an anti-mosquito electric diffuser in the hospital room, or symptomatic treatments~\cite{moaaHee2012}. Thus we modify the recovery rate such that $\sigma^{c}_{h}:=\sigma_{h}+\alpha_2u_3$. $\alpha_2$ is the effectiveness of the anti-arboviral diseases drugs with $\alpha_2=0.3$ \cite{moaaHee2012}. Note that this control also permit to reduce the disease-induced death.
	\item[4.] The fourth control $0\leq u_4(t)\leq 1$ represents  mosquitoes adulticiding effort with killing efficacy $c_m$. Thus the mosquito natural mortality rate becomes $\mu^{c}_v=\mu_v+c_mu_4(t)$.
	\item[5.] The fifth control $0\leq u_5(t)\leq 1$ represents the effect of interventions used for the vector control. It mainly consists in the reduction of breeding sites with chemical application methods, for instance using larvicides like BTI (\emph{Bacillus Thuringensis Israelensis}) which is a biological larvicide, or by introducing larvivore fish. This control focuses on the reduction of the number of larvae, and thus eggs, of any natural or artificial water-filled container \cite{moaaHee2012}. Thus the eggs and Larvae natural mortality rate become $\mu^{c}_E=\mu_E+\eta_1u_5(t)$ and $\mu^{c}_L=\mu_L+\eta_2u_5(t)$ where $\eta_1$, $\eta_2$, represent the chemical eggs and larvae mortality rate, respectively~\cite{moaaHee2012}.
\end{itemize}

Note that $0\leq u_i\leq 1$, for $i = 1,\hdots,5$, means that when the control is zero there is no any effort invested (i.e. no control) and when it is one, the maximum control effort is invested. 

Therefore, our optimal control model of arboviral diseases reads as
\begin{equation}
\label{AR3optimal}
\left\lbrace \begin{array}{ll}
\dot{S}_h&=\Lambda_h-\left[ (1-\alpha_1u_2(t))\lambda_h+\mu_h +u_1(t)\right] S_h +\omega u_1(t) R_h\\
\dot{E}_h&=(1-\alpha_1 u_2(t)) \lambda_hS_h-(\mu_h+\gamma_h)E_h\\
\dot{I}_h&=\gamma_hE_h-\left[\mu_h+(1-\alpha_2u_3(t))\delta+\sigma+\alpha_2u_3(t)\right] I_h\\
\dot{R}_h&=(\sigma+\alpha_2u_3(t))I_h+u_1S_h-(\mu_h+\omega u_1)R_h\\
\dot{S}_v&=\theta P-(1-\alpha_1 u_2(t))\lambda_vS_v-(\mu_v+c_mu_4(t))S_v\\
\dot{E}_v&=(1-\alpha_1 u_2(t))\lambda_vS_v-(\mu_v+\gamma_v+c_mu_4(t))E_v\\
\dot{I}_v&=\gamma_vE_v-(\mu_v+c_mu_4(t))I_v\\
\dot{E}&=\mu_b\left(1-\dfrac{E}{\Gamma_E} \right)(S_v+E_v+I_v)-(s+\mu_E+\eta_1u_5(t))E\\
\dot{L}&=sE\left(1-\dfrac{L}{\Gamma_L} \right)-(l+\mu_L+\eta_2u_5(t))L\\
\dot{P}&=lL-(\theta+\mu_P)P\\
\end{array}\right.
\end{equation}
with initial conditions given at $t =0$. 

For the non-autonomous system~\eqref{AR3optimal}, the rate of change of the total populations of humans and adults vectors is given, respectively, by
\begin{equation}
\label{ratePopulOpcFull}
\left\lbrace \begin{array}{ll}
\dot{N}_h&=\Lambda_h-\mu_hN_h-(1-\alpha_2u_3(t))\delta I_h \\
\dot{N}_v&=\theta P-(\mu_v+c_mu_4(t))N_v \\
\end{array}\right.
\end{equation}
For bounded Lebesgue measurable controls and non-negative initial conditions, non-negative bounded solutions to the state system exist~\cite{Lukes}.


The objective of control is to minimize: the number of symptomatic humans infected with arboviruses (that is, to reduce sub-population $I_h$ ), the number of vector ($N_v$ ) and the number of eggs and larvae (that is, to reduce sub-population $E$ and $L$, respectively), while keeping the costs of the control as low as possible. 

To achieve this objective we must incorporate the relative costs associated with each policy (control) or combination of policies directed towards controlling the spread of arboviral diseases. We define the objective function as 
\begin{equation}
\label{OCF}
\begin{split}
&J(u_1,u_2,u_3,u_4,u_5)
=\int^{t_{f}}_{0}\left[D_1I_h(t)+D_2N_v(t)+D_3E(t)+D_4L(t)+\sum\limits^{5}_{i=1}B_iu^{2}_{i}(t) \right]dt 
\end{split} 
\end{equation}
and the control set 
$$
\Delta=\lbrace (u_1,u_2,u_3,u_4,u_5)|u_i(t)\,\text{is Lebesgue measurable on }[0,t_{f}],\,
0\leq u_i(t)\leq 1, i=1,\hdots,5\rbrace. 
$$ 

The first fourth terms in the integrand $J$ represent benefit of $I_h$, $N_v$, $E$ and $L$ populations, describing the comparative importance of the terms in the functional.
A high value of $ D_1$ for example, means that it is more important to reduce the burden of disease as reduce the costs related to all control strategies~\cite{bbMBE2011}.
Positive constants $B_i$, $i=1,\hdots,5$ are weight for vaccination, individual protection (human), treatment and vector control effort respectively, which regularize the optimal control. In line with the authors of some studies on the optimal control (see \cite{moaaHee2012,HelenaSofiaRodrigues2014,Dias2015,Adams2004,bbMBE2011,Zaman2008,Jung2002,Yusuf2012}), we choose a linear function for the cost on infection, $D_1I_h$, $D_2N_v$, $D_3E$, $D_4L$, and quadratic forms for the cost on the controls $B_1u^{2}_{1}$, $B_2u^{2}_{2}$, $B_3u^{2}_{3}$, $B_4u^{2}_{4}$, and $B_5u^{2}_{5}$. This choice can be justified by the following arguments:
\begin{enumerate}
	\item[(i)] An epidemiological control can be likened to an expenditure of energy, by bringing to the applications of physics in control theory;
	\item[(ii)] In a certain sense, minimize $u_i$ is like minimize $u^{2}_i$, because $u_i>0$, $i=1,\hdots, 5$.
	\item[(iii)] The quadratic controls give rise to controls as feedback law, which is convenient for calculations.
\end{enumerate}

We solve the problem using optimal control theory.
\begin{theorem}
	Let $X=(S_h,E_h,I_h,R_h,S_v,E_v,I_v,E,L,P)$. The following set
	\footnotesize
	\[
	\begin{split}
	\Omega&=\left\lbrace
	X\in\mathbf{R}^{10}:
	N_h\leq\dfrac{\Lambda_h}{\mu_h}; E\leq \Gamma_E; L\leq \Gamma_L;P\leq\dfrac{l\Gamma_L}{k_7};N_v\leq \dfrac{\theta l\Gamma_L}{k_7k_8}\right\rbrace\\
	\end{split}
	\]
	\normalsize
	is positively invariant under system \eqref{AR3optimal}.
\end{theorem}
\begin{proof}
	On the one hand, one can easily see that it is possible to get,
	\begin{equation}
	\label{}
	\left\lbrace  \begin{array}{ll}
	\dot{S}_h&\geq-\left(\lambda_h+\mu_h\right)S_h\\
	\dot{E}_h&\geq -(\mu_h+\gamma_h)E_h\\
	\dot{I}_h&\geq -(\mu_h+\delta+\sigma)I_h\\
	\dot{R}_h&\geq -\mu_h R_h\\
	\dot{E}&\geq -(\dfrac{\mu_b}{K_E}+s+\mu_E+\eta_1)E\\
	\dot{L}&\geq -(\dfrac{s}{K_L}+l+\mu_L+\eta_2)L\\
	\dot{P}&\geq -(\theta+\mu_P+\eta_3)P\\
	\dot{S}_v&\geq -(\lambda_v+\mu_v)S_v\\
	\dot{E}_v&\geq -(\mu_v+\gamma_v)E_v\\
	\dot{I}_v&\geq -\mu_v I_v\\
	\end{array}\right.
	\end{equation}
	for $\left(S_h(0),E_h(0),I_h(0),R_h(0),E(0),A(0),P(0),S_v(0),E_v(0),I_v(0)\right)\geq 0$. Thus, solutions with initial value in $\Omega$ remain nonnegative for all $t \geq 0$. On the other hand, we have
	\begin{equation}
	\label{}
	\left\lbrace \begin{array}{ll}
	\dot{N}_h&\leq \Lambda_h-\mu_h N_h \\
	\dot{N}_v&\leq \theta P-\mu_vN_v\\
	\dot{E}&\leq \mu_b\left(1-\dfrac{E}{K_E} \right)(S_v+E_v+I_v)-(s+\mu_E)E\\
	\dot{L}&\leq sE\left(1-\dfrac{L}{K_L} \right)-(l+\mu_L)L\\
	\dot{P}&\leq lL-(\theta+\mu_P)P\\
	\end{array}\right.
	\end{equation}
	The right hand side of the inequalities correspond to the transmission model without control, and it is easy to show that solutions remain in $\Omega$. Then using Gronwall's inequality, we deduce that solutions of \eqref{AR3optimal} are bounded.\hfill
\end{proof}

\subsection{Existence of an optimal control}

The existence of an optimal control can be obtained by using a result of Fleming and Rishel \cite{Fleming}.
\begin{theorem}
	\label{existenceofOC}
	Consider the control problem with system \eqref{AR3optimal}. \\
	There exists $u^{\star}=\left(u^{\star}_1,u^{\star}_2,u^{\star}_3,u^{\star}_4,u^{\star}_5\right)$ such that
	\[
	\min\limits_{(u_1,u_2,u_3,u_4,u_5)\in\Delta}J(u_1,u_2,u_3,u_4,u_5)
	=J(u^{\star}_1,u^{\star}_2,u^{\star}_3,u^{\star}_4,u^{\star}_5)
	\]
\end{theorem}
\begin{proof}
	To use an existence result, Theorem III.4.1 from \cite{Fleming}, we must check if the following properties are satisfied:
	\begin{itemize}
		\item[1-] the set of controls and corresponding state variables is non empty;
		\item[2-] the control set $\Delta$ is convex and closed;
		\item[3-] the right hand side of the state system is bounded by a linear function in the
		state and control;
		\item[4-] the integrand of the objective functional is convex;
		\item[5-] there exist constants $c_1>0$ , $c_2>0$ , and $\beta>1$ such that the integrand of the
		objective functional is bounded below by $c_1\left( \sum\limits_{i=1}^{5}\vert u_i\vert^{2}\right)^{\frac{\beta}{2}}-c_2$.
	\end{itemize}
	
	In order to verify these properties, we use a result from Lukes \cite{Lukes} to give the existence of solutions for the state system \eqref{AR3optimal} with bounded coefficients, which gives condition 1. Since by definition, the control set $\Delta$ is bounded , then condition 2 is satisfied. The right hand side of the state system \eqref{AR3optimal} satisfies condition 3 since the state solutions are bounded. The integrand of our objective functional is clearly convex on $\Delta$, which gives condition 4. There are $c_1>0$, $c_2>0$ and $\beta>1$ satisfying
	$D_1I_h+D_2N_v+D_3E+D_4L+\sum\limits^{5}_{i=1}B_iu^{2}_{i}
	\geq c_1\left( \sum\limits_{i=1}^{5}\vert u_i\vert^{2}\right)^{\frac{\beta}{2}}-c_2$,
	because the states variables are bounded. Thus condition 5 is satisfied. We conclude that there exists an optimal control $u^{*}=\left(u^{\star}_1,u^{\star}_2,u^{\star}_3,u^{\star}_4,u^{\star}_5\right)$ that minimizes the objective functional $J\left(u_1,u_2,u_3,u_4,u_5\right)$.\hfill
\end{proof}

\subsection{Characterization of an optimal control}

The necessary conditions that an optimal control must satisfy come from the Pontryagin's Maximum Principle (PMP)~\cite{Pontryagin}. This principle converts \eqref{AR3optimal}-\eqref{OCF} into a problem of minimizing
point wise a Hamiltonian $\mathbb H$, with respect to $\left( u_1,u_2,u_3,u_4,u_5\right)$:
\begin{equation}
\label{Hamiltonian}
\begin{split}
\mathbb H&=D_1I_h+D_2N_v+D_3E+D_4L+\sum\limits^{5}_{i=1}B_iu^{2}_{i}\\
&+\lambda_{S_h}\left\lbrace\Lambda_h-\left[ (1-\alpha_1u_2)\lambda_h+\mu_h +u_1\right] S_h +\omega u_1 R_h  \right\rbrace\\
&+\lambda_{E_h}\left\lbrace \left[ 1-\alpha_1 u_2\right] \lambda_hS_h-(\mu_h+\gamma_h)E_h\right\rbrace\\
&+\lambda_{I_h}\left\lbrace \gamma_hE_h-(\mu_h+(1-\alpha_2u_3)\delta+\sigma+\alpha_2u_3)I_h\right\rbrace\\
&+\lambda_{R_h}\left\lbrace (\sigma+\alpha_2u_3)I_h+u_1S_h-(\mu_h+\omega u_1)R_h\right\rbrace\\ 
&+\lambda_{S_v}\left\lbrace \theta P-\left[ 1-\alpha_1 u_2\right] \lambda_vS_v
-(\mu_v+c_mu_4)S_v\right\rbrace\\ 
&+\lambda_{E_v}\left\lbrace \left( 1-\alpha_1 u_2\right) \lambda_vS_v
-(\mu_v+\gamma_v+c_mu_4)E_v\right\rbrace
+\lambda_{I_v}\left\lbrace \gamma_vE_v-(\mu_v+c_mu_4)I_v\right\rbrace\\
&+\lambda_{E}\left\lbrace \mu_b\left(1-\dfrac{E}{\Gamma_E} \right)(S_v+E_v+I_v)-(s+\mu_E+\eta_1u_5)E\right\rbrace\\
&+\lambda_{L}\left\lbrace sE\left(1-\dfrac{L}{\Gamma_L} \right)-(l+\mu_L+\eta_2u_5)L\right\rbrace\\
&+\lambda_{P}\left\lbrace lL-(\theta+\mu_P)P \right\rbrace\\
\end{split}
\end{equation}
where the $\lambda_i$, $i=S_{h},E_{h},I_{h},R_{h},S_{v},E_{v},I_{v},E,L,P$ are the adjoint variables or co-state variables. 
Applying Pontryagin's Maximum Principle \cite{Pontryagin}, we obtain the following result.

\begin{theorem}
	\label{OCresult}
	Given an optimal control $u^{\star}=\left( u^{\star}_1,u^{\star}_2,u^{\star}_3,u^{\star}_4,u^{\star}_5\right)$ and solutions\\
	$\left(S^{\star}_h,E^{\star}_h,I^{\star}_h,R^{\star}_h,S^{\star}_v,E^{\star}_v,I^{\star}_v
	,E^{\star},A^{\star},P^{\star}\right)$ of the corresponding state system \eqref{AR3optimal}, there exist adjoint variables $\Pi=\left(\lambda_{S_{h}},\lambda_{E_{h}},\lambda_{I_{h}},\lambda_{R_{h}},\lambda_{S_{v}},\lambda_{E_{v}},\lambda_{I_{v}},\lambda_{E},\lambda_{L},\lambda_{P}\right)$ satisfying,
	{\footnotesize
		\begin{equation}
		\label{adjoints1}
		\begin{split}
		\dfrac{d\lambda_{S_h}}{dt}&=\mu_h\lambda_{S_h}+u_1(\lambda_{S_h}-\lambda_{R_h})
		+(1-\alpha_1u_2)\lambda_h\left(1-\dfrac{S_h}{N_h}\right)(\lambda_{S_h}-\lambda_{E_h})
		+(1-\alpha_1u_2)\dfrac{S_v\lambda_v}{N_h}(\lambda_{E_v}-\lambda_{S_v})\\
		\end{split}
		\end{equation}
		\begin{equation}
		\label{adjoints2}
		\begin{split}
		\dfrac{d\lambda_{E_h}}{dt}&=\mu_h\lambda_{E_h}+\gamma_h(\lambda_{E_h}-\lambda_{I_h})
		+(1-\alpha_1u_2)\dfrac{S_h\lambda_h}{N_h}\left( \lambda_{E_h}-\lambda_{S_h}\right) 
		+(1-\alpha_1u_2)\dfrac{S_v}{N_h}\left( a\beta_{vh}\eta_h-\lambda_v\right) 
		\left( \lambda_{S_v}-\lambda_{E_v}\right) 
		\end{split}
		\end{equation}
		\begin{equation}
		\label{adjoints3}
		\begin{split}
		\dfrac{d\lambda_{I_h}}{dt}&=-D_1+\left[\mu_h+(1-\alpha_2u_3)\delta\right]\lambda_{I_h}
		+(\sigma+\alpha_2u_3)(\lambda_{I_h}-\lambda_{R_h})
		+(1-\alpha_1u_2)\dfrac{S_h\lambda_h}{N_h}(\lambda_{E_h}-\lambda_{S_h})\\
		&+(1-\alpha_1u_2)\dfrac{S_v}{N_h}\left( a\beta_{vh}-\lambda_v\right)
		\left( \lambda_{S_v}-\lambda_{E_v}\right)
		\end{split}
		\end{equation}
		\begin{equation}
		\label{adjoints4}
		\begin{split}
		\dfrac{d\lambda_{R_h}}{dt}&=\mu_h\lambda_{R_h}+\omega u_1(\lambda_{R_h}-\lambda_{S_h})
		+(1-\alpha_1u_2)\dfrac{S_h\lambda_h}{N_h}(\lambda_{E_h}-\lambda_{S_h})
		+(1-\alpha_1u_2)\dfrac{S_v\lambda_v}{N_h}(\lambda_{E_v}-\lambda_{S_v})
		\end{split}
		\end{equation}
		\begin{equation}
		\label{adjoints5}
		\begin{split}
		\dfrac{d\lambda_{S_v}}{dt}&=-D_2+(\mu_v+c_mu_4)\lambda_{S_v}+(1-\alpha_1u_2)\lambda_v
		(\lambda_{S_v}-\lambda_{E_v})-\mu_b\left( 1-\dfrac{E}{\Gamma_E}\right)\lambda_{E}
		\end{split}
		\end{equation}
		\begin{equation}
		\label{adjoints6}
		\begin{split}
		\dfrac{d\lambda_{E_v}}{dt}&=-D_2+(\mu_v+c_mu_4)\lambda_{E_v}+\gamma_v(\lambda_{E_v}-\lambda_{I_v})
		+a\eta_v\beta_{hv}(1-\alpha_1u_2)(\lambda_{S_h}-\lambda_{E_h})\dfrac{S_h}{N_h}
		-\mu_b\left( 1-\dfrac{E}{\Gamma_E}\right)\lambda_{E}
		\end{split}
		\end{equation}
		\begin{equation}
		\label{adjoints7}
		\begin{split}
		\dfrac{d\lambda_{I_v}}{dt}&=-D_2+(\mu_v+c_mu_4)\lambda_{I_v}
		+a\beta_{hv}(1-\alpha_1u_2)\dfrac{S_h}{N_h}(\lambda_{S_h}-\lambda_{E_h})
		-\mu_b\left( 1-\dfrac{E}{\Gamma_E}\right)\lambda_{E}
		\end{split}
		\end{equation}
		\begin{equation}
		\label{adjoints8}
		\begin{split}
		\dfrac{d\lambda_{E}}{dt}&=-D_3+\left[\dfrac{\mu_b}{\Gamma_E}N_v+s+\mu_E+\eta_1u_5\right]\lambda_E
		-s\left( 1-\dfrac{L}{\Gamma_L}\right)\lambda_{L}
		\end{split}
		\end{equation}
		\begin{equation}
		\label{adjoints9}
		\begin{split}
		\dfrac{d\lambda_{L}}{dt}&=-D_4-l\lambda_P+\left[\dfrac{s}{\Gamma_L}E+\mu_L+l+\eta_2u_5 \right]\lambda_L 
		\end{split}
		\end{equation}
		\begin{equation}
		\label{adjoints10}
		\begin{split}
		\dfrac{d\lambda_{P}}{dt}&=(\mu_P+\theta)\lambda_P-\theta\lambda_{S_v}
		\end{split}
		\end{equation}
	}
	and the transversality conditions 
	\begin{equation}
	\label{Trans--cond}
	\lambda^{*}_{i}(t_f)=0,\qquad i=1,\hdots 10.
	\end{equation}
	
	Furthermore,
	\begin{equation}
	\label{maxControl}
	\left. \begin{array}{ll}
	u^{\star}_1=\min\left\lbrace 1,\max\left(0,\dfrac{(S_h-\omega R_h)(\lambda_{S_h}-\lambda_{R_h})}{2B_1}\right)\right\rbrace,\\
	u^{\star}_2=\min\left\lbrace 1,\max\left(0,\dfrac{\alpha_1\left[\lambda_hS_h(\lambda_{E_h}-\lambda_{S_h})
		+\lambda_vS_v(\lambda_{E_v}-\lambda_{S_v})\right] }{2B_2}\right)\right\rbrace,\\
	u^{\star}_3=\min\left\lbrace 1,\max\left(0,\dfrac{\alpha_2\left[(1-\delta)\lambda_{I_h}-\lambda_{R_h}\right]I_h}{2B_3}\right)\right\rbrace,\\
	u^{\star}_4=\min\left\lbrace 1,\max\left(0,\dfrac{c_m\left[S_v\lambda_{S_v}+E_v\lambda_{E_v}+I_v\lambda_{I_v} \right] }{2B_4}\right)\right\rbrace, \\
	u^{\star}_5=\min\left\lbrace 1,\max\left(0,\dfrac{\eta_1E\lambda_{E}+\eta_2L\lambda_{L}}{2B_5}\right)\right\rbrace. 
	\end{array}\right.
	\end{equation}
\end{theorem}
\begin{proof}
	The differential equations governing the adjoint variables are obtained by differentiation of the Hamiltonian function, evaluated at the optimal control. Then the adjoint system can be written as
	\[
	\begin{split}
	\dfrac{d\lambda_{S_h}}{dt}&=-\dfrac{\partial \mathbb H}{\partial S_h},\,\,
	\dfrac{d\lambda_{E_h}}{dt}=-\dfrac{\partial \mathbb H}{\partial E_h},\,\,
	\dfrac{d\lambda_{I_h}}{dt}=-\dfrac{\partial \mathbb H}{\partial I_h},\,\,
	\dfrac{d\lambda_{R_h}}{dt}=-\dfrac{\partial \mathbb H}{\partial R_h},\\
	\dfrac{d\lambda_{S_v}}{dt}&=-\dfrac{\partial \mathbb H}{\partial S_v},\,\,
	\dfrac{d\lambda_{E_v}}{dt}=-\dfrac{\partial \mathbb H}{\partial E_v},\,\,
	\dfrac{d\lambda_{I_v}}{dt}=-\dfrac{\partial \mathbb H}{\partial I_v},\\
	\dfrac{d\lambda_E}{dt}&=-\dfrac{\partial \mathbb H}{\partial E},\,\,
	\dfrac{d\lambda_L}{dt}=-\dfrac{\partial \mathbb H}{\partial L},\,\,
	\dfrac{d\lambda_{P}}{dt}=-\dfrac{\partial \mathbb H}{\partial P},
	\end{split}
	\]
	with zero final time conditions (transversality).
	
	To get the characterization of the optimal control given by \eqref{maxControl}, we follow \cite{HelenaSofiaRodrigues2014,Lenhart} and solve the equations on the interior of the control set,
	\[
	\begin{split}
	\dfrac{\partial \mathbb H}{\partial u_i}=0,\,\,i=1,\hdots,5.
	\end{split}
	\]
	Using the bounds on the controls, we obtain the desired characterization. This ends the proof.\hfill
\end{proof}

\section{Numerical simulations and discussion}
\label{NUMopc}
The simulations were carried out using the values of Table \ref{vaueR0ar3opc}. We use an iterative scheme to solve the optimality system. We first solve the state equations~\eqref{AR3optimal} with a guess for the controls over the simulated time using fourth order Runge--Kutta scheme. Then, we use
the current iterations solutions of the state equation to solve the adjoint equations~\eqref{adjoints1}--~\eqref{adjoints10} by a backward fourth order Runge--Kutta scheme. Finally, we update the controls by using a convex combination of the previous controls and the value from the characterizations~\eqref{maxControl} (see e.g. \cite{moaaHee2012,bbMBE2011,Zaman2008,Lenhart,Okosun2011}).
The values chosen for the weights in the objective functional $J$ (see Eq.~\eqref{OCF}) are given in Table \ref{OPCcost}. 
Table \ref{OPCinittab1} gives the initials conditions of state variables. We simulated the system~\eqref{AR3optimal} in a period of twenty days ($t_f=20$). 
\begin{table}[h!]
	\begin{center}
		\caption{Value of parameters using in numerical simulations.\label{vaueR0ar3opc}}
		\begin{tabular}{cccccccc}
			\hline
			\hline
			Parameter&value&Parameter&value&Parameter&value&Parameter&value\\
			\hline
			$\mu_v$&$\frac{1}{30}$&$l$&0.5&$\alpha_2$&0.5&$\gamma_h$&$\frac{1}{14}$\\
			$a$&1&$\mu_E$&0.2&$\mu_h$&$\frac{1}{67*365}$&$\gamma_v$&$\frac{1}{21}$\\
			$\Lambda_h$&2.5&$\mu_b$&6&$\theta$&0.08&$\mu_P$&0.4\\
			$\beta_{hv}$&0.75&&&$\sigma$&0.1428&$\eta_v$&0.35\\
			$\beta_{vh}$&0.75&$\omega$&0.05&$\mu_L$&0.4&$\delta$&$10^{-3}$\\
			$\Gamma_{E}$&10000&$s$&0.7&$\eta_1$&0.001&$\eta_2$&0.3\\
			$\Gamma_{L}$&5000&$\eta_h$&0.35&$c_m$&0.2&$\alpha_1$&0.5\\
			\hline
		\end{tabular}
	\end{center}
\end{table}

\begin{table}[h!]
	\begin{center}
		\caption{Numerical values for the cost functional parameters.} \label{OPCcost}
		\begin{tabular}{lllrrr}
			\hline
			\hline
			Parameters&Value&Source&Parameters&Value&Source\\
			\hline
			$D_1$:&10,000&\cite{moaaHee2012}&$B_1$&10 &Assumed\\
			$D_{2}$:&10,000& \cite{moaaHee2012}&$B_2$:&10 &\cite{moaaHee2012}\\
			$D_{3}$:&5000&Assumed&$B_3$:&10&\cite{moaaHee2012}\\
			$D_4$:&1&\cite{moaaHee2012}&$B_4$:&10 &Assumed\\
			&&&$B_5$&10&\cite{moaaHee2012}\\
			\hline
		\end{tabular}
	\end{center}
\end{table}
\begin{table}[h!]
	\begin{center}
		\caption{Initial conditions.} \label{OPCinittab1}
		\begin{tabular}{llllll}
			\hline
			\hline
			Human states& Initial value& Adult Vector & Initial value & Aquatic states & Initial value\\
			& & states& & & \\
			\hline
			$S_{h_0}$:& 700& $S_{v_0}$&3000& $E_{0}$&10000 \\
			$E_{h_0}$:& 220& $E_{v_0}$& 400 &$L_{0}$&5000 \\
			$I_{h_0}$:& 100& $I_{v_0}$& 120&$P_{0}$&3000\\
			$R_{h_0}$:& 60 & \\
			\hline
		\end{tabular}
	\end{center}
\end{table}

As the purpose of our study is to seek the best combination linking vaccination to other control mechanism, we will just determine the best strategy among the strategies listed in Table \ref{ControlStrategies}. 
Therefore, we will distinguished five control strategies at follows:

\begin{table}[h!]
	\begin{center}
		\caption{Description of the different Control strategies.} \label{ControlStrategies}
		\begin{tabular}{|l|l|}
			\hline
			\hline
			Strategy& Description\\
			\hline
			$Z_{1}$&Vaccine combined with individual protection, treatment and adulticide\\
			\hline
			$Z_{2}$&Vaccine combined with individual protection, treatment and larvicide\\
			\hline
			$Z_{3}$&Vaccine combined with treatment, adulticide and larvicide\\
			\hline
			$Z_{4}$&Vaccine combined with individual protection, adulticide and larvicide\\
			\hline
			Z&Combination of the five controls\\
			\hline
		\end{tabular}
	\end{center}
\end{table}

\begin{figure}[h!]
	\begin{center}
		\includegraphics[scale=0.20]{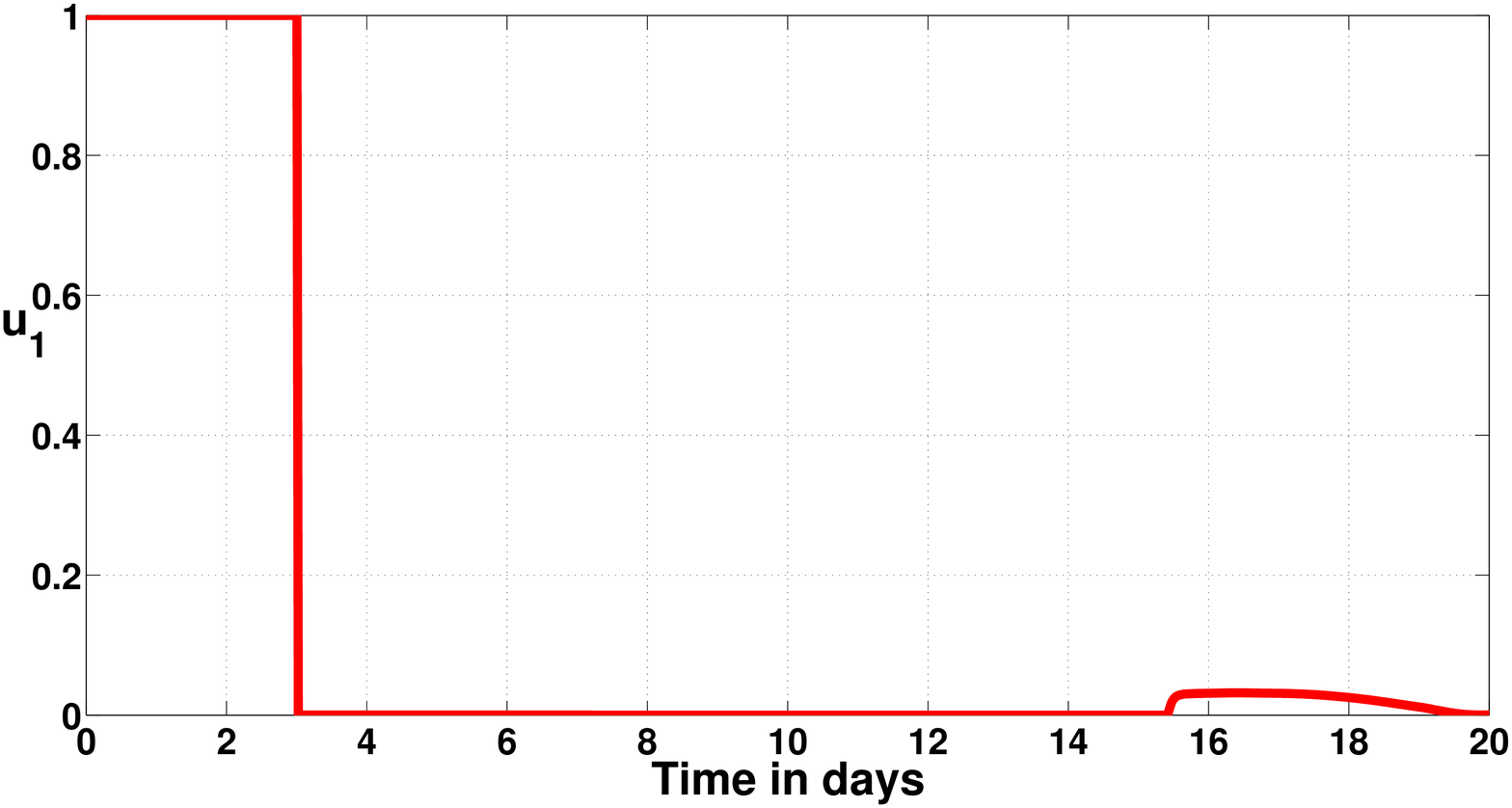}
		\includegraphics[scale=0.20]{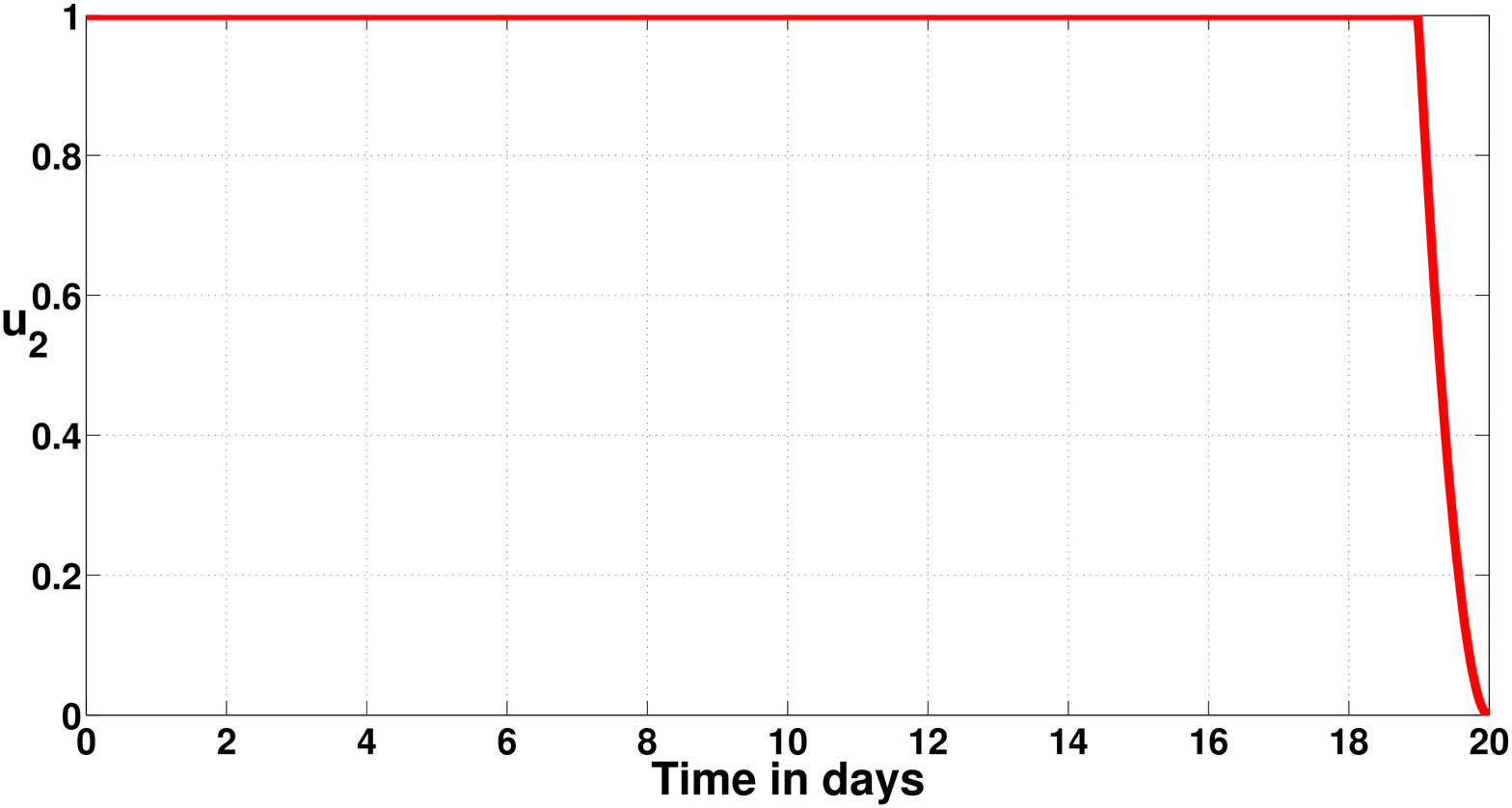}
		\includegraphics[scale=0.20]{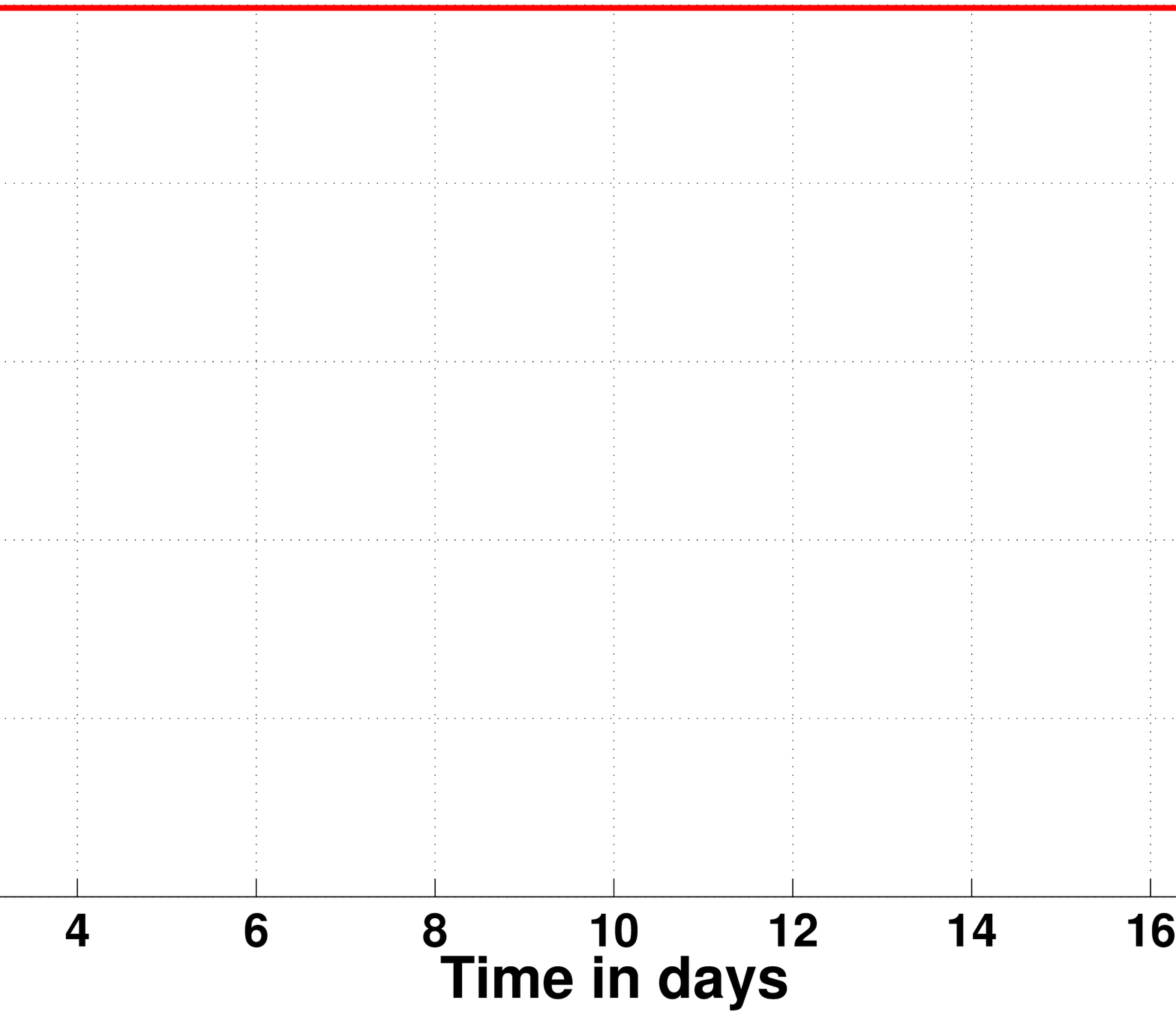}
		\includegraphics[scale=0.20]{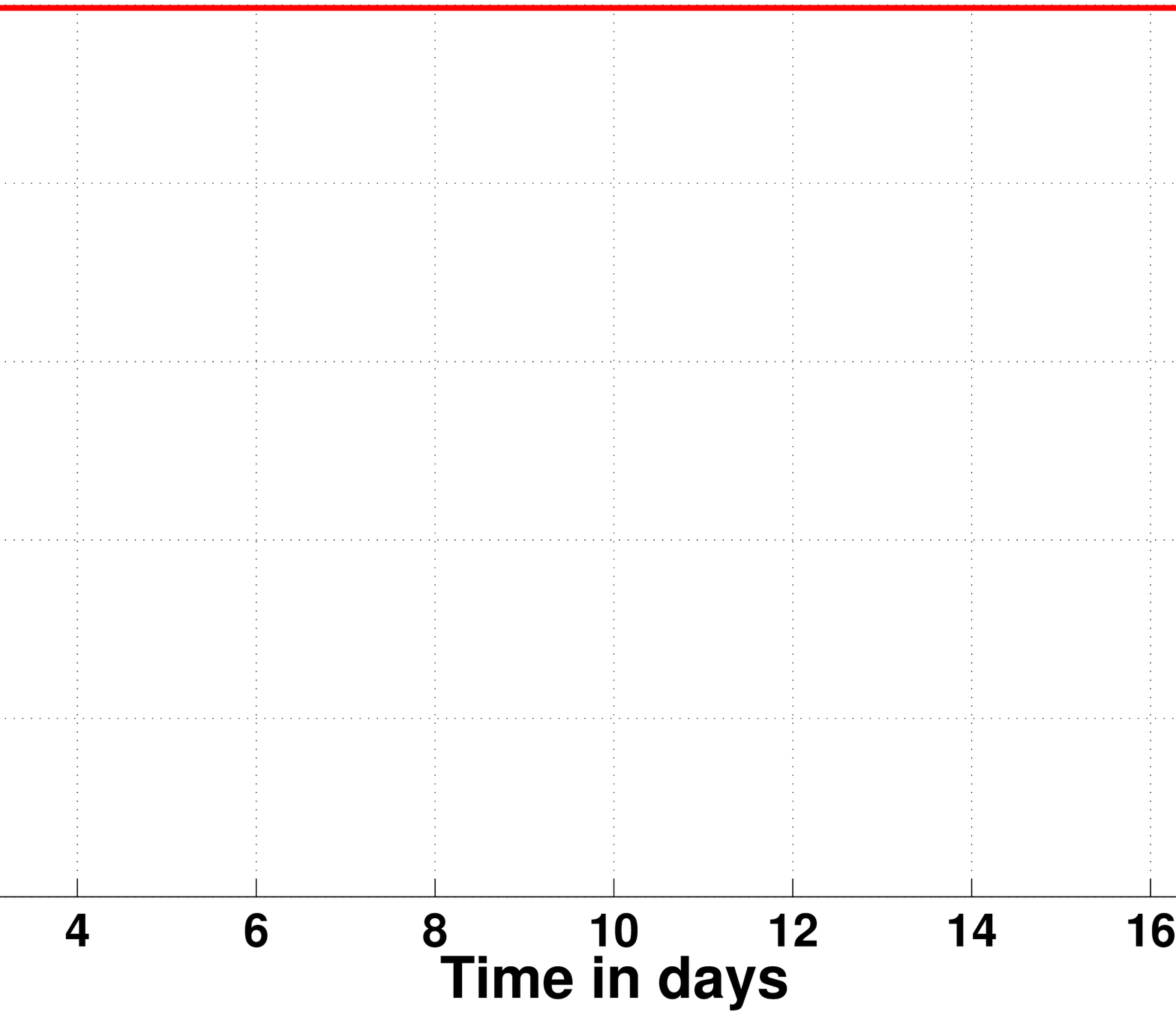}
		\includegraphics[scale=0.20]{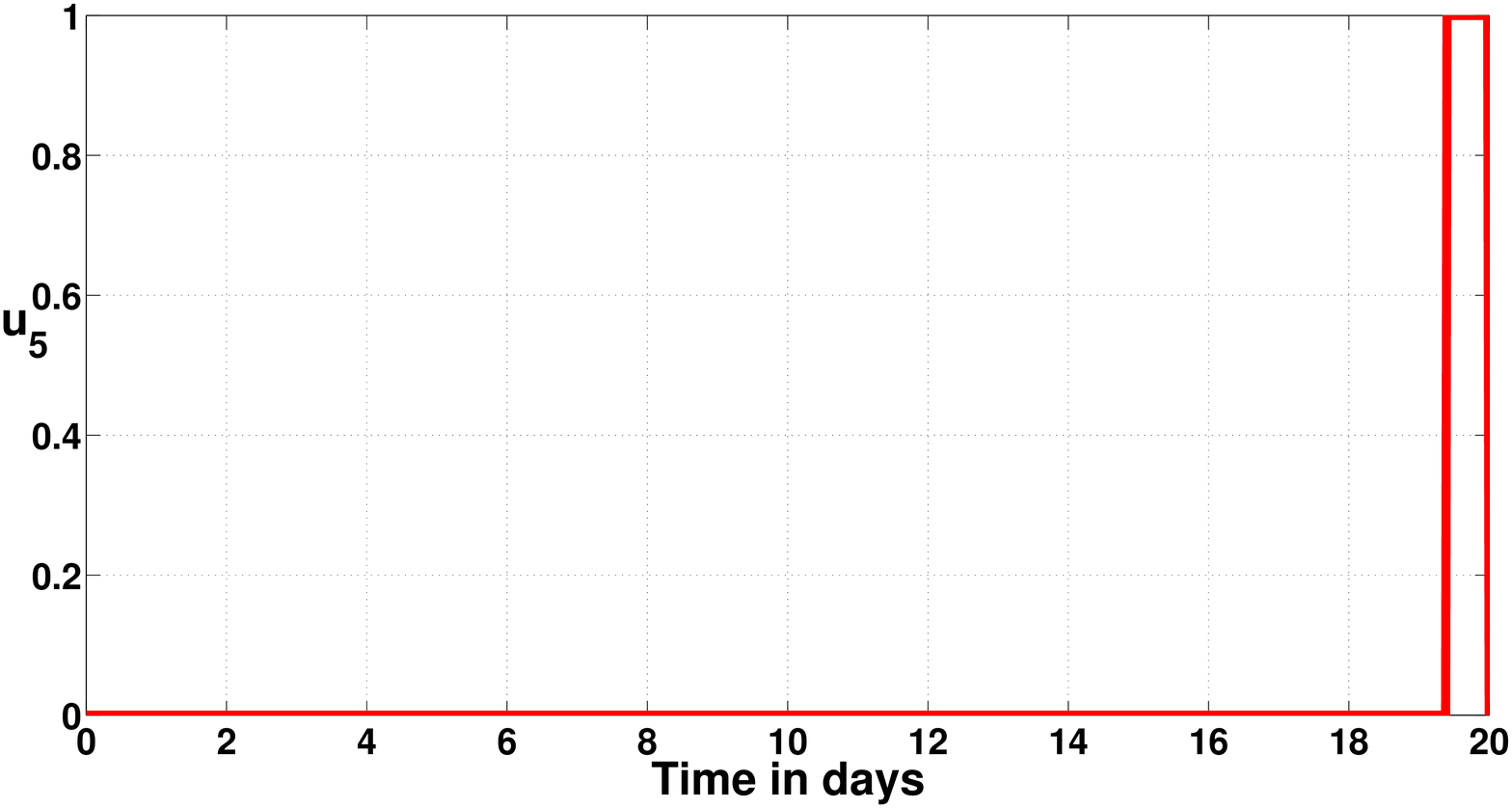}
		\caption{Control functions.\label{controls}}
	\end{center}
\end{figure}

\paragraph{(i) Vaccination combined with individual protection, treatment and adulticide}
$\;$

With this strategy, only the combination of the control $u_1$ on vaccination, the control $u_2$ on individual protection, the control $u_3$ on treatment and the control $u_4$ on adulticide, is used to minimise the objective function $J$~\eqref{OCF}, while the other control $u_5$ is set to zero. On figure~\ref{onlyu1+u2+u3+u4}, we observed that the control strategy resulted in a decrease in the number of infected humans ($I_h$) while an increase is observed in the number of infected humans ($I_h$) in strategy without control. The use of this combination have also a great impact on the decreasing total vector population ($N_v$), as well as aquatic vector populations ($E$ and $L$). 
\begin{figure}[h!]
	\begin{center}
		\includegraphics[scale=0.20]{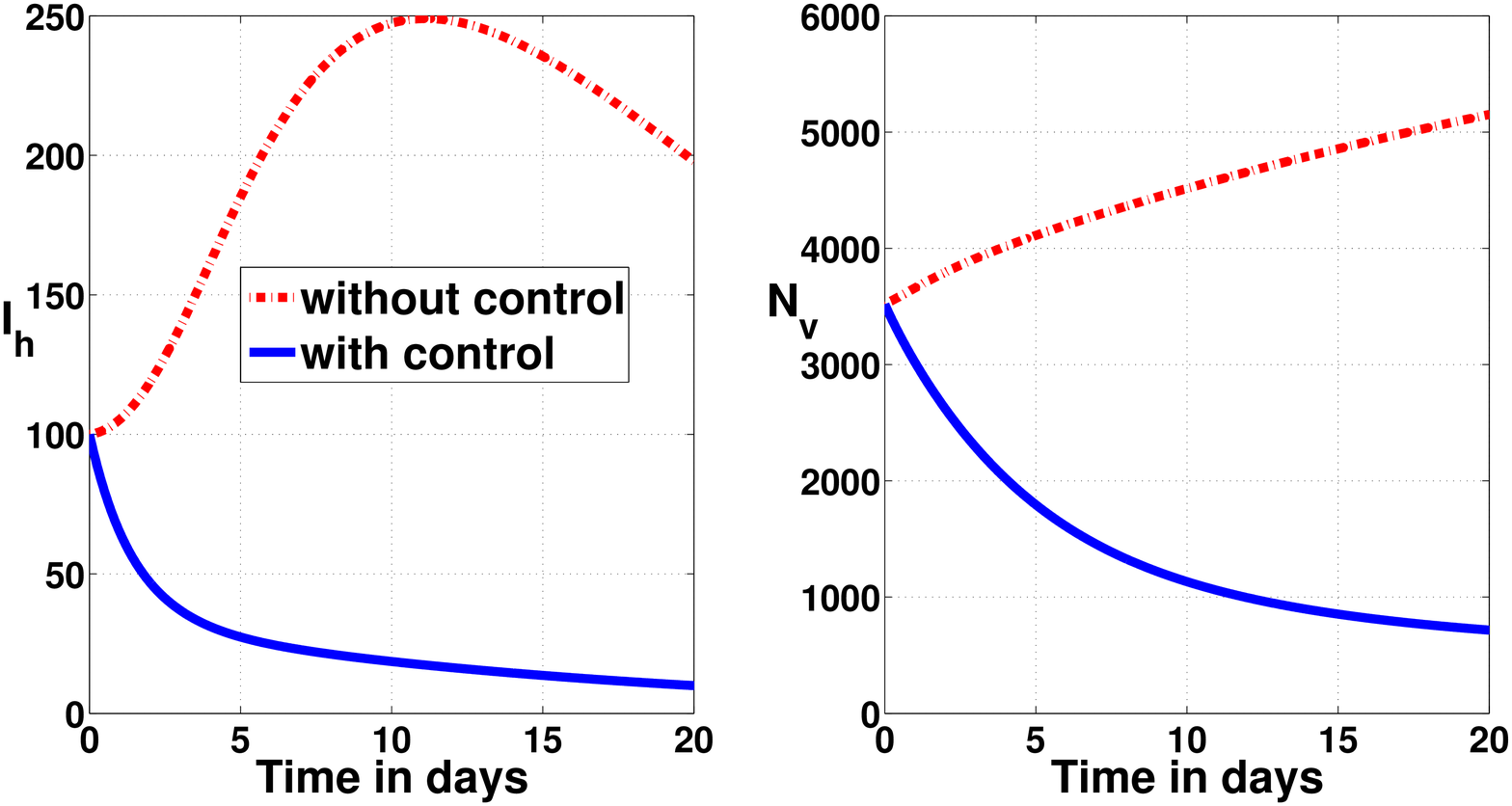}
		\includegraphics[scale=0.20]{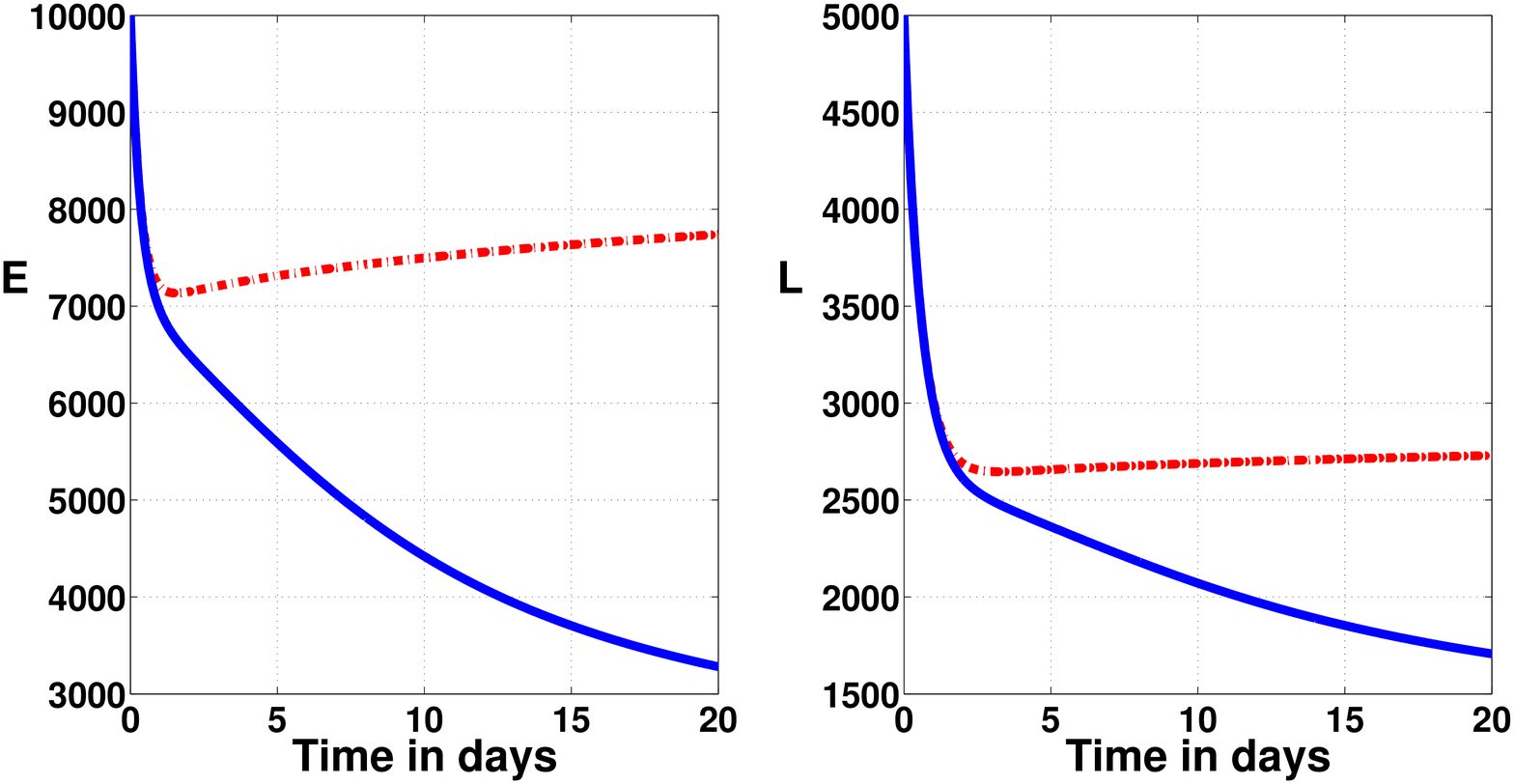}
		\caption{Simulation results of optimal control model \eqref{AR3optimal} showing the effect of using optimal vaccination combined with individual protection, treatment and adulticide ($u_1\neq 0$, $u_2\neq 0$, $u_3\neq 0$, $u_4\neq 0$ ).\label{onlyu1+u2+u3+u4}}
	\end{center}
\end{figure}

\paragraph{(ii) Vaccination combined with individual protection, treatment and larvicide}
$\;$

With this strategy, only the combination of the control $u_1$ on vaccination, the control $u_2$ on individual protection, the control $u_3$ on treatment and the control $u_5$ on larvicide, is used to minimise the objective function $J$~\eqref{OCF}, while the other control $u_4$ is set to zero. On figure~\ref{onlyu1+u2+u3+u5}, we observed that the control strategy resulted in a decrease in the number of infected humans ($I_h$) while an increase is observed in the number of infected humans ($I_h$) in strategy without control. The use of this combination have no impact on the decreasing total vector population ($N_v$), as well as aquatic vector populations ($E$ and $L$). 
\begin{figure}[h!]
	\begin{center}
		\includegraphics[scale=0.20]{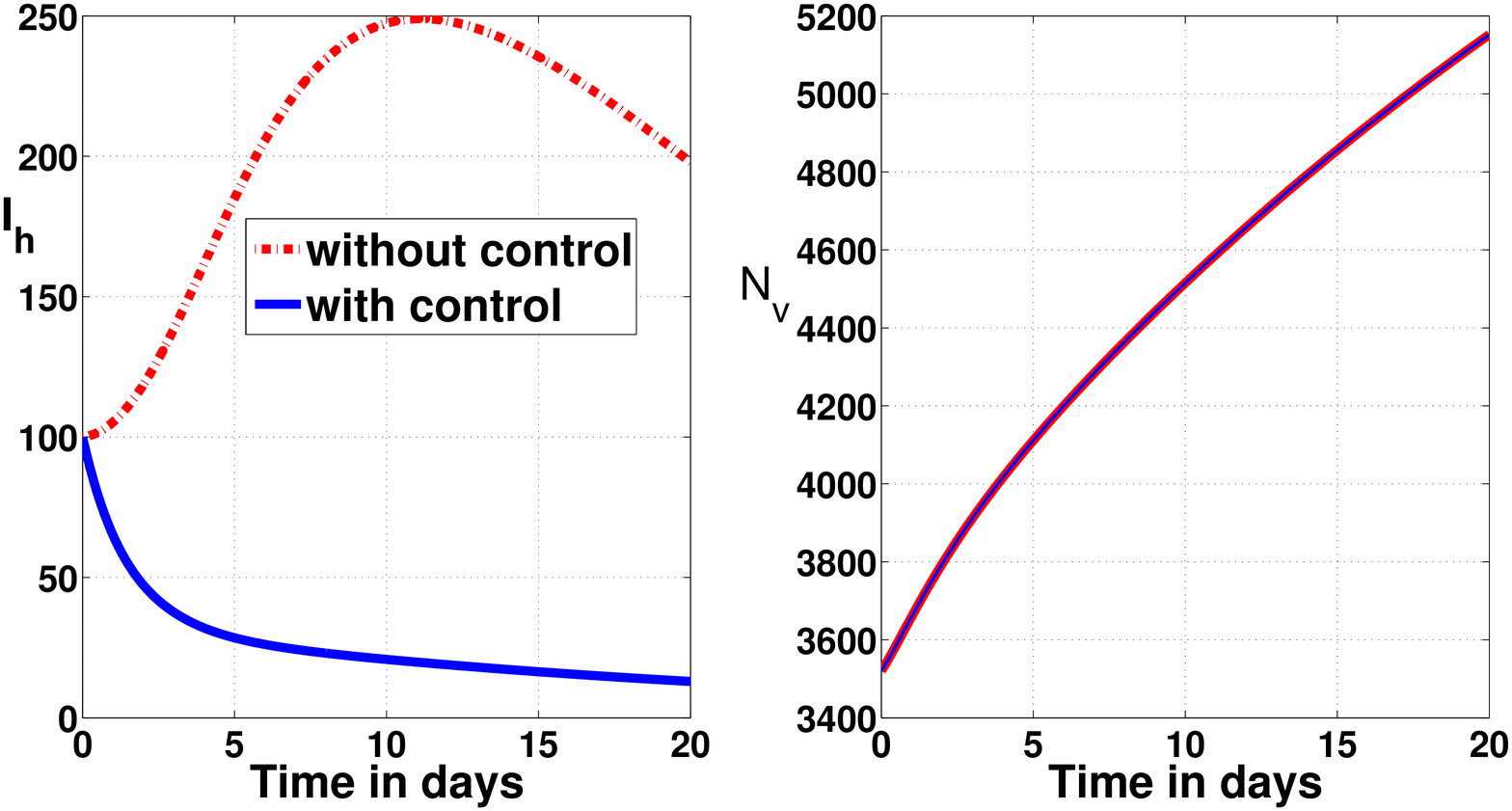}
		\includegraphics[scale=0.20]{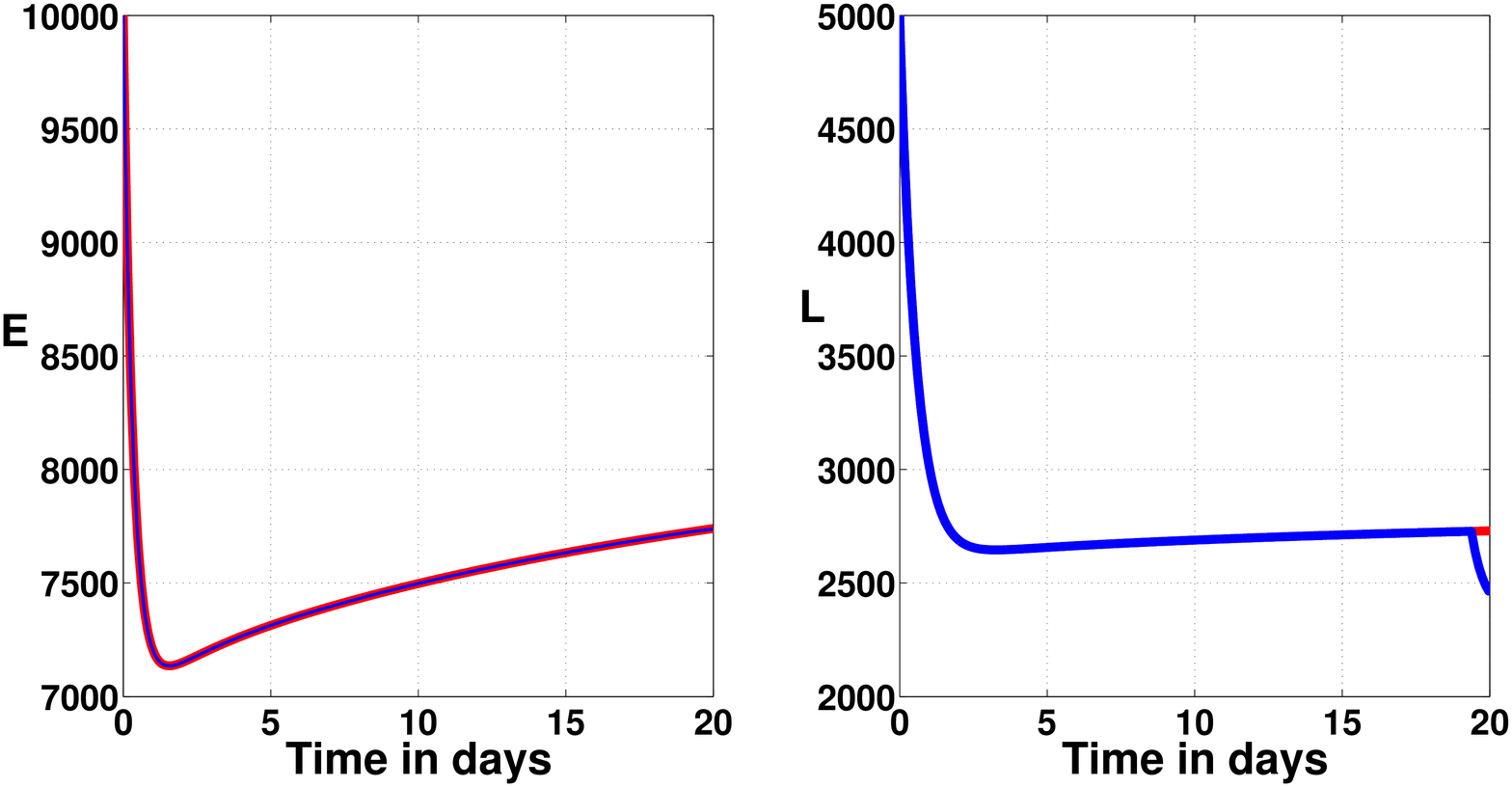}
		\caption{Simulation results of optimal control model \eqref{AR3optimal} showing the effect of using optimal vaccination combined with individual protection, treatment and larvicide ($u_1\neq 0$, $u_2\neq 0$, $u_3\neq 0$, $u_5\neq 0$ ).\label{onlyu1+u2+u3+u5}}
	\end{center}
\end{figure}

\paragraph{(iii) Vaccination combined with treatment, adulticide and larvicide}
$\;$

With this strategy, only the combination of the control $u_1$ on vaccination, the control $u_3$ on treatment, the control $u_4$ on adulticide and the control $u_5$ on larvicide, is used to minimise the objective function $J$~\eqref{OCF}, while the other control $u_2$ is set to zero. On figure~\ref{onlyu1+u3+u4+u5}, we observed that the control strategy resulted in a decrease in the number of infected humans ($I_h$) while an increase is observed in the number of infected humans ($I_h$) in strategy without control. The use of this combination have a considerable impact on the decreasing total vector population ($N_v$), as well as aquatic vector populations ($E$ and $L$). 
\begin{figure}[h!]
	\begin{center}
		\includegraphics[scale=0.20]{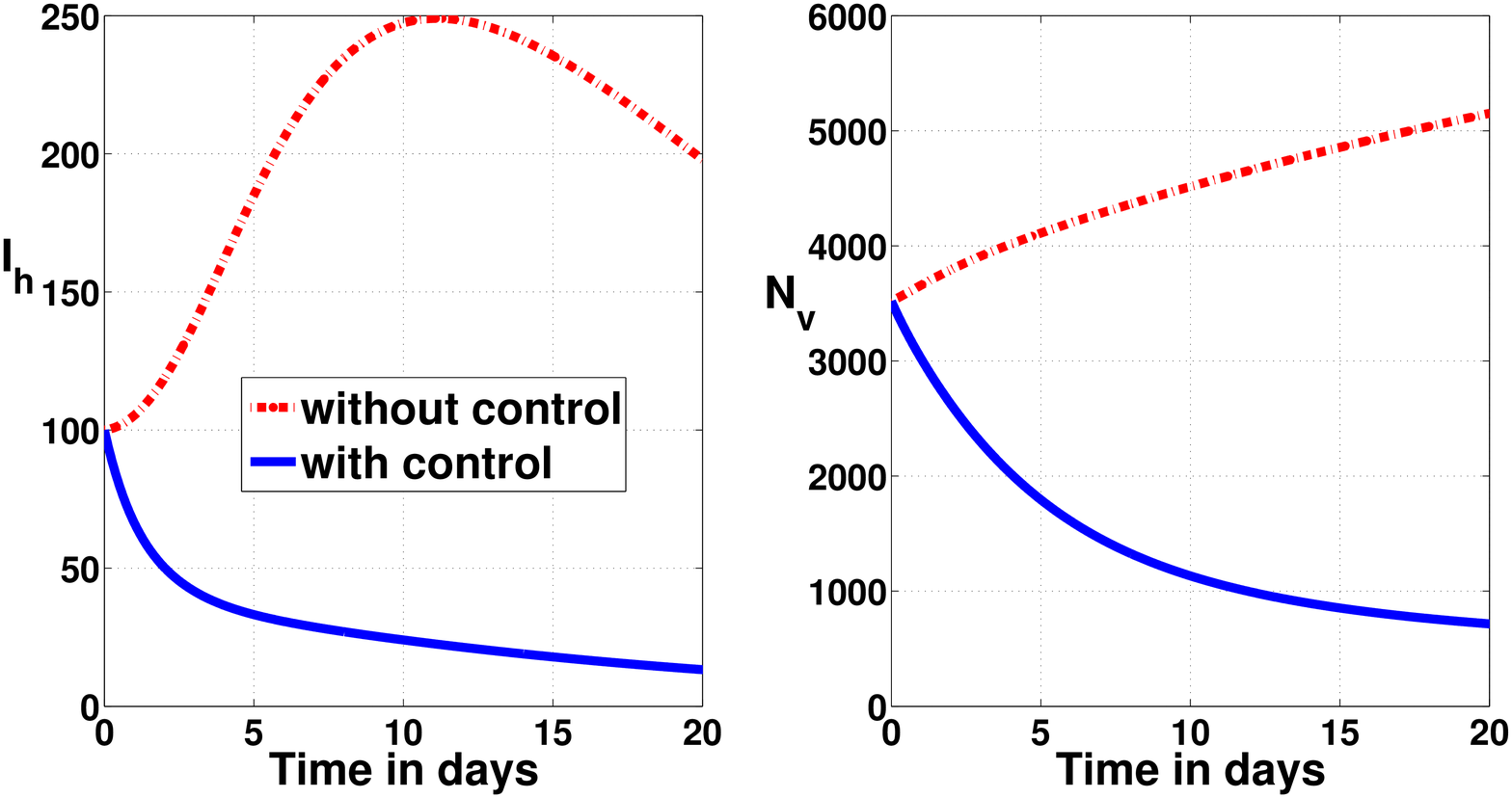}
		\includegraphics[scale=0.20]{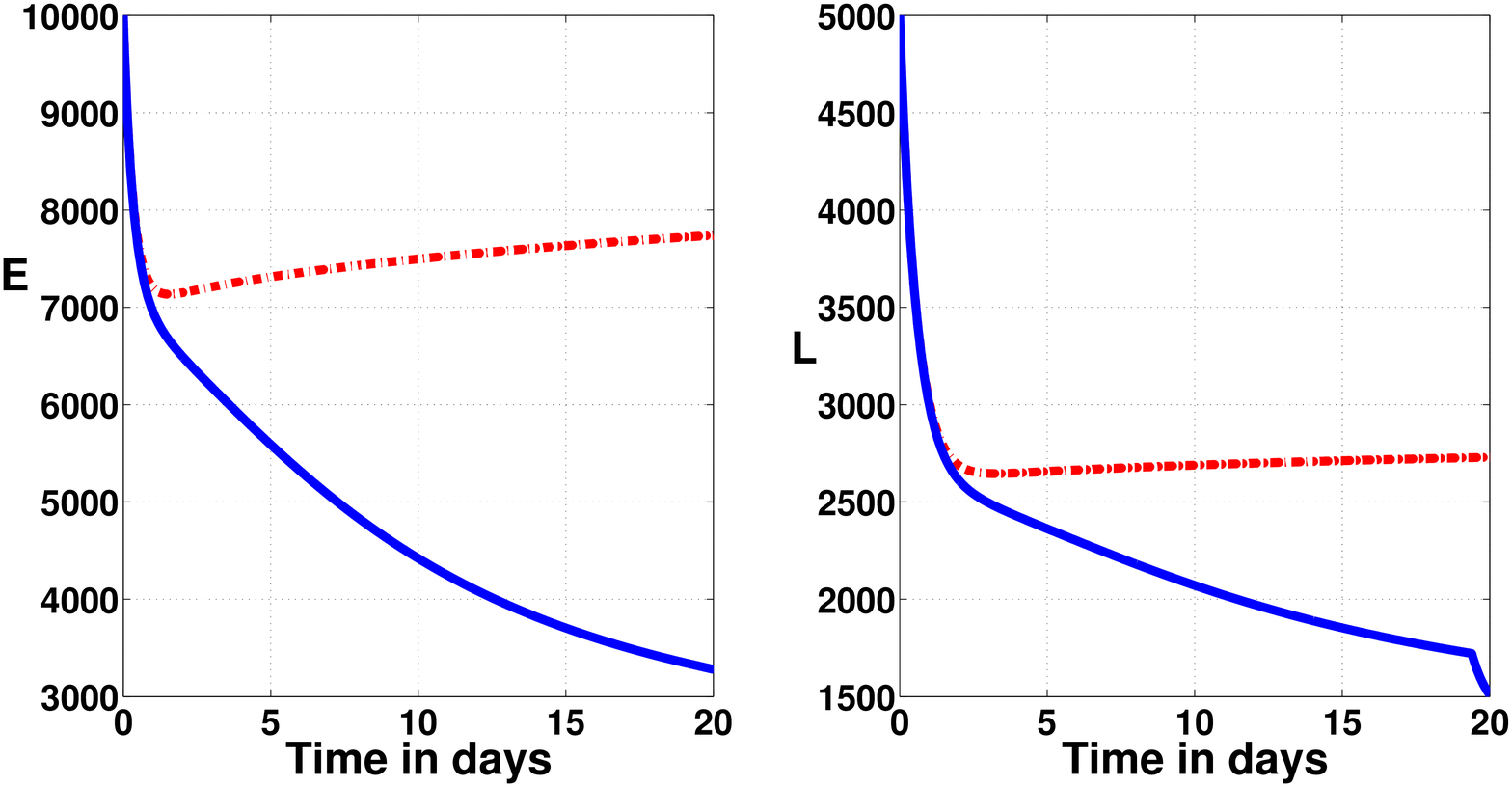}
		\caption{Simulation results of optimal control model \eqref{AR3optimal} showing the effect of using optimal vaccination combined with treatment, adulticide and larvicide ($u_1\neq 0$, $u_3\neq 0$, $u_4\neq 0$, $u_5\neq 0$ ).\label{onlyu1+u3+u4+u5}}
	\end{center}
\end{figure}

\paragraph{(iv) Vaccination combined with individual protection, adulticide and larvicide}
$\;$

With this strategy, only the combination of the control $u_1$ on vaccination, the control $u_2$ on individual protection, the control $u_4$ on adulticide and the control $u_5$ on larvicide, is used to minimise the objective function $J$~\eqref{OCF}, while the other control $u_4$ are set to zero. On figure~\ref{onlyu1+u2+u4+u5}, we observed that the control strategy resulted in a decrease in the number of infected humans ($I_h$) while an increase is observed in the number of infected humans ($I_h$) in strategy without control. The use of this combination have a great impact on the decreasing total vector population ($N_v$), as well as aquatic vector populations ($E$ and $L$). 
\begin{figure}[h!]
	\begin{center}
		\includegraphics[scale=0.20]{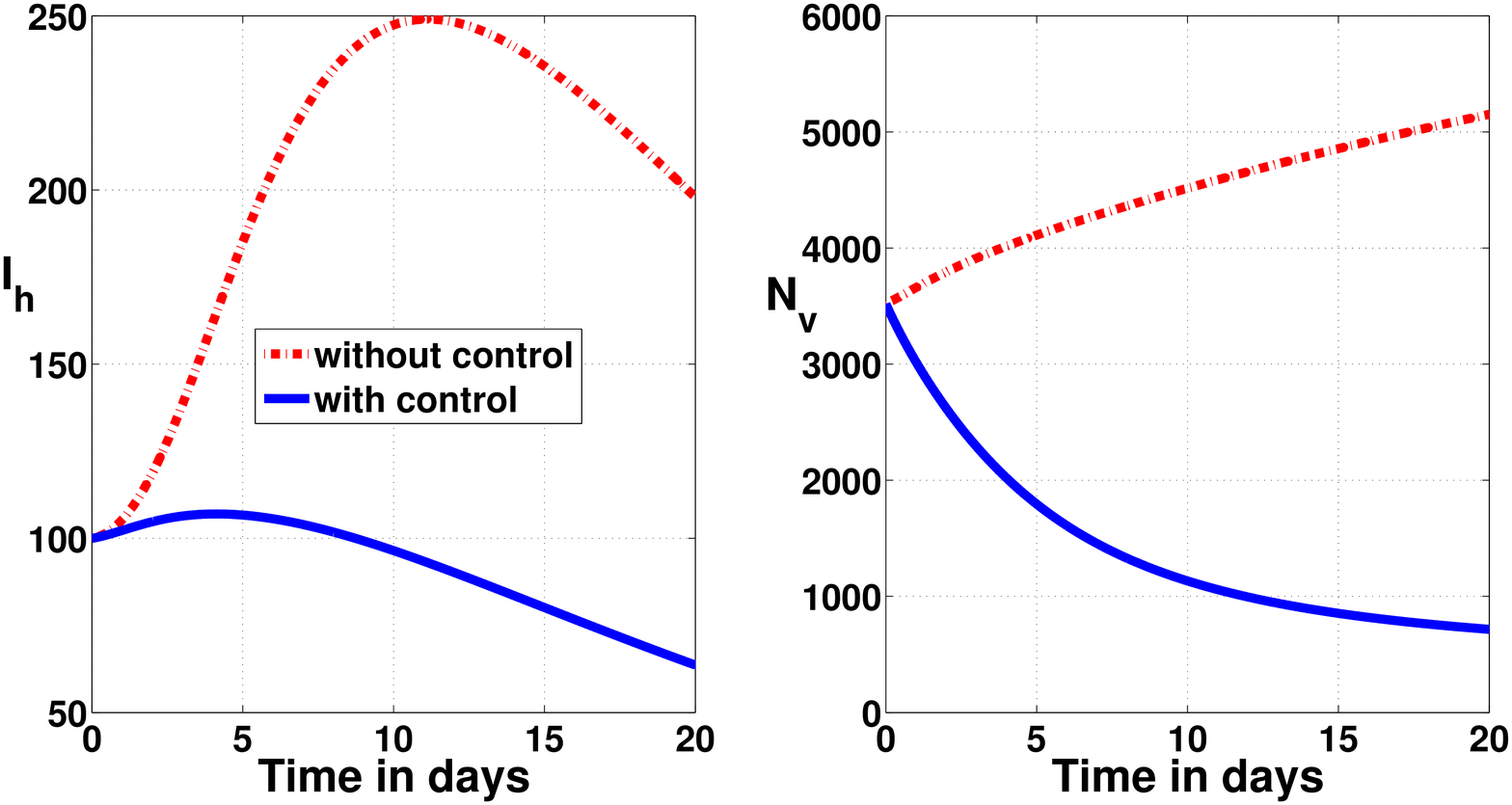}
		\includegraphics[scale=0.20]{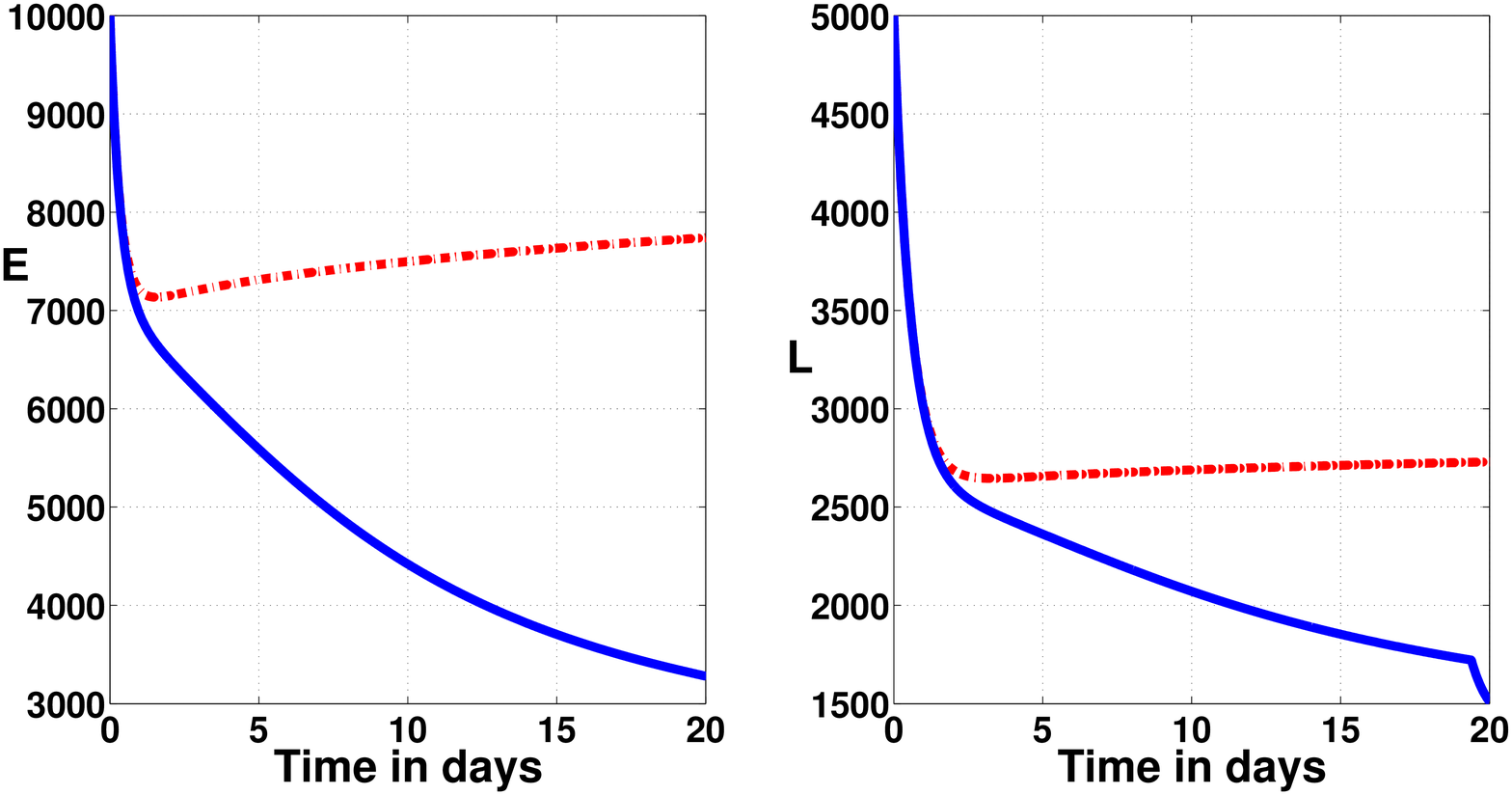}
		\caption{Simulation results of optimal control model \eqref{AR3optimal} showing the effect of using optimal vaccination combined with individual protection, adulticide and larvicide ($u_1\neq 0$, $u_2\neq 0$, $u_4\neq 0$, $u_5\neq 0$ ).\label{onlyu1+u2+u4+u5}}
	\end{center}
\end{figure}

\paragraph{(v) The combination of all the five controls}
$\;$

In this strategy, the combination of all the five controls is applied. On figure~\ref{Fullcontrol}, we observed that combining all the five controls give a better result in a decrease in the number of infected humans ($I_h$), as well as, the total number of vector population ($N_v$), and the aquatic vector populations ($E$ and $L$).
\begin{figure}[h!]
	\begin{center}
		\includegraphics[scale=0.20]{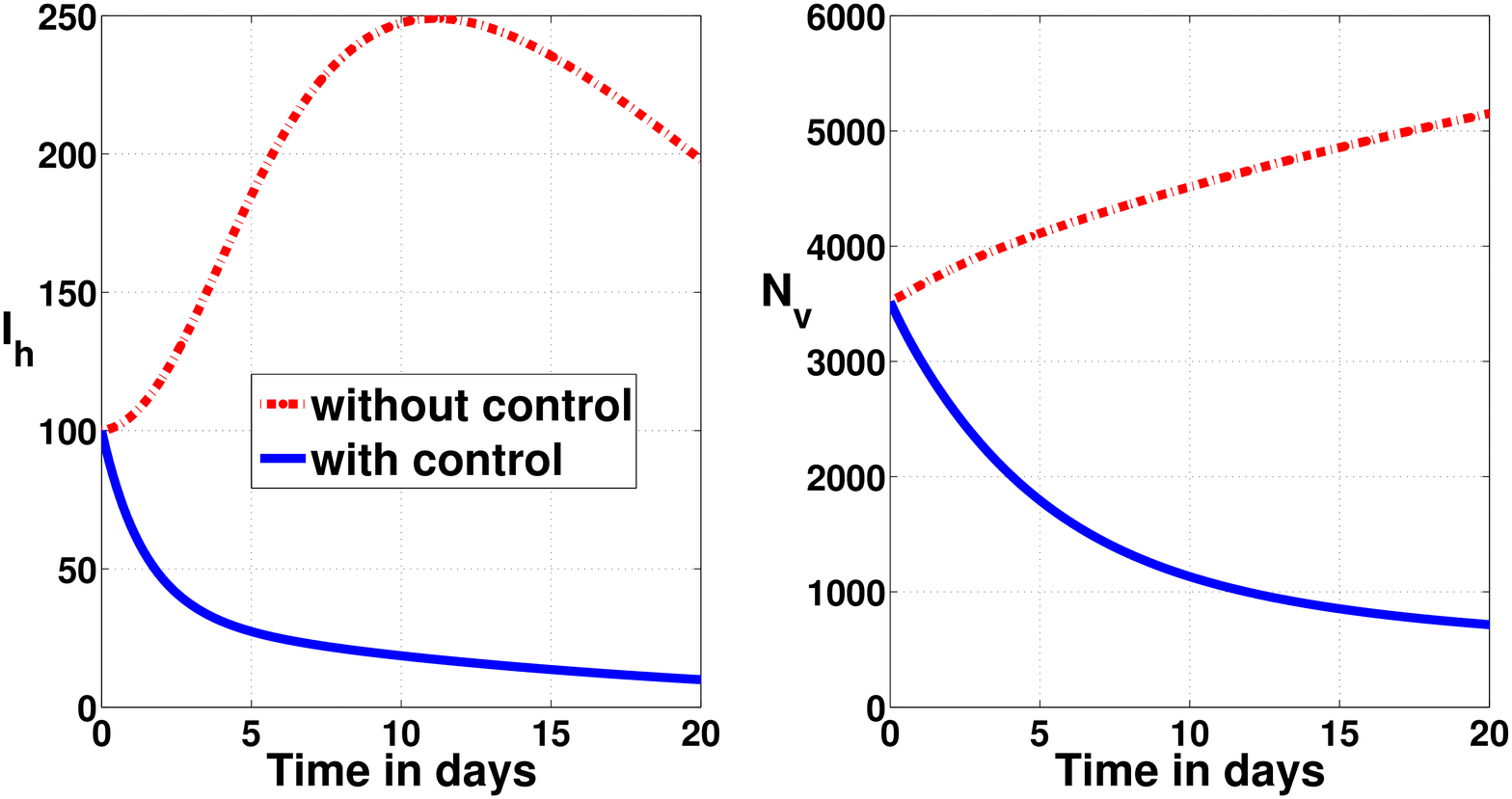}
		\includegraphics[scale=0.20]{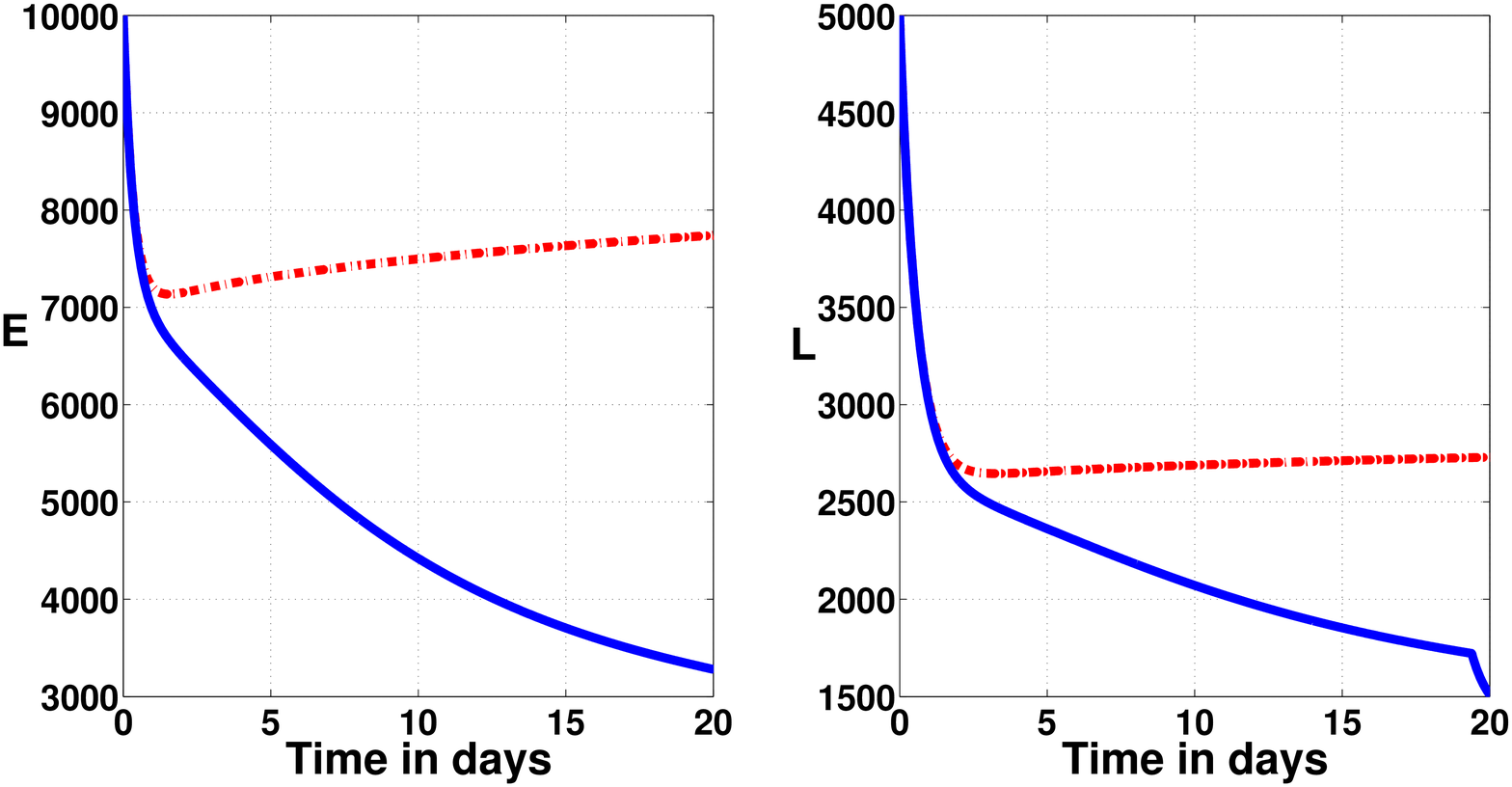}
		\caption{Simulation results of optimal control model \eqref{AR3optimal} showing the effect of using the combination of all the five controls ($u_i\neq 0$, $i=1,\hdots, 5$).\label{Fullcontrol}}
	\end{center}
\end{figure}

Figures \ref{onlyu1+u2+u3+u4}, \ref{onlyu1+u3+u4+u5}, \ref{onlyu1+u2+u4+u5}, and \ref{Fullcontrol} have the same shape, which it is difficult to say what is the best control strategy. In the following, we make an efficiency analysis and a cost effectiveness analysis to determine the best strategy in terms of efficiency and cost

\subsection{Efficiency analysis}
In line with Yang and Ferreira~\cite{YangFerreira2008}, Dumont and Chiroleu~\cite{duch}, and Carvalho et~{\it al.}~\cite{Carvalho2015}, we compare the effects of different strategies applied on the arboviral diseases, by the introduction of the efficiency index, designed by $\mathbb F$. To this aim, we define the variable $\mathcal A$ as the area comprised between the curve of the symptomatic infectious human ($I_h$) population size, for instance, and the time axis during the period of time from 0 to $t_f$, as

\begin{equation}
\label{area}
\mathcal A=\int_{0}^{t_h}I_h(t)dt,
\end{equation}
which measures the cumulated number of infectious human during the time interval $[0,t_f]$
~\cite{Carvalho2015,YangFerreira2008}. Hence the efficiency index, $\mathbb F$, be can defined by 

\begin{equation}
\label{effecFormulae}
\mathbb F=\left(1-\dfrac{\mathcal A^{c}_{h}}{\mathcal A^{(0)}_{h}}\right)\times 100 ,
\end{equation}
where $\mathcal A^{c}_{h}$ and $\mathcal A^{0}_{h}$ are the cumulated number of infectious human with and without the different controlling mechanisms, respectively. So, It follows that the best strategy will be the one whom efficiency index will be the biggest~\cite{Carvalho2015,YangFerreira2008}.

With the previous simulations, we resume the efficiency index of different strategies in the Tables~\ref{efficiencyTable1}. 
\begin{table}[h!]
	\begin{center}
		\caption{Table of efficiency index (the case of infected humans). \label{efficiencyTable1}} 
		\begin{tabular}{|l|l|l|l|l|l|}
			\hline
			\hline
			Strategy&$\mathcal A_{I_h}=$&$\mathbb F_{I_h}=100\times$&Strategy&$\mathcal A_{I_h}=$&$\mathbb F_{I_h}=100\times$\\
			&$\int_{0}^{t_h}I_h(t)dt$&$\left(1-\dfrac{\mathcal A^{c}_{I_h}}{\mathcal A^{(0)}_{I_h}}\right)$
			&&$\int_{0}^{t_h}I_h(t)dt$&$\left(1-\dfrac{\mathcal A^{c}_{I_h}}{\mathcal A^{(0)}_{I_h}}\right)$\\
			\hline
			No controls&$4.1052\times 10^{3}$&0\%&$Z_{1}$&$490.6350$&88.048451\%\\
			\hline
			$Z_{2}$&$528.9018$&87.116296\%&$Z_{3}$&$580.5772$&85.857517\%\\
			\hline
			$Z_{4}$&$1.8391\times 10^{3}$&55.200721\%&Z&$490.6350$&88.048451\%\\
			\hline
		\end{tabular}
	\end{center}
\end{table}

From table~\ref{efficiencyTable1}, we can conclude that the best strategies are $Z_1$ and $Z$.
\begin{remark}
	The same reasoning (about efficiency analysis) can be done for vector population, by replacing 
	$I_{h}$ by $N_{v}$, and $\mathcal A_{h}$ by $\mathcal A_{N_v}$ in Eq.~\eqref{area} and \eqref{effecFormulae}, respectively.
\end{remark}

\subsection{Cost Effectiveness Analysis}
\label{costeffectiveness}
The analysis of efficiency allowed us to determine the most efficient strategy, regardless of the cost associated with each control. In what follows, we will among the strategies listed in Table~\ref{ControlStrategies}, determine which one is the most efficient and which can be implemented at lower cost. To this aim, we follow Okosun and co-workers~\cite{Okosun2011}, and use cost effectiveness analysis to determine the most cost effective strategy to use the controls of arboviral diseases (Strategies $Z_1-Z$). To this aim, we calculate the Incremental Cost-Effectiveness Ratio (ICER) which is generally described as the additional cost per additional health outcome (see \cite{Okosun2011}). 

Based on the model simulation results, we rank the strategies in decreasing number of Total number of infectious individuals (see Table~\ref{CostTable}).
\begin{table}[H]
	\begin{center}
		\caption{Cost Effectiveness of different  strategies.} \label{CostTable}
		\begin{tabular}{|c|c|c|c|}
			\hline
			Strategies&Total number of infected &Total infection &Total cost(\$)\\
			&individuals &averted&\\
			\hline
			$Z_{4}$&1839&87538&$7.6218\times 10^{8}$\\
			\hline
			$Z_{3}$&581&88796&$7.6093\times 10^{8}$\\
			\hline
			$Z_{2}$&529&88848&$1.6452\times 10^{9}$\\
			\hline
			$Z_{1}$&491&88886&$7.6081\times 10^{8}$\\
			\hline
			$Z$&491&88886&$7.6081\times 10^{8}$\\
			\hline
		\end{tabular}
	\end{center}
\end{table}
 
The difference between the total number of infectious individuals without control and the total number of infectious individuals with control was used to determine the "total number of infection averted" used in the tables of cost--effectiveness analysis.
 
We obtain the ICER of strategies (i) and (j) by applying the formula given by equation \eqref{costi--costj}.
\begin{equation}
\label{costi--costj}
\begin{array}{l}
ICER(i)=\dfrac{Total\,\, cost(i)}{Total\,\, infection\,\, averted(i)}\\
\begin{split}
ICER(j)&=\dfrac{Total\,\,cost(j)-Total\,\, cost(i)}{Total\,\, infection\,\, averted(j)-Total\,\, infection\,\, averted(i)}\\
\end{split}
\end{array}
\end{equation}
\begin{table}[H]
	\begin{center}
		\caption{Strategy $Z_4$ vs  Strategy $Z_3$.\label{Z4vsZ3}}
		\begin{tabular}{ccc}
			\hline
			Strategies&Total infection averted  &Total cost ($\$$)\\
			\hline
			$Z_4$&87538&$7.6218\times 10^{8}$\\
			\hline
			$Z_3$&88796&$7.6093\times 10^{8}$\\
			\hline
		\end{tabular}
	\end{center}
\end{table}

\begin{equation}
	\label{costZ4--costZ3}
	\begin{array}{ll}
		ICER(Z_4)&=\dfrac{7.6218\times 10^{8}}{87538}=8706.8473\\
		ICER(Z_3)&=\dfrac{(7.6093-7.6218)\times 10^{8}}{88796-87538}=-993.6407\\
	\end{array}
\end{equation}
The comparison between ICER($Z_4$) and ICER($Z_3$) shows a cost saving of \$ 993.6407 for strategy $Z_3$ over strategy $Z_4$. The negative ICER for strategy $Z_3$ indicates that the strategy $Z_4$ is "strongly dominated". That is, strategy $Z_4$ is more costly and less effective than strategy $Z_3$. Therefore, strategy $Z_4$, the strongly dominated is excluded from the set of alternatives so it does not consume limited resources. 

We exclude strategy $Z_4$ and compare strategy $Z_3$ with $Z_2$. From table~\ref{CostTable}, we have:
\begin{table}[H]
	\begin{center}
		\caption{Strategy $Z_3$ vs  Strategy $Z_2$.\label{Z3vsZ2}}
		\begin{tabular}{ccc}
			\hline
			Strategies&Total infection averted  &Total cost ($\$$)\\
			\hline
			$Z_3$&88796&$7.6093\times 10^{8}$\\
			\hline
			$Z_2$&88848&$1.6452\times 10^{9}$\\
			\hline
		\end{tabular}
	\end{center}
\end{table}

This leads to the following values for the ICER,
\begin{equation}
\label{costZ3--costZ2}
\begin{array}{l}
ICER(Z_3)=\dfrac{7.6093\times 10^{8}}{88796}=8569.4175 \\
ICER(Z_2)=\dfrac{1.6452\times 10^{9}-7.6093\times 10^{8}}{88848-88796}=17005192\\
\end{array}
\end{equation}

The comparison between ICER($Z_3$) and ICER($Z_2$) shows a cost saving of \$ 8569.4175 for strategy $Z_3$ over strategy $Z_2$. That is, strategy $Z_2$ is more costly and less effective than strategy $Z_3$. Therefore, strategy $Z_2$, the strongly dominated is excluded.

We then compare strategy $Z_3$ with $Z_1$. From table~\ref{CostTable}, we have:
\begin{table}[H]
	\begin{center}
		\caption{Strategy $Z_3$ vs  Strategy $Z_1$.\label{Z3vsZ1}}
		\begin{tabular}{ccc}
			\hline
			Strategies&Total infection averted  &Total cost ($\$$)\\
			\hline
			$Z_3$&88796&$7.6093\times 10^{8}$\\
			\hline
			$Z_1$&88886&$7.6081\times 10^{8}$\\
			\hline
		\end{tabular}
	\end{center}
\end{table}

This leads to the following values for the ICER,
\begin{equation}
\label{costZ3--costZ1}
\begin{array}{l}
ICER(Z_3)=\dfrac{7.6093\times 10^{8}}{88796}=8569.4175 \\
ICER(Z_1)=\dfrac{7.6081\times 10^{8}-7.6093\times 10^{8}}{88886-88796}=-1333.3333 \\
\end{array}
\end{equation}
The comparison between ICER($Z_3$) and ICER($Z_1$) shows a cost saving of \$ 1333.3333 for strategy $Z_1$ over strategy $Z_3$. So the strategy $Z_3$ is "strongly dominated". That is, strategy $Z_3$ is more costly and less effective than strategy $Z_1$. Therefore, strategy $Z_3$, the strongly dominated is excluded.

We then compare strategy $Z_1$ with $Z$. From table~\ref{CostTable}, we have
\begin{table}[H]
	\begin{center}
		\caption{Strategy $Z_1$ vs  Strategy $Z$.\label{Z1vsZ}}
		\begin{tabular}{ccc}
			\hline
			Strategies&Total infection averted  &Total cost ($\$$)\\
			\hline
			$Z_1$&88886&$7.6081\times 10^{8}$\\
			\hline
			$Z$&88886&$7.6081\times 10^{8}$\\
			\hline
		\end{tabular}
	\end{center}
\end{table}
From Table \ref{Z1vsZ}, it follows that strategy $Z_3$ is equivalent in term of efficiency and cost at strategy $Z$.  

With these results, we conclude that the strategies $Z$ (combination of the five control) and $Z_3$ (vaccination $u_1$ combined with personal protection $u_2$, the treatment of individuals with clinical signs of the disease $u_3$, killing adult vectors with adulticide, $u_4$) are most cost-effective that all the strategies studied in this work.

\section{Conclusion}
In this paper, we derived and analysed a model for the control of arboviral diseases with non linear form of infection and complete stage structured model for vectors, and which takes into account a vaccination with waning immunity, treatment, individual protection and vector control strategies (adult vectors, eggs and larvae reduction strategies). 

We have begun by calculate the net reproductive number $\mathcal{N}$ and the basic reproduction number $\mathcal R_0$, of the basic model (the model without control), and investigate the existence and stability of equilibria. The stability analysis revealed that for $\mathcal{N}\leq 1$, the trivial equilibrium is globally asymptotically stable. When $\mathcal{N}>1$ and $\mathcal R_0<1$, the disease--free equilibrium is locally asymptotically stable. We have found that the model exhibits backward bifurcation. The epidemiological implication of this phenomenon is that for effective eradication and control of diseases, $\mathcal R_0$ should be less than a critical values less than one. We have explicitly derived threshold conditions for saddle--node bifurcation in term of the transmission rate, $\beta_{hv}$, as well as the basic reproduction number. Then, we have proved, that the disease--induced death is the principal cause of the backward bifurcation phenomenon in model. 

Using data from literature related to the transmission dynamics of dengue fever, we also estimated the probability  that the model predicts the existence of multiple endemic equilibrium and of the likely stability of the disease--free equilibrium point, through Latin Hypercube Sampling (LHS).
The result showed that the model is in an endemic state, since the mean of $\mathcal R_0$ is greater than unity. Then, using global sensitivity analysis, we have computed the Partial Rank Correlation Coefficients between $\mathcal R_0$ and each parameter of the model. This analysis showed that the basic reproduction number is sensitive to changes in the parameters $\beta_{vh}$, the probability of transmission of infection from an infected human to a susceptible vector, $\beta_{hv}$, the probability of transmission of infection from an infected vector to a susceptible human, $a$, the average number of bites, $\theta$, the maturation rate from pupae to adult, $\mu_v$, the natural mortality rate of vector, $\Lambda_h$, the recruitment rate of humans and $l$, the transfer rate from larvae to pupae, which suggested that the control of the epidemic of arboviral diseases pass through a combination of immunization against arbovirus (vaccination of susceptible humans), individual protection against vector bites, treatment of infected human, vector control through chemical interventions (adulticide and larvicide).

We then considered five time dependent controls as a way out, to ensure the eradication of the disease in a finite time. We performed optimal control analysis of the model. In this light, we addressed the optimal control by deriving and analysing the conditions for optimal eradication of the disease and in a situation where eradication is impossible or of less benefit compared with the cost of intervention, we also derived and analysed the necessary conditions for optimal control of the disease. 

From the numerical results and efficiency analysis, as well as, cost--effectiveness analysis, we concluded that the optimal strategy to effectively control arboviral diseases is the combination of vaccination, individual protection, treatment, and other mechanisms of vector control (by chemical intervention--the adulticides). 
However this conclusion must be taken with caution because of the uncertainties around the parameter values and to the budget/resource limitation. It is also important to note that in most of the work which speak of the optimal control of infectious diseases (see e.g.~\cite{Adams2004,bbMBE2011,Zaman2008,Jung2002,Yusuf2012}), and particularly the arboviral diseases
~\cite{Aldila2013,BlaynehaetAl,moaaHee2012,HelenaSofiaRodrigues2014,Dias2015}, cost effectiveness analysis, to our knowledge, is not did by the authors. This therefore represents a contribution to the study of optimal control models of arboviral diseases.	

In addition, the utilization of a vaccine of small efficacy could have a negative impact on the health of the population. Indeed, for the particular case of dengue, the fact that sequential infections with different strains can cause severe forms of the disease must be taken into account. For instance, it is not currently known if a vaccinated individual, for which the efficacy for a given strain is small, can develop a severe form of the disease when coming into contact with such a strain. Also, its efficiency is higher in children 9-16 years (two thirds are immune) and in individuals who have already been infected. The vaccine appears to contrast against-productive in younger children without the researchers knowing why. The results of clinical trials - which involved more than 40,000 volunteers-- were therefore not lifted all the uncertainties about the impact of the vaccine~\cite{LeMonde2015}. Therefore, pending the completion of Phase III trials on the efficacy of the vaccine against dengue (Dengvaxia$^{\circledR}$), and therefore its acceptance by public health organizations such as WHO and the Centre for Disease Control (CDC), it is important to focus on other control mechanisms.

All simulated intervention combinations can be considered cost-effective in the context of available resources for health in countries affected by arboviruses. 
These results have the potential to help managers control programs against arbovirus infections in high endemicity countries by modifying the implementation of current interventions, or by adding new control mechanisms.

\section*{Acknowledgments}
The first author (Hamadjam ABBOUBAKAR) thanks the  Direction of UIT of Ngaoundere for the financial help granted under research missions in the year 2015. The first author also thanks the Department of Mathematics of ENS of Yaounde for their hospitality during the research visit where the work was initiated.
\appendix
\section{Proof of Theorem~\ref{th3Basic}}
\label{AppBasic1}
The Jacobian matrix of $f=\left(\dot{S}_h,\dot{E}_h,\dot{I}_h,\dot{R}_h,\dot{S}_v,\dot{E}_v,\dot{I}_v,\dot{E},\dot{L},
\dot{P}\right)^{T}$ at the Trivial equilibrium is given by
\begin{equation}
\label{Jacm1Ar3ap}
Df(\mathcal{E}_{0})=\left( \begin{array}{cccccccccc}
-\mu_h&0&0&0&0&-a\beta_{hv}\eta_v&-a\beta_{hv}&0&0&0\\
0&-k_3&0&0&0&a\beta_{hv}\eta_v&a\beta_{hv}&0&0&0\\
0&\gamma_h&-k_4&0&0&0&0&0&0&0\\
0&0&\sigma&-\mu_h&0&0&0&0&0&0\\
0&0&0&0&-\mu_v&0&0&0&0&\theta\\
0&0&0&0&0&-k_9&0&0&0&0\\
0&0&0&0&0&\gamma_v&-\mu_v&0&0&0\\
0&0&0&0&\mu_b&\mu_b&\mu_b&-k_5&0&0\\
0&0&0&0&0&0&0&s&-k_6&0\\
0&0&0&0&0&0&0&0&l&-k_7\\
\end{array} \right) .
\end{equation} 

The characteristic polynomial of $Df(\mathcal{E}_{0})$ is given by:
\[
P(\lambda)=(\lambda-\mu_h)^{2}(\lambda-k_3)(\lambda-k_4)(\lambda-k_9)(\lambda-\mu_v)\phi_1(\lambda)
\]
where\\
$
\phi_1(\lambda)=\lambda^{4}+A_1\lambda^{3}+A_2\lambda^{2}+A_3\lambda+A_4,
$
with
\[
\begin{array}{l}
A_1=\mu_{v}+k_{7}+k_{6}+k_{5},\;\;
A_2=\left(k_{7}+k_{6}+k_{5}\right)\mu_{v}+\left(k_{6}+k_{5}\right)k_{7}+k_{5}k_{6},\\
A_3=\left(\left(k_{6}+k_{5}\right)k_{7}+k_{5}k_{6}\right)\mu_{v}+k_{5}k_{6}k_{7},\;\; 
A_4=k_{5}k_{6}k_{7}\mu_{v}\left(1-\mathcal N\right).
\end{array}
\]
The roots of $P(\lambda)$ are $\lambda_1=-\mu_h$, $\lambda_1=-k_3$, $\lambda_2=-k_4$, $\lambda_3=-\mu_v$, $\lambda_4=-k_9$, and the others roots are the roots of $\phi_1(\lambda)$. Since $\mathcal{N}<1$, it is clear that all coefficients of $\phi_1(\lambda)$ are always positive. Now we just have to verify that the Routh--Hurwitz criterion holds for polynomial $\phi_2(\lambda)$.
To this aim, setting 
$H_1=A_1$, $H_2=\begin{vmatrix}
A_1&1\\
A_3&A_2
\end{vmatrix}$,
$H_3=\begin{vmatrix}
A_1&1&0\\
A_3&A_2&A_1\\
0&A_4&A_3
\end{vmatrix}$,
$H_4=\begin{vmatrix}
A_1&1&0&0\\
A_3&A_2&A_1&1\\
0&A_4&A_3&A_2\\
0&0&0&A_4
\end{vmatrix}=A_4H_3$.\\
The Routh-Hurwitz criterion of stability of the trivial equilibrium $\mathcal{E}^{0}$ is given by
\begin{equation}
\label{routh-Hurwith}
\left\lbrace 
\begin{array}{c}
H_1>0\\
H_2>0\\
H_3>0\\
H_4>0
\end{array}\right.\Leftrightarrow
\left\lbrace 
\begin{array}{c}
H_1>0\\
H_2>0\\
H_3>0\\
A_4>0
\end{array}\right.
\end{equation}
We have $H_1=A_1>0$,
\[
\begin{split}
H_2&=A_1A_2-A_3\\
&=\left(k_{7}+k_{6}+k_{5}\right)\,\mu_{v}^2+\left(k_{7}^2+\left(2\,k
_{6}+2\,k_{5}\right)\,k_{7}+k_{6}^2+2\,k_{5}\,k_{6}+k_{5}^2\right)\,
\mu_{v}+\left(k_{6}+k_{5}\right)\,k_{7}^2\\
&+\left(k_{6}+k_{5}^2\right)^2k_{7}
+k_{5}\,k_{6}^2+k_{5}^2\,k_{6},
\end{split}
\]
{\scriptsize
	\[
	\begin{split}
	H_3&=A_1A_2A_3-A^{2}_1A_4-A^{2}_3\\
	&=\left(\left(k_{6}+k_{5}\right)k_{7}^2+\left(k_{6}^2+2k_{5}k_{
		6}+k_{5}^2\right)k_{7}+k_{5}k_{6}^2+k_{5}^2k_{6}\right)\mu_{
		v}^3+\left(\mu_{b}ls\theta+\left(k_{6}+k_{5}\right)k_{7}^
	3+\left(2k_{6}^2+4k_{5}k_{6}+2k_{5}^2\right)k_{7}^2\right.\\
	&\left.+\left(
	k_{6}^3+4k_{5}k_{6}^2+4k_{5}^2k_{6}+k_{5}^3\right)k_{7}+k
	_{5}k_{6}^3+2k_{5}^2k_{6}^2+k_{5}^3k_{6}\right)\mu_{v}^2+
	\left(\left(2k_{7}+2k_{6}+2k_{5}\right)\mu_{b}ls
	\theta+\left(k_{6}^2+2k_{5}k_{6}+k_{5}^2\right)k_{7}^3\right.\\
	&\left.+
	\left(k_{6}^3+4k_{5}k_{6}^2+4k_{5}^2k_{6}+k_{5}^3\right)k
	_{7}^2+\left(2k_{5}k_{6}^3+4k_{5}^2k_{6}^2+2k_{5}^3k_{6}
	\right)k_{7}+k_{5}^2k_{6}^3+k_{5}^3k_{6}^2\right)\mu_{v}\\
	&+\left(k_{7}^2+\left(2k_{6}+2k_{5}\right)k_{7}+k_{6}^2+2k_{5}
	k_{6}+k_{5}^2\right)\mu_{b}ls\theta+\left(k_{5}k_{6}^
	2+k_{5}^2k_{6}\right)k_{7}^3+\left(k_{5}k_{6}^3+2k_{5}^2k
	_{6}^2+k_{5}^3k_{6}\right)k_{7}^2+\left(k_{5}^2k_{6}^3+k_{5}^3
	k_{6}^2\right)k_{7}
	\end{split}
	\]
}
We always have $H_1>0$, $H_2>0$, $H_3>0$ and  $H_4>0$ if $\mathcal{N}<1$. Thus, the trivial equilibrium 
$\mathcal{E}_{0}$ is locally asymptotically stable whenever $\mathcal{N}<1$.

We assume the net reproductive number $\mathcal{N}>1$. Following the procedure and the
notation in \cite{vawa02}, we may obtain the basic reproduction number $\mathcal R_0$ as the dominant eigenvalue of the \emph{next--generation matrix} \cite{dihe,vawa02}. Observe that model~\eqref{BasicModel} has four infected populations, namely $E_h$, $I_h$, $E_v$ and $I_v$. It follows that the matrices $F$ and $V$ defined in \cite{vawa02}, which take into account the new infection terms and remaining transfer terms, respectively, are given by\\
$
F=\left( \begin{array}{cccc}
0&0&\beta_{hv}\eta_v&\beta_{hv}\\
0&0&0&0\\
\dfrac{\beta_{vh}\eta_vN^{0}_v}{N^{0}_h}&\dfrac{\beta_{vh}N^{0}_v}{N^{0}_h}&0&0\\
0&0&0&0
\end{array}\right)$ and
$V=\left(\begin{array}{cccc}
k_3&0&0&0\\
-\gamma_h&k_4&0&0\\
0&0&k_9&0\\
0&0&-\gamma_v&k_8
\end{array}\right).
$

The dominant eigenvalue of the next--generation matrix $FV^{-1}$ is given by \eqref{R0ar3}.
The local stability of the disease--free equilibrium $\mathcal{E}_{1}$ is a direct consequence of Theorem 2 in \cite{vawa02}. This ends the proof.

\section{Proof of Theorem~\ref{GasTEBasic}}
\label{AppBasic2}
Setting $Y=X-TE$ with $X=(S_h,E_h,I_h,R_h,S_v,E_v,I_v,E,L,P)^{T}$,
$A_{88}=\left(k_5+\mu_b\dfrac{S_v+E_v+I_v}{K_E}\right)$, and $A_{99}=\left(k_6+s\dfrac{E}{\Gamma_L}\right)$.
we can rewrite \eqref{AR3} in the following manner
\begin{equation}
\label{model2Ar2} \dfrac{dY}{dt}=\mathcal{B}(Y)Y
\end{equation}
where
{\footnotesize 
	$$ \mathcal{B}(Y)=\left( \begin{array}{cccccccccc}
	-(\lambda_h+\mu_h)&0&0&0&0&-\dfrac{a\beta_{hv}\eta_vS^{0}_h}{N_h}&-\dfrac{a\beta_{hv}S^{0}_h}{N_h}&0&0&0\\
	\lambda_h&-k_3&0&0&0&\dfrac{a\beta_{hv}\eta_vS^{0}_h}{N_h}&\dfrac{a\beta_{hv}S^{0}_h}{N_h}&0&0&0\\
	0&\gamma_h&-k_4&0&0&0&0&0&0&0\\
	0&0&\sigma&-\mu_h&0&0&0&0&0&0\\
	0&0&0&0&-(\lambda_v+\mu_v)&0&0&0&0&\theta\\
	0&0&0&0&\lambda_v&-k_9&0&0&0&0\\
	0&0&0&0&0&\gamma_v&-\mu_v&0&0&0\\
	0&0&0&0&\mu_b&\mu_b&\mu_b&-A_{88}&0&0\\
	0&0&0&0&0&0&0&s&-A_{99}&0\\
	0&0&0&0&0&0&0&0&l&-k_7\\
	\end{array}\right).
	$$
}
It is clear that $Y=(0,0,0,0,0,0,0,0,0,0)$ is the only equilibrium. 
Then it suffices to consider the following Lyapunov function $\mathcal{L}(Y)=<g,Y>$ were
{\footnotesize $g=\left(1,1,1,1,1,1,1,\dfrac{k_8}{\mu_b},\dfrac{k_5k_8}{\mu_bs},\dfrac{k_5k_6k_8}{\mu_bsl}\right)$}.
Straightforward computations lead that
{\small
	\[
	\begin{split}
	\dot{\mathcal{L}}(Y)&=<g,\dot{Y}>\overset{\mathrm{def}}{=}
	<g,\mathcal{B}(Y)Y>\\
	&=-\mu_hY_1-\mu_hY_2-(\mu_h+\delta)Y_3-\mu_hY_4
	-\dfrac{k_8}{K_E}(Y_5+Y_6+Y_7)-\dfrac{k_5k_8}{\mu_bK_L}Y_8Y_{9}+\theta\left(1-\dfrac{1}{\mathcal{N}}\right)Y_{10}.
	\end{split}
	\]
}
We have $\dot{\mathcal{L}}(Y)<0$ if $\mathcal{N}\leq 1$ and $\dot{\mathcal{L}}(Y)=0$ if $Y_i=0$, $i=1,2,\hdots, 10$ (i.e $S_h=S^{0}_h$ and $E_h=I_h=R_h=S_v=E_v=I_v=E=L=P=0$). Moreover, the maximal invariant set contained in $\left\lbrace \mathcal{L}\vert \dot{\mathcal{L}}(Y)=0\right\rbrace $ is $\left\lbrace(0,0,0,0,0,0,0,0,0,0,0)\right\rbrace$. Thus, from Lyapunov theory, we deduce that $\left\lbrace(0,0,0,0,0,0,0,0,0,0,0)\right\rbrace$ and thus, $\mathcal{E}_0$, is GAS if and only if $\mathcal{N}\leq 1$.

\section{Proof of Theorem~\ref{EEBasic}}
\label{AppBasic3}
\begin{proof}
In order to determine the existence of endemic equilibria, i.e., equilibria with all positive components, say
\[
\mathcal{E}^{**}=\left(S^{*}_h,V^{*}_h,E^{*}_h,I^{*}_h,R^{*}_h,S^{*}_v,E^{*}_v,I^{*}_v,E,L,P\right),
\]
we have to look for the solution of the algebraic system of equations obtained by equating
the right sides of system~\eqref{BasicModel} to zero.

Solving the equations in the model~\eqref{BasicModel} in terms of $\lambda^{*}_h$ and $\lambda^{*}_v$, gives
\begin{equation}
\label{EEhBasic}
\begin{array}{l}
S^{*}_h=\dfrac{\Lambda_h}{\mu_h+\lambda^{*}_{h}},\;\;\;
E^{*}_h=\dfrac{\lambda^{*}_{h}S^{*}_h}{k_3},\;\;I^{*}_h=\dfrac{\gamma_h\lambda^{*}_{h}S^{*}_h}{k_3k_4},\;\;R^{*}_h=\dfrac{\sigma\gamma_h\lambda^{*}_{h}S^{*}_h}{\mu_hk_3k_4},
\end{array}
\end{equation}
and
\begin{equation}
\label{EEvBasic}
\begin{array}{l}
S^{*}_v=\dfrac{\theta P}{(\lambda^{*}_v+k_8)},\,\;\; E^{*}_v=\dfrac{\theta P\lambda^{*}_v}{k_9(\lambda^{*}_v+k_8)},\;\;I^{*}_v=\dfrac{\gamma_v\theta P\lambda^{*}_v}{k_8k_9(\lambda^{*}_v+k_8)},\\
E=\dfrac{\mu_b\theta \Gamma_E P}{(k_5k_8\Gamma_E+\mu_b\theta P)},\;\;
L=\dfrac{\mu_b\theta s\Gamma_E\Gamma_LP}{k_6\Gamma_L(k_5k_8\Gamma_E+\mu_b\theta P)+s\mu_b\theta \Gamma_E P},
\end{array}
\end{equation}
where $P$ is solution of the following equation
\begin{equation}
\label{eggs}
f(P)=-k_7P\left[\mu_b\theta(s\Gamma_E+k_6\Gamma_L)P-k_5k_6k_8\Gamma_E\Gamma_L(\mathcal{N}-1)\right]=0
\end{equation}
A direct resolution of the above equation give $P=0$ or $P=\dfrac{k_5k_6k_8\Gamma_E\Gamma_L(\mathcal{N}-1)}{\mu_b\theta(s\Gamma_E+k_6\Gamma_L)}$.

Note that $P=0$ corresponds to the trivial equilibrium $\mathcal{E}_{0}$. Now we consider $P>0$ i.e. $\mathcal{N}>1$. Substituting \eqref{EEhBasic} and \eqref{EEvBasic} into the expression of $\lambda^{*}_h$ and $\lambda^{*}_v$ and simplifying, lead the non-zero equilibria of the basic model~\eqref{BasicModel} satisfy the quadratic equation
\begin{equation}
\label{eqEEsansBasic}
k_{9}\mu_{b}\Lambda_{h}(s\Gamma_E+k_6\Gamma_L)\left[ d_2(\lambda^{*}_h)^{2}+d_1\lambda^{*}_h+d_0\right] =0,
\end{equation}
where $d_i$, $i=0,1,2$, are given by 
\begin{subequations}
	\begin{align}
	d_2&=-k_{2}\left[k_{10}a\mu_{h}\beta_{vh}+k_{2}k_{8}\right]<0, \\
	d_1&=k^{2}_{3}k^{2}_{4}k_{8}\mu_{h}(\mathcal R_{0}^{2}-\mathcal R^{2}_{c}),\\
	d_0 &=k^{2}_{3}k^{2}_{4}k_{8}\mu_{h}^2\left(\mathcal R_{0}^{2}-1\right).
	\end{align}
\end{subequations}
and $\mathcal R_0$ and $\mathcal R_c$ are given by \eqref{R0ar3} and \eqref{R_car3opc}, respectively.

This equation may be simply analyzed through the Descartes' rule of signs. First of all,
note that $d_2$ is negative. Therefore the following cases are possible:\\
\begin{enumerate}
	\item There is a unique endemic equilibrium if $d_0>0$;
	\item There is a unique endemic equilibrium if
	\begin{equation}
	\label{oneee1ar3opc} \left(d_1>0\;\; \text{ and }\;\; d_0=0\right)\;\; \text{ or }\;\; (d_1 > 0\;\; \text{ and } \;\; d_0 <0\;\; \text{ and }\;\; d^{2}_1-4d_2d_0=0);
	\end{equation}
	\item There are two endemic equilibria if
	\begin{equation}
	\label{oneee2} d_1>0\;\; \text{ and }\;\; d_0<0\;\; \text{ and }\;\; d^{2}_1-4d_2d_0>0;
	\end{equation}
	\item There are no endemic equilibria otherwise.
\end{enumerate}

We observe $d_0>0$ is equivalent to $\mathcal R_0>1$ so statement (i) of Theorem~\ref{EEBasic} is equivalent to point (\emph{a}).

When $\mathcal R_0=1$, $d_0=0$. We observe that $d_1>0$ is equivalent to $\mathcal R_c<\mathcal R_0$. Therefore, when $\mathcal R_0=1$ and $\mathcal R_c<1$, $d_0=0$ and $d_1>0$,  so statement (ii) \emph{a}) of Theorem~\ref{EEBasic} follows from statement (\emph{b}) above. Since the condition $d_0 = 0$ does not appear elsewhere in statements (\emph{a}), (\emph{b}), or (\emph{c}) above, statement (ii) \emph{b}) of Theorem~\ref{EEBasic} follows from statement (\emph{d}) above.

When $\mathcal R_0<1$, $d_0<0$, and when $\mathcal R_c<\mathcal R_0$, $d_1>0$. We also note that for $\mathcal R_0<1$, when $d_1>0$, $\psi\leq 0$ because $\psi>0$ is equivalent to $d_1<0$. Indeed,
{\scriptsize
\begin{equation}
\begin{split}
\psi>0&\Longleftrightarrow k_{10}a\mu_{h}\beta_{vh}>\delta\gamma_hk_{8}\\
&\Longleftrightarrow k_{10}a\mu_{h}\beta_{vh}k_9+2k_{8}k_{9}k_{2}>\delta\gamma_hk_{8}k_9+2k_{8}k_{9}k_{2}\\
&\Longleftrightarrow a^{2}\beta_{hv}\beta_{vh}k_{3}k_{4}k_{10}k_{11}\mu_{h}N^{0}_{v} -k_{3}k_{4}\mu_{h}N^{0}_{h}(k_{10}a\mu_{h}\beta_{vh}k_9+2k_{8}k_{9}k_{2})<
k_{3}k_{4}k_{8}k_9\mu_{h}N^{0}_{h}\left[ (\mathcal R^{2}_0-1)k_3k_4 -k_3k_4+\delta\gamma_h)\right] \\
&\Longleftrightarrow a^{2}\beta_{hv}\beta_{vh}k_{3}k_{4}k_{10}k_{11}\mu_{h}N^{0}_{v} -k_{3}k_{4}\mu_{h}N^{0}_{h}(k_{10}a\mu_{h}\beta_{vh}k_9+2k_{8}k_{9}k_{2})<
k_{3}k_{4}k_{8}k_9\mu_{h}N^{0}_{h}\left[ (\mathcal R^{2}_0-1)k_3k_4 -k_{2})\right] \\
&\Longleftrightarrow d_1<k_{3}k_{4}k_{8}\mu_{h}\left[ (\mathcal R^{2}_0-1)k_3k_4 -k_{2})\right]<0,\,{\rm since}\, \mathcal R_0<1.\\
\end{split}
\end{equation}
}

Consequently, we show that $d^{2}_1-4d_2d_0=0$ is equivalent to,
\begin{equation}
\label{Delta_Rar3opc} \rho_2R^{4}_0+\rho_1R^{2}_0+\rho_0=0,
\end{equation}
where
\begin{subequations}
	\begin{align}
	\rho_2 &=k_{3}^4k_{4}^4k_{8}^2\mu_{h}^2 , \\
	\label{B0}
	\rho_1 &=2k_{3}^2k_{4}^2k_{8}\mu_{h}^2
	\left[k_{2}(k_{10}a\mu_{h}\beta_{vh}-k_{8}\delta\gamma_h)
	-(k_{10}a\mu_{h}\beta_{vh}+k_{8}k_{2})\delta\gamma_h\right] , \\
	\rho_0 &=k_{3}^2k_{4}^2k_{10}^2a^2\mu_{h}^4\beta_{vh}^2 .
	\end{align}
\end{subequations}

We again use Descartes' rule of signs to analyse equation \eqref{Delta_Rar3opc}. The discriminant of \eqref{Delta_Rar3opc} is $\Delta_r=\rho^{2}_1-4\rho_2\rho_0$, and can be written
\begin{equation*}
\Delta_r=-16k_{3}^4k_{4}^4k_{8}^2k_{2}\mu_{h}^4\delta\gamma_h
\left(k_{10}a\mu_{h}\beta_{vh}+k_{8}k_{2}\right)\left(k_{10}a\mu_{h}\beta_{vh}
-k_{8}\delta\gamma_h\right)
\end{equation*}

Since $\rho_2>0$ and $\rho_0>0$, equation \eqref{Delta_Rar3opc} allows real positive solutions if and only if $\rho_1<0$ and $\Delta_r\geq 0$. Now, we express the obtained inequalities in terms of the quantities \eqref{psiar3opc}--\eqref{R_2bar3opc}. To this aim, we note that $\Delta_r\geq 0$ is equivalent to $\psi\leq 0$. From the definition of $\psi$ \eqref{psiar3opc} and $\rho_1$ \eqref{B0}, this implies that $\rho_1<0$. Therefore, equation \eqref{Delta_Rar3opc} has exactly two positive solutions 
given by \eqref{R_1bar3opc} and \eqref{R_2bar3opc}. Therefore, statement (iii) \emph{b}) follows from statement (\emph{b}) above. 

Similarly, $d^{2}_1-4d_2d_0>0$, with $d_1>0$ and $d_0 < 0$ written in terms of the
basic reproduction number, is equivalent to
\begin{equation}
\label{conditionOf2EEar3opc} \mathcal R_0<\mathcal R_{1b} \quad \text{ or } \quad \mathcal R_0>\mathcal R_{2b},
\end{equation}
so statement (iii) \emph{a}) follows from statement (\emph{c}) above. 

Finally, statement (iii) \emph{c}) then follows from statement (\emph{d}) above.  Thus Theorem \ref{EEBasic} is established.\hfill
\end{proof}

\section{Derivation of formula \eqref{ccsBasicModel2} }
\label{ccBasicModeldetails}
The local bifurcation analysis near the bifurcation point ($\beta_{hv}=\beta^{*}_{hv}$) is then determined by the signs of two associated constants, denoted by $\mathcal{A}_1$ and $\mathcal{A}_2$, defined by
\begin{equation}
\mathcal{A}_1=\sum\limits_{k,i,j=1}^{10}v_kw_iw_j\dfrac{\partial^{2}f_k(0,0)}{\partial x_i\partial x_j}\qquad and\qquad
\mathcal{A}_2=\sum\limits_{k,i=1}^{10}v_kw_i\dfrac{\partial^{2}f_k(0,0)}{\partial x_i\partial \phi}
\end{equation}
with $\phi=\beta_{hv}-\beta^{*}_{hv}$. It is important to note that in $f_k(0,0)$, the first zero corresponds to the disease--free equilibrium, $\mathcal{E}_{1}$, for the system \eqref{BasicModel}. Since $\beta_{hv}=\beta^{*}_{hv}$ is the bifurcation parameter, it follows from $\phi=\beta_{hv}-\beta^{*}_{hv}$ that $\phi=0$ when $\beta_{hv}=\beta^{*}_{hv}$ which is the second component in $f_k(0,0)$.

Using Eqs.~\eqref{LeftVectorBasicModel} and \eqref{RightVectorBasicModel} in Eq.~\eqref{ccsBasicModel}, we obtain	
{\footnotesize
	\begin{equation}
	\mathcal{A}_1=v_2\sum\limits_{i,j=1}^{10}w_iw_j\dfrac{\partial^{2}f_2(0,0)}{\partial x_i\partial x_j}
	+v_6\sum\limits_{i,j=1}^{10}w_iw_j\dfrac{\partial^{2}f_6(0,0)}{\partial x_i\partial x_j}\quad and \quad
	\mathcal{A}_2=v_2\sum\limits_{i=1}^{10}w_i\dfrac{\partial^{2}f_2(0,0)}{\partial x_i\partial \phi}.
	\end{equation}
}
Let
{\footnotesize
	\[
	\begin{split}
	\mathcal{A}^{(1)}_1&=\sum\limits_{i,j=1}^{10}w_iw_j\dfrac{\partial^{2}f_2(0,0)}{\partial x_i\partial x_j}\\
	&=w_1\sum\limits_{i,j=1}^{10}w_j\dfrac{\partial^{2}f_2(0,0)}{\partial x_1\partial x_j}
	+w_2\sum\limits_{i,j=1}^{10}w_j\dfrac{\partial^{2}f_2(0,0)}{\partial x_2\partial x_j}
	+w_3\sum\limits_{i,j=1}^{10}w_j\dfrac{\partial^{2}f_2(0,0)}{\partial x_3\partial x_j}
	+w_4\sum\limits_{i,j=1}^{10}w_j\dfrac{\partial^{2}f_2(0,0)}{\partial x_4\partial x_j}\\
	&+w_5\sum\limits_{i,j=1}^{10}w_j\dfrac{\partial^{2}f_2(0,0)}{\partial x_5\partial x_j}
	+w_6\sum\limits_{i,j=1}^{10}w_j\dfrac{\partial^{2}f_2(0,0)}{\partial x_6\partial x_j}
	+w_7\sum\limits_{i,j=1}^{10}w_j\dfrac{\partial^{2}f_2(0,0)}{\partial x_7\partial x_j}\\
	&=w_2\left(w_6\dfrac{\partial^{2}f_2}{\partial E_h\partial E_v}(0,0)
	+w_7\dfrac{\partial^{2}f_2}{\partial E_h\partial I_v}(0,0)\right) 
	+w_3\left(w_6\dfrac{\partial^{2}f_2}{\partial I_h\partial E_v}(0,0)
	+w_7\dfrac{\partial^{2}f_2}{\partial I_h\partial I_v}(0,0)\right)\\
	&+w_4\left(w_6\dfrac{\partial^{2}f_2}{\partial R_h\partial E_v}(0,0)
	+w_7\dfrac{\partial^{2}f_2}{\partial R_h\partial I_v}(0,0)\right)
	+w_6\left(w_2\dfrac{\partial^{2} f_2}{\partial E_v\partial E_h}
	+w_3\dfrac{\partial^{2} f_2}{\partial E_v\partial I_h}
	+w_4\dfrac{\partial^{2} f_2}{\partial E_v\partial R_h} \right) \\
	&+w_7\left(w_2\dfrac{\partial^{2} f_2}{\partial I_v\partial E_h}
	+w_3\dfrac{\partial^{2} f_2}{\partial I_v\partial I_h}
	+w_4\dfrac{\partial^{2} f_2}{\partial I_v\partial R_h} \right)\\
	&=w_2\left(-\dfrac{a\beta^{*}_{hv}\eta_v}{N^{0}_h}w_6
	-\dfrac{a\beta^{*}_{hv}}{N^{0}_h}w_7\right) 
	+w_3\left(-\dfrac{a\beta^{*}_{hv}\eta_v}{N^{0}_h}w_6
	-\dfrac{a\beta^{*}_{hv}}{N^{0}_h}w_7\right)\\
	&+w_4\left(-\dfrac{a\beta^{*}_{hv}\eta_v}{N^{0}_h}w_6
	-\dfrac{a\beta^{*}_{hv}}{N^{0}_h}w_7\right)
	+w_6\left(-\dfrac{a\beta^{*}_{hv}\eta_v}{N^{0}_h}w_2
	-\dfrac{a\beta^{*}_{hv}\eta_v}{N^{0}_h}w_3
	-\dfrac{a\beta^{*}_{hv}\eta_v}{N^{0}_h}w_4 \right) \\
	&+w_7\left(-\dfrac{a\beta^{*}_{hv}}{N^{0}_h}w_2
	-\dfrac{a\beta^{*}_{hv}}{N^{0}_h}w_3
	-\dfrac{a\beta^{*}_{hv}}{N^{0}_h}w_4\right)\\
	&=-\dfrac{a\beta^{*}_{hv}}{N^{0}_h}\left(\eta_vw_6+w_7\right)(w_2+w_3+w_4)
	-\dfrac{a\beta^{*}_{hv}}{N^{0}_h}\left(w_2+w_3+w_4 \right)(\eta_vw_6+w_7)\\
	&=-2\dfrac{a\beta^{*}_{hv}}{N^{0}_h}\left(\eta_vw_6+w_7\right)(w_2+w_3+w_4),
	\end{split}
	\]
}
and
{\scriptsize
	
	\[
	\begin{split}
	\mathcal{A}^{(2)}_1&=\sum\limits_{i,j=1}^{10}w_iw_j\dfrac{\partial^{2}f_6(0,0)}{\partial x_1\partial x_j}\\
	&=w_1\sum\limits_{i,j=1}^{10}w_j\dfrac{\partial^{2}f_6(0,0)}{\partial x_1\partial x_j}
	+w_2\sum\limits_{i,j=1}^{10}w_j\dfrac{\partial^{2}f_6(0,0)}{\partial x_2\partial x_j}
	+w_3\sum\limits_{i,j=1}^{10}w_j\dfrac{\partial^{2}f_6(0,0)}{\partial x_3\partial x_j}
	+w_4\sum\limits_{i,j=1}^{10}w_j\dfrac{\partial^{2}f_6(0,0)}{\partial x_4\partial x_j}\\
	&+w_5\sum\limits_{i,j=1}^{10}w_j\dfrac{\partial^{2}f_6(0,0)}{\partial x_5\partial x_j}
	+w_6\sum\limits_{i,j=1}^{10}w_j\dfrac{\partial^{2}f_6(0,0)}{\partial x_6\partial x_j}
	+w_7\sum\limits_{i,j=1}^{10}w_j\dfrac{\partial^{2}f_6(0,0)}{\partial x_7\partial x_j}\\
	&=w_1\left(w_2\dfrac{\partial f_6}{\partial S_h\partial E_h}(0,0)+ 
	w_3\dfrac{\partial f_6}{\partial S_h\partial I_h}(0,0)\right) \\
	&+w_2\left(w_1\dfrac{\partial f_6}{\partial E_h\partial S_h}(0,0)+ 
	w_2\dfrac{\partial f_6}{\partial E^{2}_h}(0,0)
	+w_3\dfrac{\partial f_6}{\partial E_h\partial I_h}(0,0)
	+w_4\dfrac{\partial f_6}{\partial E_h\partial R_h}(0,0)
	+w_5\dfrac{\partial f_6}{\partial E_h\partial S_v}(0,0)\right)\\
	&+w_3\left(w_1\dfrac{\partial f_6}{\partial I_h\partial S_h}(0,0)+ 
	w_2\dfrac{\partial f_6}{\partial I_h\partial E_h}(0,0)
	+w_3\dfrac{\partial f_6}{\partial I^{2}_h}(0,0)
	+w_4\dfrac{\partial f_6}{\partial I_h\partial R_h}(0,0)
	+w_5\dfrac{\partial f_6}{\partial I_h\partial S_v}(0,0)\right)\\
	&+w_4\left(w_2\dfrac{\partial f_6}{\partial R_h\partial E_h}(0,0)+ 
	w_3\dfrac{\partial f_6}{\partial R_h\partial I_h}(0,0)\right)
	+w_5\left(w_2\dfrac{\partial f_6}{\partial S_v\partial E_h}(0,0)+ 
	w_3\dfrac{\partial f_6}{\partial S_v\partial I_h}(0,0)\right)\\
	&=w_1\left(-\dfrac{a\beta_{vh}\eta_hS^{0}_v}{(N^{0}_h)^{2}}w_2
	-\dfrac{a\beta_{vh}S^{0}_v}{(N^{0}_h)^{2}}w_3\right) \\
	&+w_2\left(-\dfrac{a\beta_{vh}\eta_hS^{0}_v}{(N^{0}_h)^{2}}w_1
	-2\dfrac{a\beta_{vh}\eta_hS^{0}_v}{(N^{0}_h)^{2}}w_2
	-\dfrac{a\beta_{vh}S^{0}_v}{(N^{0}_h)^{2}}(\eta_h+1)w_3
	-\dfrac{a\beta_{vh}\eta_hS^{0}_v}{(N^{0}_h)^{2}}w_4
	+\dfrac{a\beta_{vh}\eta_h}{N^{0}_h}w_5\right)\\
	&+w_3\left(-\dfrac{a\beta_{vh}S^{0}_v}{(N^{0}_h)^{2}}w_1
	-\dfrac{a\beta_{vh}S^{0}_v}{(N^{0}_h)^{2}}(\eta_h+1)w_2
	-2\dfrac{a\beta_{vh}\eta_hS^{0}_v}{(N^{0}_h)^{2}}w_3
	-\dfrac{a\beta_{vh}S^{0}_v}{(N^{0}_h)^{2}}w_4
	+\dfrac{a\beta_{vh}}{N^{0}_h}w_5\right)\\
	&+w_4\left(-\dfrac{a\beta_{vh}\eta_hS^{0}_v}{(N^{0}_h)^{2}}w_2 
	-\dfrac{a\beta_{vh}S^{0}_v}{(N^{0}_h)^{2}}w_3\right)
	+w_5\left(\dfrac{a\beta_{vh}\eta_h}{N^{0}_h}w_2+ 
	\dfrac{a\beta_{vh}}{N^{0}_h}w_3\right)\\
	&=-\dfrac{a\beta_{vh}\eta_hS^{0}_v}{(N^{0}_h)^{2}}w_1w_2
	-\dfrac{a\beta_{vh}S^{0}_v}{(N^{0}_h)^{2}}w_1w_3\\
	&-\dfrac{a\beta_{vh}\eta_hS^{0}_v}{(N^{0}_h)^{2}}w_1w_2
	-2\dfrac{a\beta_{vh}\eta_hS^{0}_v}{(N^{0}_h)^{2}}w^{2}_2
	-\dfrac{a\beta_{vh}S^{0}_v}{(N^{0}_h)^{2}}(\eta_h+1)w_2w_3
	-\dfrac{a\beta_{vh}\eta_hS^{0}_v}{(N^{0}_h)^{2}}w_2w_4
	+\dfrac{a\beta_{vh}\eta_h}{N^{0}_h}w_2w_5\\
	&-\dfrac{a\beta_{vh}S^{0}_v}{(N^{0}_h)^{2}}w_1w_3
	-\dfrac{a\beta_{vh}S^{0}_v}{(N^{0}_h)^{2}}(\eta_h+1)w_2w_3
	-2\dfrac{a\beta_{vh}\eta_hS^{0}_v}{(N^{0}_h)^{2}}w^{2}_3
	-\dfrac{a\beta_{vh}S^{0}_v}{(N^{0}_h)^{2}}w_3w_4
	+\dfrac{a\beta_{vh}}{N^{0}_h}w_3w_5\\
	&-\dfrac{a\beta_{vh}\eta_hS^{0}_v}{(N^{0}_h)^{2}}w_2w_4
	-\dfrac{a\beta_{vh}S^{0}_v}{(N^{0}_h)^{2}}w_3w_4
	+\dfrac{a\beta_{vh}\eta_h}{N^{0}_h}w_2w_5
	+\dfrac{a\beta_{vh}}{N^{0}_h}w_3w_5\\
	&=-2\dfrac{a\beta_{vh}\eta_hS^{0}_v}{(N^{0}_h)^{2}}w_1w_2
	-2\dfrac{a\beta_{vh}S^{0}_v}{(N^{0}_h)^{2}}w_1w_3
	-2\dfrac{a\beta_{vh}\eta_hS^{0}_v}{(N^{0}_h)^{2}}w^{2}_2
	-2\dfrac{a\beta_{vh}S^{0}_v}{(N^{0}_h)^{2}}(\eta_h+1)w_2w_3
	-2\dfrac{a\beta_{vh}\eta_hS^{0}_v}{(N^{0}_h)^{2}}w_2w_4\\
	&-2\dfrac{a\beta_{vh}\eta_hS^{0}_v}{(N^{0}_h)^{2}}w^{2}_3
	-2\dfrac{a\beta_{vh}S^{0}_v}{(N^{0}_h)^{2}}w_3w_4
	+2\dfrac{a\beta_{vh}\eta_h}{N^{0}_h}w_2w_5
	+2\dfrac{a\beta_{vh}}{N^{0}_h}w_3w_5\\
	&=-2\dfrac{a\beta_{vh}\eta_hS^{0}_v}{(N^{0}_h)^{2}}w^{2}_2
	-2\dfrac{a\beta_{vh}S^{0}_v}{(N^{0}_h)^{2}}(\eta_h+1)w_2w_3
	-2\dfrac{a\beta_{vh}\eta_hS^{0}_v}{(N^{0}_h)^{2}}w_2w_4
	-2\dfrac{a\beta_{vh}\eta_hS^{0}_v}{(N^{0}_h)^{2}}w^{2}_3
	-2\dfrac{a\beta_{vh}S^{0}_v}{(N^{0}_h)^{2}}w_3w_4\\
	&+2\dfrac{a\beta_{vh}}{N^{0}_h}\eta_hw_2w_5
	+2\dfrac{a\beta_{vh}}{N^{0}_h}w_3w_5
	-2\dfrac{a\beta_{vh}S^{0}_v}{(N^{0}_h)^{2}}\eta_hw_1w_2
	-2\dfrac{a\beta_{vh}S^{0}_v}{(N^{0}_h)^{2}}w_1w_3\\
	&=-2\dfrac{a\beta_{vh}\eta_hS^{0}_v}{(N^{0}_h)^{2}}w^{2}_2
	-2\dfrac{a\beta_{vh}S^{0}_v}{(N^{0}_h)^{2}}(\eta_h+1)w_2w_3
	-2\dfrac{a\beta_{vh}\eta_hS^{0}_v}{(N^{0}_h)^{2}}w_2w_4
	-2\dfrac{a\beta_{vh}\eta_hS^{0}_v}{(N^{0}_h)^{2}}w^{2}_3
	-2\dfrac{a\beta_{vh}S^{0}_v}{(N^{0}_h)^{2}}w_3w_4\\
	&+2\dfrac{a\beta_{vh}}{N^{0}_h}\left( \eta_hw_2+w_3\right)w_5 
	-2\dfrac{a\beta_{vh}S^{0}_v}{(N^{0}_h)^{2}}\eta_hw_1w_2
	-2\dfrac{a\beta_{vh}S^{0}_v}{(N^{0}_h)^{2}}w_1w_3\\
	&=-2\dfrac{a\beta_{vh}\eta_hS^{0}_v}{(N^{0}_h)^{2}}w^{2}_2
	-2\dfrac{a\beta_{vh}S^{0}_v}{(N^{0}_h)^{2}}(\eta_h+1)w_2w_3
	-2\dfrac{a\beta_{vh}\eta_hS^{0}_v}{(N^{0}_h)^{2}}w_2w_4
	-2\dfrac{a\beta_{vh}\eta_hS^{0}_v}{(N^{0}_h)^{2}}w^{2}_3
	-2\dfrac{a\beta_{vh}S^{0}_v}{(N^{0}_h)^{2}}w_3w_4\\
	&+2\dfrac{a\beta_{vh}}{N^{0}_h}\left( \eta_hw_2+w_3\right)
	\left(-\dfrac{k_8+\gamma_v}{\gamma_v}w_7+\dfrac{k_7K_2K_4}{lK_1K_3}w_{10}\right) 
	-2\dfrac{a\beta_{vh}S^{0}_v}{(N^{0}_h)^{2}}\eta_hw_1w_2
	-2\dfrac{a\beta_{vh}S^{0}_v}{(N^{0}_h)^{2}}w_1w_3\\
	&=-2\dfrac{a\beta_{vh}S^{0}_v}{(N^{0}_h)^{2}}\left( \eta_hw^{2}_2
	+(\eta_h+1)w_2w_3+\eta_hw_2w_4+\eta_hw^{2}_3+w_3w_4\right)
	-2\dfrac{a\beta_{vh}(k_8+\gamma_v)}{\gamma_vN^{0}_h}\left( \eta_hw_2+w_3\right)w_7 \\
	&+2\dfrac{k_7K_2K_4}{lK_1K_3}\dfrac{a\beta_{vh}}{N^{0}_h}\left( \eta_hw_2+w_3\right)w_{10} 
	-2\dfrac{a\beta_{vh}S^{0}_v}{(N^{0}_h)^{2}}\left( \eta_hw_2+w_3\right)w_1 \\
	\end{split}
	\]
}
It follows then, after some algebraic computations, that	

{\footnotesize
	\begin{equation}
	\begin{split}
	\mathcal{A}_1&=v_2\left\lbrace -2\dfrac{a\beta^{*}_{hv}}{N^{0}_h}\left(\eta_vw_6+w_7\right)(w_2+w_3+w_4)\right\rbrace  \\
	&+v_6\left\lbrace -2\dfrac{a\beta_{vh}S^{0}_v}{(N^{0}_h)^{2}}\left( \eta_hw^{2}_2
	+(\eta_h+1)w_2w_3+\eta_hw_2w_4+\eta_hw^{2}_3+w_3w_4\right)
	-2\dfrac{a\beta_{vh}(k_8+\gamma_v)}{\gamma_vN^{0}_h}\left( \eta_hw_2+w_3\right)w_7\right.\\ 
	&\left.\qquad\qquad+2\dfrac{k_7K_2K_4}{lK_1K_3}\dfrac{a\beta_{vh}}{N^{0}_h}\left( \eta_hw_2+w_3\right)w_{10} 
	-2\dfrac{a\beta_{vh}S^{0}_v}{(N^{0}_h)^{2}}\left( \eta_hw_2+w_3\right)w_1\right\rbrace\\
	&=v_2\left\lbrace -2\dfrac{a\beta^{*}_{hv}}{N^{0}_h}\left(\eta_vw_6+w_7\right)(w_2+w_3+w_4)\right\rbrace  \\
	&+v_6\left\lbrace -2\dfrac{a\beta_{vh}S^{0}_v}{(N^{0}_h)^{2}}\left( \eta_hw^{2}_2
	+(\eta_h+1)w_2w_3+\eta_hw_2w_4+\eta_hw^{2}_3+w_3w_4\right)
	-2\dfrac{a\beta_{vh}(k_8+\gamma_v)}{\gamma_vN^{0}_h}\left( \eta_hw_2+w_3\right)w_7\right\rbrace \\ 
	&+v_6\left\lbrace 2\dfrac{k_7K_2K_4}{lK_1K_3}\dfrac{a\beta_{vh}}{N^{0}_h}\left( \eta_hw_2+w_3\right)w_{10} 
	-2\dfrac{a\beta_{vh}S^{0}_v}{(N^{0}_h)^{2}}\left( \eta_hw_2+w_3\right)w_1\right\rbrace \\
	&=v_6\left\lbrace 2\dfrac{k_7K_2K_4}{lK_1K_3}\dfrac{a\beta_{vh}}{N^{0}_h}\left( \eta_hw_2+w_3\right)w_{10} 
	-2\dfrac{a\beta_{vh}S^{0}_v}{(N^{0}_h)^{2}}\left( \eta_hw_2+w_3\right)w_1\right\rbrace \\
	&+v_2\left\lbrace -2\dfrac{a\beta^{*}_{hv}}{N^{0}_h}\left(\eta_vw_6+w_7\right)(w_2+w_3+w_4)\right\rbrace  \\
	&+v_6\left\lbrace -2\dfrac{a\beta_{vh}S^{0}_v}{(N^{0}_h)^{2}}\left( \eta_hw^{2}_2
	+(\eta_h+1)w_2w_3+\eta_hw_2w_4+\eta_hw^{2}_3+w_3w_4\right)
	-2\dfrac{a\beta_{vh}(k_8+\gamma_v)}{\gamma_vN^{0}_h}\left( \eta_hw_2+w_3\right)w_7\right\rbrace\\
	&=v_6\left\lbrace 2\dfrac{k_7K_2K_4}{lK_1K_3}\dfrac{a\beta_{vh}}{N^{0}_h}\left( \eta_hw_2+w_3\right)w_{10} 
	-2\dfrac{a\beta_{vh}S^{0}_v}{(N^{0}_h)^{2}}\left( \eta_hw_2+w_3\right)w_1\right\rbrace \\
	&-v_2\left\lbrace 2\dfrac{a\beta^{*}_{hv}}{N^{0}_h}\left(\eta_vw_6+w_7\right)(w_2+w_3+w_4)\right\rbrace  \\
	&-v_6\left\lbrace 2\dfrac{a\beta_{vh}S^{0}_v}{(N^{0}_h)^{2}}\left( \eta_hw^{2}_2
	+(\eta_h+1)w_2w_3+\eta_hw_2w_4+\eta_hw^{2}_3+w_3w_4\right)
	+2\dfrac{a\beta_{vh}(k_8+\gamma_v)}{\gamma_vN^{0}_h}\left( \eta_hw_2+w_3\right)w_7\right\rbrace\\
	&=\zeta_1-\zeta_2, \\ 
	\end{split}
	\end{equation}
}	
where we have set
\begin{equation}
\label{}
\begin{array}{l}
\zeta_1=v_6\left\lbrace 2\dfrac{k_7K_2K_4}{lK_1K_3}\dfrac{a\beta_{vh}}{N^{0}_h}\left( \eta_hw_2+w_3\right)w_{10} 
-2\dfrac{a\beta_{vh}S^{0}_v}{(N^{0}_h)^{2}}\left( \eta_hw_2+w_3\right)w_1\right\rbrace,\\
\begin{split}
\zeta_2&=v_2\left\lbrace 2\dfrac{a\beta^{*}_{hv}}{N^{0}_h}\left(\eta_vw_6+w_7\right)(w_2+w_3+w_4)\right\rbrace  \\
&+v_6\left\lbrace 2\dfrac{a\beta_{vh}S^{0}_v}{(N^{0}_h)^{2}}\left( \eta_hw^{2}_2
+(\eta_h+1)w_2w_3+\eta_hw_2w_4+\eta_hw^{2}_3+w_3w_4\right)\right.\\
&\left.\qquad\qquad 2\dfrac{a\beta_{vh}(k_8+\gamma_v)}{\gamma_vN^{0}_h}\left( \eta_hw_2+w_3\right)w_7\right\rbrace.\\
\end{split}
\end{array}
\end{equation}
According to \eqref{LeftVectorBasicModel} and \eqref{RightVectorBasicModel}, we have $\zeta_1>0$ and $\zeta_2>0$.

We then have
\[
\begin{split}
\mathcal{A}_2&=v_2\sum\limits_{i=1}^{10}w_i\dfrac{\partial^{2}f_2(0,0)}{\partial x_i\partial \phi}\\
&=v_2\left(w_6\dfrac{\partial^{2}f_2}{\partial E_v\partial \phi}(0,0)
+w_7\dfrac{\partial^{2}f_2}{\partial E_v\partial \phi}(0,0)\right) \\
&=v_2\left(\eta_hw_6\dfrac{aS^{0}_v}{N^{0}_h}
+w_7\dfrac{aS^{0}_v}{N^{0}_h}\right) \\
&=\dfrac{aS^{0}_v}{N^{0}_h}\left(\eta_hw_6+w_7\right)v_2.
\end{split}
\]

\section{Proof of Theorem~\ref{sdlenodeBasic}}
\label{sadleAr3opc}
We follow the approach given in~\cite{saetal}. At this aim, note that equation~\eqref{eqEEsansBasic} may be written as
\begin{equation}
\label{cvsn23Basic}
F(\beta_{hv},\lambda_{h}):=d_2(\lambda^{*}_h)^{2}+d_1\lambda^{*}_h+d_0=0,
\end{equation}
where $d_2$, $d_1$ and $d_0$ are the same coefficients as in \eqref{eqEEsansBasic}.
Thus, the positive endemic equilibria of model~\eqref{BasicModel} are obtained by solving (\ref{cvsn23Basic}) for positive $\lambda^{*}_{h}$ and substituting the results into \eqref{EEhBasic}. Clearly, the coefficient $d_2$, of \eqref{cvsn23Basic}, is always negative while $d_1$ and $d_0$ may change sign. Therefore, there is a single endemic equilibrium if and only if $d_0>0$, which correspond to $\mathcal R_0>1$. There are two endemic  equilibria if and only if $d_0<0$, $d_1>0$ and $d^{2}_1-4d_2d_0>0$.\\
Now, first remember that $d_0<0$ (i. e. $\mathcal R_0<1$) is equivalent to $\beta_{hv}<\beta^{*}_{hv}$, where $\beta^{*}_{hv}$ is given at Eq.~\eqref{bifparam}.

Then, inequality  $d_1>0$, is equivalent to
\begin{equation}
\label{betacrit1Basic}
\beta_{hv}>\bar{\beta}
\end{equation}
where $\bar{\beta}$ is given by \eqref{betacrit1Basic1}.

Finally, equation  $d^{2}_1-4d_2d_0=0$, in terms of $\beta_{hv}$, is equivalent to
\begin{equation}
\label{f3Basic}
\alpha_2\beta^{2}_{hv}+\alpha_1\beta_{hv}+\alpha_0=0,
\end{equation}
where $\alpha_2=k_{3}^2k_{4}^2k_{10}^2k_{11}^2a^4\mu_{h}^2(N^{0}_{v})^2\beta_{vh}^2$,
$
\alpha_0=k_{3}^2\,k_{4}^2\,k_{10}^2\,a^2\,\mu_{h}^4\,(N^{0}_{h})^2\beta_{vh}^2,
$
and\\
$
\alpha_1=2k_{3}k_{4}k_{10}k_{11}a^2\mu_{h}^2N^{0}_{h}N^{0}_{v}\beta_{vh}
\left[2k_{2}(k_{10}a\mu_{h}\beta_{vh}-k_{8}\delta\gamma_h)-k_{3}k_{4}k_{10}a\mu_{h}\beta_{vh}
\right].
$\\

Now we compute the discriminant $\Delta:=\alpha^{2}_{1}-4\alpha_2\alpha_0$, to obtain:
\[
\begin{split}
\Delta=-16k_{3}^2k_{4}^2k_{2}
k_{10}^2k_{11}^2\delta\gamma_ha^4\mu_{h}^4(N^{0}_{h})^2(N^{0}_{v})^2\beta_{vh}^2
\left(k_{10}a\mu_{h}\beta_{vh}+k_{8}k_{2}\right)\left(k
_{10}a\mu_{h}\beta_{vh}-\delta\gamma_hk_{8}\right).
\end{split}
\]
Equation \eqref{f3Basic} admits a real solution if and only if $\Delta\geq 0$. This condition is equivalent to
\begin{equation}
\label{f5Basic}
\psi:=k_{10}a\mu_{h}\beta_{vh}-\delta\gamma_hk_{8}\leq 0
\end{equation}
Under condition \eqref{f5Basic}, we conclude that $\alpha_1<0$.  Thus, Eq.~\eqref{f3Basic} admits exactly two positive solutions which are given by
\[
\begin{split}
\beta_{\pm}&=\dfrac{-\alpha_1\pm\sqrt{\Delta}}{2\alpha_2}\\
&=\dfrac{N^{0}_{h}}{k_{3}k_{4}k_{10}k_{11}a^2N^{0}_{v}\beta_{vh}}
\left\lbrace \sqrt{\delta\gamma_h(a\mu_{h}\beta_{vh}k_{10}+k_{2}k_{8})}
\pm\sqrt{(-k_{2}\psi)}\right\rbrace^{2} 
\end{split}
\]
Thus, condition $d^{2}_1-4d_2d_0>0$ written in the terms of $\beta_{hv}$ is equivalent to
\[
\beta_{hv}<\beta_{-}\quad or\quad \beta_{hv}>\beta_{+}.
\]
and the inequalities \eqref{Betasn2Basic} then follow.

\begin{remark}
	Note that condition \eqref{Betasn2Basic} is equivalent to condition \eqref{conditionOf2EEar3opc}. Therefore, Theorem~\ref{sdlenodeBasic} is equivalent to Theorem~\ref{EEBasic}.
\end{remark}

\section{Proof of Theorem~\ref{BasicdeltaNULL}}
\label{appBasic4}

Considering the model \eqref{BasicModel} without disease--induced death in human, and applying the same procedure as appendix~\ref{EEBasic}, we obtain that the non-zero equilibria of the basic model~\eqref{BasicModel} satisfy the linear	equation
\[
(s\Gamma_E + k_6\Gamma_L )(p_1\lambda^{*}_h+p_0)=0, 
\]
where $p_1=\mu_{b}\Lambda_{h}k_{9}\left(k_{2}a\mu_{h}\beta_{vh}
+k_{3}k_{8}(\mu_h+\sigma)\right) $ and 	$p_0=-\mu_{h}k_{3}k_{4}k_{8}k_{9}\mu_{b}\Lambda_{h}
\left(\mathcal R^{2}_{0,\delta=0}-1\right)$.

Clearly, $p_1>0$ and $p_0\geq 0$ whenever $R_{0,\delta=0}\leq 1$, so that $\lambda^{*}_h=-\dfrac{p_0}{p_1}\leq 0$. Therefore, the model~\eqref{BasicModel} without disease--induced death in human, has no endemic equilibrium whenever $\mathcal R_{0,\delta=0}\leq 1$.
The above result suggests the impossibility of backward bifurcation in the model \eqref{BasicModel} without disease--induced death, since no endemic equilibrium exists when $\mathcal R_{0,\delta=0}<1$ (and backward bifurcation requires the presence of at least two endemic equilibria when $\mathcal R_{0,\delta=0}<1$)~\cite{gaguab,SharomietAl2007}. 

To completely rule out backward bifurcation in model~\eqref{BasicModel}, we use the direct Lyapunov method to prove the global stability of the DFE.
Define the positively-invariant and attracting region
\[
\begin{split}
\mathcal{D}_{1}&=\left\lbrace
(S_h,E_h,I_h,R_h,S_v,E_v,I_v,E,L,P)\in\mathcal{D}:
S_h\leq N^{0}_h; S_v\leq N^{0}_v\right\rbrace.\\
\end{split}
\]
Consider the Lyapunov function
\[
\mathcal{G}=q_1E_h+q_2I_h+q_3E_v+q_4I_v,
\]
where
\[
\begin{array}{l}
q_1=\dfrac{1}{k_3};\,\,
q_3=\dfrac{\tau_1S^{0}_h}{k_3k_8}\dfrac{k_{11}}{k_9},\;\;
q_2=\dfrac{\tau_1S^{0}_h}{k_3k_8}\dfrac{k_{11}\zeta_2S^{0}_v}{k_4k_9},\;\;
q_4=\dfrac{\tau_1S^{0}_h}{k_3k_8}.
\end{array}
\]
and we have set $\tau_1=\dfrac{\mu_h\beta_{hv}}{\Lambda_h}$ and $\tau_2=\dfrac{\mu_h\beta_{vh}}{\Lambda_h}$.
The derivative of $\mathcal{G}$ is given by 
{\footnotesize
	\[
	\begin{split}
	\dot{\mathcal{G}}&=q_1\dot{E_h}+q_2\dot{I_h}+q_3\dot{E_v}+q_4\dot{I_v}\\
	&=q_1(\lambda_hS_h-k_3E_h)+q_2(\gamma_hE_h-k_4I_h)+q_3(\lambda_vS_v-k_9E_v)+q_4(\gamma_vE_v-k_8I_v)\\
	&=q_1\tau_1S_h(\eta_vE_v+I_v)-q_3k_9E_v+q_4\gamma_vE_v-q_4k_8I_v
	+q_3\tau_2S_v(\eta_hE_h+I_h)-q_1k_3E_h+q_2\gamma_hE_h-q_2k_4I_h\\
	&=(q_1\tau_1S_h\eta_v+q_4\gamma_v-q_3k_9)E_v+(q_1\tau_1S_h-q_4k_8)I_v
	+(q_3\tau_2S_v\eta_h+q_2\gamma_h-q_1k_3)E_h+(q_3\tau_2S_v-q_2k_4)I_h\\
	&\leq (q_1\tau_1S^{0}_h\eta_v+q_4\gamma_v-q_3k_9)E_v+(q_1\tau_1S^{0}_h-q_4k_8)I_v
	+(q_3\tau_2S^{0}_v\eta_h+q_2\gamma_h-q_1k_3)E_h+(q_3\zeta_2S^{0}_v-q_2k_4)I_h,\\
	&\,\, \text{since}\,\,
	S_h\leq S^{0}_h,\,\,S_v\leq S^{0}_v\,\,\, in\,\,\,\mathcal{D}_1.\\
	\end{split}
	\]
}
Replacing $q_i$, $i=1,\hdots,4$, by their value gives after straightforward simplifications
\[
\begin{split}
\dot{\mathcal{G}}
&\leq \left(\mathcal R^{2}_{0,\delta=0}-1\right) E_h\\
\end{split}
\]\normalsize
We have $\dot{\mathcal{G}}\leq 0$ if $\mathcal R_{0,\delta=0}\leq 1$, with $\dot{\mathcal{G}}=0$ if $\mathcal{R}_{0,\delta=0}=1$ or $E_h=0$. Whenever $E_h=0$, we also have $I_h=0$, $E_v=0$ and $I_v=0$. Substituting  $E_h=I_h=E_v=I_v=0$ in the first, fourth and fifth equation of Eq.~\eqref{BasicModel} with $\delta=0$ gives $S_h(t)\rightarrow S^{0}_h=N^{0}_h$, $R_h(t)\rightarrow 0$, and $S_v(t)\rightarrow S^{0}_v=N^{0}_v$ as
$t\rightarrow \infty$. Thus 
\small
\[
\begin{split}
&\left[S_h(t),E_h(t),I_h(t),R_h(t),S_v(t),E_v(t),I_v(t),E(t),L(t),P(t)\right]
\rightarrow (N^{0}_h,0,0,0,N^{0}_v,0,0,E,L,P)\\
&\text{as} \,\,\, t\rightarrow \infty.
\end{split}
\]\normalsize
It follows from the LaSalle's invariance principle \cite{invariance,14,13}
that every solution of \eqref{BasicModel} (when $\mathcal{R}_{0,\delta=0}\leq 1$), with
initial conditions in $\mathcal{D}_1$ converges to $\mathcal{E}_{1}$, as
$t\rightarrow \infty$. Hence, the DFE, $\mathcal{E}_{1}$, of model \eqref{BasicModel} without disease--induced death, is GAS in $\mathcal{D}_1$  if $\mathcal{R}_{0,\delta=0}\leq 1$.

\bibliographystyle{siam}
\bibliography{mybibfileOPC}

\begin{thebibliography}{10}

\bibitem{AbbouEtAl2015}
{\sc H.~Abboubakar, J.~C. Kamgang, L.~N. Nkamba, D.~Tieudjo, and L.~Emini},
  {\em Modeling the dynamics of arboviral diseases with vaccination
  perspective}, Biomath, 4 (2015), pp.~1--30.

\bibitem{AbbouEtAl2016}
{\sc H.~Abboubakar, J.~C. Kamgang, and D.~Tieudjo}, {\em Backward bifurcation
  and control in transmission dynamics of arboviral diseases},
  Submitted,~\url{https://hal.archives-ouvertes.fr/hal-01200471v3},  (2015),
  pp.~1--43.

\bibitem{Adams2004}
{\sc B.~M. Adams, H.~T. Banks, H.~Kwon, and H.~T. Tran}, {\em Dynamic multidrug
  therapies for hiv: optimal and sti control approaches}, Math. Biosci. Eng. 1,
  2 (2004), pp.~223--241.

\bibitem{Aldila2013}
{\sc D.~Aldila, T.~G{\"o}tz, and E.~Soewono}, {\em An optimal control problem
  arising from a dengue disease transmission model}, Mathematical Biosciences,
  242 (2013), pp.~9--16.

\bibitem{BrasseurThsesis}
{\sc {Anthony Brasseur}}, {\em Analyse des pratiques actuelles destin\'ees \'a
  limiter la propagation d'Aedes albopictus dans la zone sud est de la France
  et propositions d'am\'elioration}, PhD thesis, \'Ecoles des Hautes \'Etudes
  en Sant\'e Publique (EHESP), 2011.

\bibitem{Antonio2001}
{\sc M.~Antonio and T.~Yoneyama}, {\em Optimal and sub-optimal control in
  dengue epidemics}, Optim. Control Appl. Methods, 63 (2001), pp.~63--73.

\bibitem{BlaynehaetAl}
{\sc K.~W. Blayneha, A.~B. Gumel, S.~Lenhart, and T.~Clayton}, {\em Backward
  bifurcation and optimal control in transmission dynamics of west nile virus},
  Bulletin of Mathematical Biology, 72 (2010), pp.~1006--1028.

\bibitem{Bosc}
{\sc P.~Bosc, V.~Boullet, M.~Echaubard, M.~L. Corre, S.~Quilici, J.~P. Quod,
  J.~Rochat, S.~Ribes, M.~Salamolard, and E.~Thybaud}, {\em Premier Bilan sur
  les Impacts des Traitements Anti-moustiques, dans le cadre de la lutte Contre
  le Chikungunya, sur les Esp\`eces et les Milieux de l'\^ile de la R\'eunion},
  Juin 2006.

\bibitem{bbMBE2011}
{\sc B.~Buonomo}, {\em A simple analysis of vaccination strategies for
  rubella}, Mathematical Biosciences and Engineering, 8 (2011), pp.~677--687.

\bibitem{caga}
{\sc J.~R. Cannon and D.~J. Galiffa}, {\em An epidemiology model suggested by
  yellow fever}, Math. Methods Appl. Sci., 35 (2012), pp.~196--206.

\bibitem{Alcaraz2007}
{\sc {Carles Alcaraz} and {Emili Garc\'ia-Berthou}}, {\em Life history
  variation of invasive mosquitofish (gambusia holbrooki) along a salinity
  gradient}, Biological Conservation, 139 (2007), pp.~83--92.

\bibitem{Carr}
{\sc J.~Carr}, {\em Applications of Centre Manifold Theory}, Springer, New
  York, 1981.

\bibitem{Carvalho2015}
{\sc S.~A. Carvalho, S.~O. {da Silva}, and I.~{da Cunha Charret}}, {\em
  Mathematical modeling of dengue epidemic: Control methods and vaccination
  strategies}, arXiv:1508.00961v1,  (2015), pp.~1--11.

\bibitem{ccso}
{\sc C.~Castillo-Chavez and B.~Song}, {\em Dynamical models of tuberculosis and
  their applications}, Math. Biosci. Eng., 1 (2004), pp.~361--404.

\bibitem{ch}
{\sc A.~Chippaux}, {\em G\'en\'eralit\'es sur arbovirus et
  arboviroses--overview of arbovirus and arbovirosis}, Med. Maladies Infect.,
  33 (2003), pp.~377--384.

\bibitem{coutinho2006}
{\sc F.~A.~B. Coutinho, M.~N. Burattini, L.~F. Lopez, and E.~Massad}, {\em
  Threshold conditions for a non-autonomous epidemic system describing the
  population dynamics of dengue}, Bulletin of Mathematical Biology, 68 (2006),
  pp.~2263--2282.

\bibitem{cretal2005}
{\sc G.~Cruz-Pacheco, L.~Esteva, J.~A. Monta$\tilde{n}$o-Hirose, and
  C.~Vargas}, {\em Modelling the dynamics of west nile virus}, Bulletin of
  Mathematical Biology, 67 (2005), pp.~1157--1172.

\bibitem{cretal}
{\sc G.~Cruz-Pacheco, L.~Esteva, and C.~Vargas}, {\em Seasonality and outbreaks
  in west nile virus infection}, Bull. Math. Biol., 71 (2009), pp.~1378--1393.

\bibitem{IRD}
{\sc F.~Darriet, S.~Marcombe, and V.~Corbel}, {\em Insecticides larvicides et
  adulticides alternatifs pour les op\'erations de d\'emoustication en france,
  synth\`ese bibliographique}, IRD,  (2007), pp.~1--46.

\bibitem{Derouich2006}
{\sc M.~Derouich and A.~Boutayeb}, {\em Dengue fever: mathematical modelling
  and computer simulation}, Applied Mathematics and Computation 177, 2 (2006),
  pp.~528--544.

\bibitem{dihe}
{\sc O.~Diekmann and J.~A.~P. Heesterbeek}, {\em Mathematical Epidemiology of
  Infectious Diseases. Model building, analysis and interpretation}, John Wiley
  \& Sons, Chichester, 2000.

\bibitem{moulayThesis}
{\sc {Djamila Moulay}}, {\em Mod\'elisation et analyse math\'ematique de
  syst\`emes dynamiques en \'epid\'emiologie. Application au cas du
  Chikungunya.}, PhD thesis, Universit\'e du Havre, 2011.

\bibitem{duch}
{\sc Y.~Dumont and F.~Chiroleu}, {\em Vector control for the chikungunya
  disease}, Math. Biosci. Eng., 7 (2010), pp.~313--345.

\bibitem{duhucc}
{\sc J.~Dushoff, W.~Huang, and C.~Castillo-Chavez}, {\em Backward bifurcations
  and catastrophe in simple models of fatal diseases}, J. Math. Biol., 36
  (1998), pp.~227--248.

\bibitem{esva98}
{\sc L.~Esteva and C.~Vargas}, {\em Analysis of a dengue disease transmission
  model}, Math. Biosci., 150 (1998), pp.~131--151.

\bibitem{esva99}
\leavevmode\vrule height 2pt depth -1.6pt width 23pt, {\em A model for dengue
  disease with variable human population}, J. Math. Biol., 38 (1999),
  pp.~220--240.

\bibitem{fevh}
{\sc Z.~Feng and V.~Velasco-Hernadez}, {\em Competitive exclusion in a
  vector--host model for the dengue fever}, J. Math. Biol., 35 (1997),
  pp.~523--544.

\bibitem{Fleming}
{\sc W.~H. Fleming and R.~W. Rishel}, {\em Deterministic and Stochastic Optimal
  Control}, Springer Verlag, 1975.

\bibitem{gaguab}
{\sc S.~M. Garba, A.~B. Gumel, and M.~R.~A. Bakar}, {\em Backward bifurcations
  in dengue transmission dynamics}, Math. Biosci., 215 (2008), pp.~11--25.

\bibitem{gu}
{\sc D.~J. Gubler}, {\em Human arbovirus infections worldwide}, Ann. N. Y.
  Acad. Sci., 951 (2001), pp.~13--24.

\bibitem{guho}
{\sc J.~Guckenheimer and P.~Holmes}, {\em Dynamical Systems and Bifurcations of
  Vector Fields}, Nonlinear Oscillations, 1983.

\bibitem{invariance}
{\sc J.~K. Hale}, {\em Ordinary Differential Equations}, John Wiley and Sons,
  1969.

\bibitem{HelenaSofia}
{\sc {Helena Sofia Ferreira Rodrigues}}, {\em Optimal Control and Numerical
  Optimization Applied to Epidemiological Models}, PhD thesis, Universidade de
  Aveiro Departamento de Matematica, 2012.

\bibitem{Jung2002}
{\sc E.~Jung, S.~Lenhart, and Z.~Feng}, {\em Optimal control of treatments in a
  two-strain tuberculosis model}, Discrete and Continuous Dynamical
  Systems--Series B, 2 (2002), pp.~473--482.

\bibitem{ka}
{\sc N.~Karabatsos}, {\em International Catalogue of Arboviruses, including
  certain other viruses of vertebrates}, San Antonio, TX. 1985, 2001 update.

\bibitem{14}
{\sc J.~P. LaSalle}, {\em Stability theory for ordinary differential
  equations}, J. Differ. Equ.,  (1968), pp.~57--65.

\bibitem{13}
\leavevmode\vrule height 2pt depth -1.6pt width 23pt, {\em The stability of
  dynamical systems}, Society for Industrial and Applied Mathematics,
  Philadelphia, Pa., 1976.

\bibitem{LeMonde2015}
{\sc {LE MONDE ECONOMIE}}, {\em Sanofi lance son vaccin contre la dengue},
  accessed 09/12/2015 at 16:30 • Updated 10/12/2015 at 11:26.

\bibitem{Lenhart}
{\sc S.~Lenhart and J.~T. Workman}, {\em Optimal Control Applied to Biological
  Models}, Chapman and Hall, 2007.

\bibitem{Lukes}
{\sc D.~L. Lukes}, {\em Differential equations : classical to controlled},
  Academic Press, New York, 1982.

\bibitem{maya}
{\sc N.~A. Maidana and H.~M. Yang}, {\em Dynamic of west nile virus
  transmission considering several coexisting avian populations}, Math. Comput.
  Modelling, 53 (2011), pp.~1247--1260.

\bibitem{Marino2008}
{\sc S.~Marino, I.~B. Hogue, C.~J. Ray, , and D.~E. Kirschner}, {\em A
  methodology for performing global uncertainty and sensitivity analysis in
  systems biology}, Journal of Theoretical Biology, 254 (2008), pp.~178--196.

\bibitem{moaaca}
{\sc D.~Moulay, M.~A. Aziz-Alaoui, and M.~Cadivel}, {\em The chikungunya
  disease: Modeling, vector and transmission global dynamics}, Math. Biosci.,
  229 (2011), pp.~50--63.

\bibitem{moaaHee2012}
{\sc D.~Moulay, M.~A. Aziz-Alaoui, and K.~Hee-Dae}, {\em Optimal control of
  chikungunya disease: larvae reduction,treatment and prevention}, Mathemtical
  Biosciences and Engineering, 9 (April 2012), pp.~369--393.

\bibitem{Okosun2011}
{\sc K.~O. Okosun, R.~Ouifki, and N.~Marcus}, {\em Optimal control analysis of
  a malaria disease transmission model that includes treatment and vaccination
  with waning immunity}, BioSystems, 106 (2011), pp.~136--145.

\bibitem{Parola}
{\sc P.~Parola, X.~{de Lamballerie}, J.~Jourdan, C.~Rovery, V.~Vaillant,
  P.~Minodier, P.~Brouqui, A.~Flahault, D.~Raoult, and R.~N. Charrel}, {\em
  Novel chikungunya virus variant in travelers returning from indian ocean
  islands}, Emerging Infectious Diseases, 12 (Octobre 2006), pp.~1--12.

\bibitem{Paupy2009}
{\sc C.~Paupy, H.~Delatte, L.~Bagny, V.~Corbel, and D.~Fontenille}, {\em Aedes
  albopictus, an arbovirus vector: from the darkness to the light}, Microbes
  Infect., 1 (2009), pp.~14--15.

\bibitem{poetal}
{\sc P.~Poletti, G.~Messeri, M.~Ajelli, R.~Vallorani, C.~Rizzo, and S.~Merler},
  {\em Transmission potential of chikungunya virus and control measures: the
  case of italy}, PLoS One, 6 (2011), pp.~1--12.

\bibitem{Pontryagin}
{\sc L.~S. Pontryagin, V.~G. Boltyanskii, R.~V. Gamkrelidze, and E.~F.
  Mishchenko}, {\em The mathematical theory of optimal processes}, Wiley, New
  York, 1962.

\bibitem{HelenaSofiaRodrigues2014}
{\sc H.~S. Rodrigues, M.~T.~T. Monteiro, and D.~F.~M. Torres}, {\em Vaccination
  models and optimal control strategies to dengue}, Mathematical Biosciences,
  247 (2014), pp.~1--12.

\bibitem{Arunee2012}
{\sc A.~Sabchareon, D.~Wallace, C.~Sirivichayakul, K.~Limkittikul,
  P.~Chanthavanich, S.~Suvannadabba, V.~Jiwariyavej, W.~Dulyachai, K.~Pengsaa,
  T.~{Anh Wartel}, A.~Moureau, M.~Saville, A.~Bouckenooghe, S.~Viviani, N.~G.
  Tornieporth, and J.~Lang}, {\em Protective efficacy of the recombinant,
  live-attenuated, cyd tetravalent dengue vaccine in thai schoolchildren: a
  randomised, controlled phase 2b trial}, Lancet, 380 (2012), pp.~1559--1567.

\bibitem{saetal}
{\sc M.~Safan, M.~Kretzschmar, and K.~P. Hadeler}, {\em Vaccination based
  control of infections in sirs models with reinfection: special reference to
  pertussis}, J. Math. Biol., 67 (2013), pp.~1083--1110.

\bibitem{Sanofi2013}
{\sc {SANOFI PASTEUR}}, {\em Dengue vaccine, a priority for global health},
  2013.

\bibitem{Sanofi2014}
\leavevmode\vrule height 2pt depth -1.6pt width 23pt, {\em Communiqu\'e de
  presse: The new england journal of medicine publie les r\'esultats de
  l'\'etude clinique d'efficacit\'e de phase iii du candidat vaccin dengue de
  sanofi pasteur}, 2014.

\bibitem{Scott2010}
{\sc T.~W. Scott and A.~C. Morrison}, {\em Vector dynamics and transmission of
  dengue virus: implications for dengue surveillance and prevention strategies:
  vector dynamics and dengue prevention}, Current Topics in Microbiology and
  Immunology, 338 (2010), pp.~115--128.

\bibitem{SharomietAl2007}
{\sc O.~Sharomi, C.~Podder, A.~Gumel, E.~Elbasha, and J.~Watmough}, {\em Role
  of incidence function in vaccine-induced backward bifurcation in some hiv
  models}, Mathematical Biosciences, 210 (2007), pp.~436--463.

\bibitem{Sota1992}
{\sc T.~Sota and M.~Mogi}, {\em Survival time and resistance to desiccation of
  diapause and non-diapause eggs of temperate aedes (stegomyia) mosquitoe},
  Entomologia Experimentalis et Applicata, 63 (1992).

\bibitem{Stein1987}
{\sc M.~Stein}, {\em Large sample properties of simulations using latin
  hypercube sampling}, Technometrics, 29 (1987), pp.~143--151.

\bibitem{vawa02}
{\sc P.~{van den Driessche} and J.~Watmough}, {\em Reproduction numbers and the
  sub-threshold endemic equilibria for compartmental models of disease
  transmission}, Math. Biosci., 180 (2002), pp.~29--48.

\bibitem{Villar2015}
{\sc L.~Villar, G.~H. Dayan, J.~L. {Arredondo-Garc\'ia}, D.~M. Rivera,
  R.~Cunha, C.~Deseda, H.~Reynales, M.~S. Costa, J.~O. {Morales-Ramírez},
  G.~Carrasquilla, L.~C. Rey, R.~Dietze, K.~Luz, E.~Rivas, M.~C.~M. Montoya,
  M.~C. Supelano, B.~Zambrano, E.~Langevin, M.~Boaz, N.~Tornieporth,
  M.~Saville, and F.~Noriega}, {\em Efficacy of a tetravalent dengue vaccine in
  children in latin america}, The New England Journal of Medicine, 372 (2015),
  pp.~113--123.

\bibitem{Wangari2015}
{\sc I.~M. Wangari, S.~Davis, and L.~Stone}, {\em Backward bifurcation in
  epidemic models: Technical note}, Applied Mathematical Modelling,  (2015),
  pp.~pp. 1--11.

\bibitem{Dias2015}
{\sc {Weverton O. Dias}, {Elizabeth F. Wanner}, and {Rodrigo T.N. Cardoso}},
  {\em A multiobjective optimization approach for combating aedes aegypti using
  chemical and biological alternated step-size control}, Mathematical
  Biosciences, 269 (2015), pp.~37--47.

\bibitem{WHO2004}
{\sc {Will Parks} and {Linda Lloyd}}, {\em Planning social mobilization and
  communication for dengue fever prevention and control}, WORLD HEALTH
  ORGANIZATION,  (2004), pp.~1--158.

\bibitem{WHO2009}
{\sc {World Health Organization}}, {\em Dengue and dengue haemorhagic fever}.
\newblock \url{www.who.int/mediacentre/factsheets/fs117/en}, 2009.

\bibitem{WHO2006}
\leavevmode\vrule height 2pt depth -1.6pt width 23pt, {\em Dengue and severe
  dengue}.
\newblock \url{www.who.int/mediacentre/factsheets/fs117/en}, Updated September
  2013.

\bibitem{Wu2013}
{\sc J.~Wu, R.~Dhingra, M.~Gambhir, and J.~V. Remais}, {\em Sensitivity
  analysis of infectious disease models: methods, advances and their
  application}, Journal of the Royal Society Interface, 10 (2013), pp.~1--14.

\bibitem{YangFerreira2008}
{\sc H.~M. Yang and C.~P. Ferreira}, {\em Assessing the effects of vector
  control on dengue transmission}, Applied Mathematics and Computation, 198
  (2008), pp.~401--413.

\bibitem{Yusuf2012}
{\sc T.~T. Yusuf and F.~Benyah}, {\em Optimal control of vaccination and
  treatment for an sir epidemiological model}, World Journal of Modelling and
  Simulation, 8 (2012), pp.~194--204.

\bibitem{Zaman2008}
{\sc G.~Zaman, Y.~H. Kang, and I.~H. Jung}, {\em Stability analysis and optimal
  vaccination of an sir epidemic model}, BioSystems, 93 (2008), pp.~240--249.

\end{thebibliography}
\end{document}